\documentclass[a4paper]{article}

\usepackage[OT1]{fontenc}

\usepackage[english]{babel}

\usepackage[utf8]{inputenc}

\usepackage[sc]{mathpazo}

\usepackage{amsmath,amssymb,amsfonts}%,mathrsfs}
\usepackage{supertabular}
\usepackage{amscd}
\usepackage{textcomp}
\usepackage{authblk}
\usepackage{accents}

\usepackage{latexsym}
\usepackage{mathtools}
\usepackage{mathdots}
\usepackage[only,llbracket,rrbracket]{stmaryrd}
\usepackage[amsmath,thmmarks]{ntheorem}

\usepackage{graphicx}

\usepackage{soul}

\usepackage{pdfpages}

\usepackage{longtable}

\usepackage{varioref}

\usepackage{datetime}

\usepackage[h]{esvect}

\usepackage{array}

\usepackage{booktabs}

\usepackage[mathscr]{euscript}

\usepackage{mathtools}

\usepackage{tikz-cd}

\numberwithin{equation}{section}

\newtheorem{theorem}{Theorem}[section]
\newtheorem{ex}[theorem]{Example}
\newtheorem{re}[theorem]{Remark}
\newtheorem{co}[theorem]{Corollary}
\newtheorem{defi}[theorem]{Definition}
\newtheorem{lem}[theorem]{Lemma}
\newtheorem{pr}[theorem]{Proposition}

\theoremstyle{nonumberplain}
\theorembodyfont{\normalfont}
\theoremsymbol{\ensuremath{\square}}
\newtheorem{proof}{Proof}
\def\o{\circ}
\def\x{\times}

\def\opt{\operatorname{opt}}

\def\id{\operatorname{id}}

\def\Lip{{\mathcal{L}ip}}
\def\loc{\text{loc}}

\def\Ez{\mathbb E}

\def\Nz{\mathbb N}

\def\Pz{\mathbb P}
\def\Qz{\mathbb Q}
\def\Rz{\mathbb R}

\def\Rz{\mathbb R}

\def\Zz{\mathbb Z}

\def\ep{\varepsilon}

\def\Bc{\mathcal B}
\def\Cc{\mathcal C}
\def\Dc{\mathcal D}
\def\Ec{\mathcal E}

\def\Gc{\mathcal G}
\def\Hc{\mathcal H}

\def\Kc{\mathcal K}
\def\Lc{\mathcal L}
\def\Mc{\mathcal M}
\def\Nc{\mathcal N}

\def\Pc{\mathcal P}

\def\Sc{\mathcal S}

\def\Uc{\mathcal U}

\def\Wc{\mathcal W}
\def\Xc{\mathcal X}

\def\Pf{\mathcal{}\mathfrak P}

\def\Rf{\mathfrak R}
\def\Sf{\mathfrak S}

\def\pf{\mathfrak p}

\def\sf{\mathfrak s}

\def\ol#1{\overline{#1}}
\def\0{{\mathbf 0}}
\def\1{{\mathbf 1}}

\def\d{\partial}
\newcommand{\der}{\operatorname{d}}
\def\define#1{\emph{#1}}

\def\T{{\operatorname{T}}}

\def\o{\circ}
\def\x{\times}

\def\supp{{\operatorname{supp}}}

\def\cov{{\operatorname{Cov}}}

\def\eps{\varepsilon}
\def\al{\alpha}
\def\be{\beta}
\def\ga{\gamma}
\def\de{\delta}

\def\la{\lambda}
\def\ta{\tau}
\def\rh{\rho}
\def\si{\sigma}
\def\ph{\varphi}

\def\ps{\psi}
\def\om{\omega}
\def\Ga{\Gamma}
\def\De{\Delta}

\def\Si{\Sigma}
\def\Ph{\Phi}

\def\Om{\Omega}

\def\ol#1{\overline{#1}}

\def\aequiv#1{\;\raise-4pt\hbox{$\substack{{\displaystyle\sim}\\\raise4pt\hbox{$\scriptstyle #1$}}$}
\;}

\DeclareMathAccent{\wtilde}{\mathord}{largesymbols}{"65}

\def\loc{{\operatorname{loc}}}

\makeatletter
\def\cn#1{\def\tmp@cn{#1}\ifx\tmp@cn\empty\nabla\else\nabla^{\tmp@cn}\fi}
\makeatother

\def\o{\circ}

\def\Xct{\Xc^{2,\infty}}

\usepackage[linkcolor=black,colorlinks=true,citecolor=black,filecolor=black]{hyperref}

\usepackage[numbers,sort&compress]{natbib}

\def\chapterauthor#1{}

\title{Non-linear degenerate parabolic flow equations and a finer differential structure on Wasserstein spaces}
\author{Arthur Schichl}
\affil{\small KOF Swiss Economic Institute\\ETH Zürich\\Leonhardstrasse 21\\8092 Zürich\\ Switzerland\\
aschichl@ethz.ch}
\date{August 21, 2025}

\begin{document}
\def\d{\partial}
\def\o{\circ}

%\frontmatter

%% Title page is autogenerated from document information above.  DO
%% NOT CHANGE.
\maketitle

%% The abstract of your thesis.  Edit the file as needed.
\chapterauthor{ Arthur Schichl \\%[-1.5mm]
        \small KOF Swiss Economic Institute \\ [-2.0mm]
        \small ETH Zürich \\ [-2.0mm]
        \small Leonhardstrasse 21 \\ [-2.0mm]
		\small 8092 Zurich, Switzerland \\ [-2.0mm]
        \small aschichl@ethz.ch } 

\begin{abstract}
	% In this work we discuss and compare various classical and new approaches of endowing the Wasserstein space
	% $\Wc_p(M)$ for $p\geq 1$ and a general Riemannian manifold $(M,g)$ with differential structures.
	
	We define new differential structures on the Wasserstein spaces $\Wc_p(M)$ for $p > 2$ and a general Riemannian manifold $(M,g)$. We consider a very general and possibly degenerate second
	order partial differential flow equation with measure dependent coefficients to expand the notion of smooth curves and to ensure that the new differential
	structure is finer than the classical one. Under weak assumptions, we explicitly construct smooth solutions as uniform limits of Average Flow Approximation Series (a variant of explicit Euler--scheme approximations) in $\Wc_p(M)$ and, 
	thus, prove a generalized version of the Central Limit Theorem. Under slightly stronger assumptions, we prove that smooth solutions of our newly introduced flow--equation are unique.
\end{abstract}

 \renewcommand{\abstractname}{Keywords}
 \begin{abstract}
 	Optimal transport, Wasserstein spaces, Infinite--dimensional differential structures, Measure-valued differential equations, Parabolic flow equations, Degenerate parabolic equations, Generalized Central Limit Theorem, Measure--dependent parabolic equations, Euler--scheme for parabolic equations.
 \end{abstract}

 \renewcommand{\abstractname}{Acknowledgements}
 \begin{abstract}
 	I am grateful to Alessio Figalli and Ulisse Stefanelli for their guidance and encouragement. I would also like to thank Hans Gersbach, Wilhelm Schlag, Darko Mitrovic, Nikolai V. Krylov and Hermann Schichl for their valuable input. 
 \end{abstract}

\newpage
%% TOC with the proper setup, do not change.
%\cleartorecto
\tableofcontents
%\mainmatter

%% Your real content!
\section{Introduction}\label{ch:intro}
\label{sec:intro1}

% Nature seems to be subject to a principle of causality---this observation is reflected in the fact that the dynamics of many systems in natural and experimental sciences can be satisfyingly modeled by differential equations. However, some of these systems naturally appear in a more general setting than Euclidean spaces. Describing the trajectories of light in non-Euclidean space-times, for example, requires a theory of differentiation on more general objects than vector spaces. In this context, topological manifolds were endowed with differential structures.

% Some of those systems on the other hand are characterized by an intrinsic uncertainty---or physical conditions are so complex that only probabilistic models prove to be viable. In that spirit, instead of studying classical differential equations of functions with image in some $\Rz^n$, it is often required to study differential equations of functions with image in $L^1(\Rz^n)$ or some other function space over $\Rz^n$, i.e., to describe the evolution of some probability density over time. A well-studied tool to model dynamics of random variables are the so-called stochastic differential equations (SDE), which are particularly useful to model dynamics subject to small perturbations or random noise. \cite{hairer2009introduction}

Recent years have seen an increased need for analytical calculus on spaces other than finite dimensional vector spaces. In particular, measure spaces like the Wasserstein spaces of optimal transport theory have been of central interest in prominent works like \citet{ambrosio2005gradient,ambrosio2008calculus,figalli2021invitation,santambrogio2015optimal,villani2009optimal,villani2021topics}.  

% In the same spirit, more recent research has even been taking the topic a step further and has directly studied differential equations of functions with images in general measure spaces.

Calculus on measure spaces first requires defining differential structures on those spaces. 
% Clearly, a discussion of differential
% equations first requires to analyse, which meaningful differential structures one can define on those spaces. 
Hence, this whole work will be devoted to defining what it means for a path of measures $[0,T] \to \Pc(M)$, $t \mapsto \mu(t)$ to satisfy a flow equation of the form
\begin{equation}\label{eq:MDE}
\tag{MDE}
\begin{aligned}
		\dot{\mu}(t) &= V(\mu(t)) \\
		\mu(0) &= \mu_0,
\end{aligned}
\end{equation}
on the Wasserstein space $\Wc_p(M)$, for a Riemannian manifold $(M,g)$, $p\geq 1$ and some probability vector--field $V$, which remains to be defined. 
To these means, \citet{otto2000generalization}, \citet{otto2001geometry} were the first to endow the Wasserstein spaces $\Wc_p(\Rz^n)$ for $p>1$ with differential structures. The so-called continuity equation
\begin{equation}
	\frac{\der \mu}{\der t} + div^g(v_t \mu) = 0,
\end{equation}
where $t \mapsto v_t$ is a curve of $p$--integrable vector--fields, is at the core of the theory. This equation is of high interest, since any absolutely continuous curve of measures $\mu(t)$ in $\Wc_p(\Rz^n)$ satisfies the continuity equation for some curve of vector fields $t\mapsto v_t$ \cite{ambrosio2005gradient}. Hence, it is now well-established to define the tangent space at $\mu \in \Wc_p(\Rz^n)$ by $L^p(\Rz^n)$ and interpret solutions of ODEs of measures, i.e.\ measure differential equations (MDEs), as solutions of the flow equation
\begin{equation}\label{eq:flowconteq}
	\frac{\der \mu}{\der t} + div^g(v(\mu) \mu) = 0.
\end{equation}
Detailed expositions of the theory in $\Rz^n$ and more general spaces can be found in
\citet{ambrosio2005gradient,gigli2012second,bonnet2021differential,cavagnari2023dissipative}.
However, the so--defined differential structures have important limitations. 

Since most measure spaces are of infinite dimension, in general there exist many non-equivalent differential structures compatible with a given topology and there is not even a canonical topology
on these spaces. 
% For example, the space of signed Borel-measures $\Mc(X)$ on a metric space $X$ and its subspaces
% can be endowed with the strong operator topology, but also with the weak--$*$--topology, which does in general not arise from a norm, and finally the topology we will mostly consider in what follows, namely the Wasserstein topology. 
\citet{piccoli2019measure} gives a new definition of solutions of MDEs on $\Wc_q(M)$ for $q\geq1$---namely curves that satisfy an integral equation of the form
\begin{equation}\label{eq:piccMDE}
\begin{aligned}
\int_M \ph \der \mu(t) &- \int_M \ph \der \mu_0 \\&= \int_0^t \int_M \int_{T_p M} \nabla\ph(p,X) dV[\mu(s)]_p(p,X) d\mu(t)(p) \der s,
\end{aligned}
\end{equation}
where $V$ sends a measure $\mu \in \Wc_q(M)$ onto a measure $V[\mu] \in \Wc_q(TM)$ on the tangent bundle of $M$.

% This definition, which is still close to the continuity equation in spirit, defines differential structures on all Wasserstein spaces $\Wc_p(\Rz^n)$ for $p \geq 1$ with larger tangent spaces than previously. 

Although \citet{piccoli2019measure} defines larger tangent spaces than the classical ones, the differential structures induced by the flow equation (\ref{eq:piccMDE}) are essentially equivalent to those induced by the continuity equation for $p>1$. However, the main result in \citet{piccoli2019measure} proves the existence of a Lipschitz semigroup of solutions of the equation above, making only very weak regularity assumptions on the probability vector field $V$ and is, thus, of independent interest. The limitations in the differential structure remain.  
In fact, solutions of (\ref{eq:flowconteq}) have the so called \textit{barycentric property} \cite{cavagnari2023dissipative}, which says that sufficiently regular solutions \cite{ambrosio2008calculus} of (\ref{eq:MDE}) look like 
\begin{equation}
\mu(t) = (\rh^{\ol{V}[\mu(t)]}_t)_* \mu_0.
\end{equation}
The so--defined differential structure is not very rich. In particular, the barycentric property implies that the standard diffusion process is not a ``sufficiently strong'' solution of a measure valued differential equation with the definitions given in \cite{bonnet2021differential,cavagnari2023dissipative,piccoli2019measure}. This results in a lack of smooth geodesics between measures, for example between a Dirac measure and any measure whose support is larger than a single point. The structure of the tangent spaces on the other hand makes repeated differentiation impossible and in general, the derivative of functions between Wasserstein spaces is not well-defined.

% Our aim is to define a new richer differential structure on the Wasserstein space $\Wc_p(M)$ --- i.e. define what it means for a curve of measures to be smooth and corresponding tangent spaces---

% such that some standard diffusion processes are sufficiently regular.  
% As in previous chapters, we proceed by defining what it means for a curve $\mu(t)$ to be a solution of an equation of type (\ref{eq:MDE}). Diffusion processes should be among the possible solutions, so what should that definition satisfy?

% With a discrete reasoning, if $t\mapsto \mu(t)$ is a smooth curve, we intuitively want to get from $\mu(t)$ to $\mu(t + \der t)$ by some kind of, non necessarily centered, diffusion described by $V(\mu(t))$. In particular, if for $p\in M$, $V(\de_p)$ is concentrated on vector fields with $X_p = X_1$ or $X_p = X_2$, we expect $\de_p$ to ``split'' in the directions $X_1$ and $X_2$. 

% Just for this motivational subsection, let us furthermore suppose that for every $X\in\Xc(M)$, $\rho^X_t$ is well-defined for every $t\in\Rz$; this is for example the case if $M$ is compact or if we restrict ourselves to vector fields in $\Xc^{2,\infty}(M)$, like in Definition~\ref{de:tangspaced}. For our discrete reasoning, we define the following discrete push--forward operators.
%Hence, his results are of independent interest and we will generalise his main results to arbitrary Riemannian manifolds in Chapter~\ref{ch:mde}.

\smallskip

In Chapter~\ref{ch:parstructwas}, we will endow the Wasserstein spaces $\Wc_p(M)$, where $p > 2$ and $(M,g)$ is a Riemannian manifold, with a new richer differential structure, which overcomes some of the drawbacks of the classical differential structures on Wasserstein spaces,
while maintaining many of their advantages. The corresponding flow equation is essentially of the form
\begin{equation}\label{eq:MDEdintro}
\begin{aligned}
\int_M \ph \der \mu(s) - \int_M \ph \der \mu(0) = \int_0^s \int_M \square_{\mu(t)}^V \ph \der \mu(t) \der t,
\end{aligned}
\end{equation} 
for all $\ph$ smooth and compactly supported on $M$, where $\square_{\mu(t)}^V \ph$ is a possibly degenerate parabolic operator depending on $\mu \in \Wc_p(M)$, which we will introduce in Chapter~\ref{ch:parstructwas}.

\bigskip

% The structure of this master thesis is as follows. 

Before presenting the main theorems, this article contains a first chapter on some useful preliminary results in optimal transport and Riemannian geometry.

% In Chapter~\ref{ch:conteq}, we will summarize the main results of the solution theory of the continuity equation on a general Riemannian manifold $M$ and about the differential structures it induces on the $\Wc_p(M)$ for $p>1$.

% Chapter~\ref{ch:mde} is devoted to the generalisation of the main results of \citet{piccoli2019measure} to Riemannian manifolds and to the discussion of the differentiable structure on $\Wc_p(M)$ for $p\geq 1$. We will see that for $p>1$ the such defined differential structures coincide with the classical ones described in
% Chapter~\ref{ch:conteq}.

Chapter~\ref{ch:parstructwas} contains the main results of this article. First, we use (\ref{eq:MDEdintro}) to define a new differential structure on the Wasserstein spaces $\Wc_p(M)$ for $p >2 $. We shall characterize the solutions of (\ref{eq:MDEdintro}) and derive some important properties they possess. Then, we will introduce the concept of smooth solutions of (\ref{eq:MDEdintro}), which are essentially the smooth curves on $\Wc_p(M)$ for our new differential structure. Our main theorems are that smooth solutions of (\ref{eq:MDEdintro}) always exist under rather weak assumptions on $V$, and they are unique on manifolds $(M,g)$ with bounded geometry with slightly stronger
assumptions on $V$. Furthermore, we show that those unique smooth solutions of (\ref{eq:MDEdintro}) are Hölder--continuous and obtained as uniform limits of an explicit Euler--like--scheme, which we call Average Flow Approximation Series (AFAS) in $\Wc_p(M)$. Finally, we show that non--smooth solutions of (\ref{eq:MDEdintro}) are not unique, even with strong assumptions on $V$.

% we define a new differential structure on the Wasserstein spaces $\Wc_p(M)$ for $p\geq 2$ by giving once again a new definition of the solutions of the MDE. For these means, we will also introduce a concept of smoothness for a single measure on $TM$ by functorially defining what we shall call measure bundles over $M$. The measure differential equation we introduce bears some similarities with SDEs and endows the Wasserstein spaces with differential structures that are really less coarse than those induced by the continuity equation and overcome some important limitations of the latter.

% The final chapter is devoted to concluding remarks and an outlook.

\section{Preliminaries}\label{ch:diffprelim}

In this preliminary chapter, we introduce some basic notation we shall use in this work and recall some classical results in optimal transport theory and Riemannian geometry. Furthermore, we prove some simple results, which shall be of use later.  
\subsection{Notation}

The basic number sets will be denoted by $\Nz$ (natural numbers, $0\in\Nz$), $\Zz$ (integers), and $\Rz$ (real numbers), and
we will denote by $\Nz^*$ the set of all positive natural numbers, and by $\Rz_+$ the set of all positive real numbers.
If $a,b\in\Zz$, we will write $\llbracket a,b\rrbracket:=\{a,a+1,\dots,b\}$.

Let $(X,\tau)$ be a topological space. We will denote for $x\in X$ the neighbourhood filter at $x$, i.e., the set of all
neighbourhoods of $x$ by $\Uc(x)$.

Let $\ph:\Rz^n\to\Rz^m$ be a differentiable function. We will denote by $D\ph(x)$ the Jacobian of $\ph$ at $x\in\Rz^n$.

We will need some real function spaces, especially Sobolev spaces, throughout this work. Let $(Z,d)$ be a complete metric space, $(E,h)$ a Banach manifold and $(M,g)$ a Riemannian manifold. Furthermore, let $A \subset E$, $B \subset Z$ and $N\subset M$ be closed. We denote by $\la$ the volume form
on $(M,g)$ and define the function spaces (see \citet{michor2008topics} for formal introductions):
\begin{center}
\begin{longtable}{lp{0.77\textwidth}}
	$\Wc_p(Z)$ & The Wasserstein space of $p$--integrable probability Borel--measures on $M$ for $p \geq 1$ (see \cite{figalli2021invitation,bogachev2012monge}).\\
	$\Wc_p(\mu, \nu)$ & The $p$--Wasserstein distance between $\mu, \nu \in \Wc_p(M)$, for $p\geq1$ (see \cite{figalli2021invitation,bogachev2012monge}).\\
	$\Wc_p^{\infty}(c_1,c_2)$ & The $\sup$ distance for two continuous curves $c_1,c_2:[0,T]\to\Wc_p(M)$, for $T >0$.\\
	$\Pc_p(Z)$ & The set of all probability Borel--measures on $M$ whose $p$th moment exists for $p \geq 1$.\\
	$\Ga(\mu,\nu)$ & The set of all transport plans between $\mu,\nu\in\Pc(M)$ (see \cite{figalli2021invitation, bogachev2012monge}).\\
	$C(B,\Rz^m)$ & The set of all continuous functions $f: B\to\Rz^m$. \\
	$C^k(A,\Rz^m)$ & For $k\in\Nz\cup\{\infty,\om\}$ the set of all $k$--times continuously differentiable functions $f\in C(A,\Rz)$. For $k=\infty$ the set of all smooth functions, and for $k=\om$ and $A \subset \Rz^n$ the set of all real analytic functions. Also, $C^0(A,\Rz^m) := C(A,\Rz^m)$.\\
	$C^k_c(A,\Rz^m)$ & For $k\in\Nz\cup\{\infty\}$ the set of all $f\in C^k(A,\Rz^m)$ with $\supp f$ compact. \\
	$C^k_b(A,\Rz^m)$ & For $k\in\Nz\cup\{\infty,\om\}$ the set of all $f\in C^k(A,\Rz^m)$ where $f$ and all of its derivatives
		up to order $k$ are bounded. \\
	$L^p(N,\Rz)$ & For $1\leq p<\infty$ the space of $L^p$--functions $f: M\supseteq N\to\Rz$, i.e., the space of all (equivalence classes of) measurable functions,
	for which the integral $\int_{A} |f|^p\,\der \la$ converges. \\
	$L^\infty(A,\Rz)$ & The set of all (equivalence classes of) measurable functions $f$ that are essentially bounded, i.e., for which the essential supremum
		$\operatornamewithlimits{esup}_{A}|f(x)|<\infty$. \\
	$W^{k,p}(A,\Rz)$ & For $1\leq p<\infty$ and $k\in\Nz$, the space of functions $f\in L^p(A,\Rz)$ who are weakly $k$--times differentiable, and for
		which all derivatives up to order $k$ are also $L^p$. \\
	$H^k(A,\Rz)$ & $W^{k,2}(A,\Rz)$ \\
	$W^{k,\infty}(A,\Rz)$ & Let $(E,h)$ be a Banach space or a Riemannian manifold, either with bounded geometry or a fixed atlas. For $k\in\Nz$, the space of functions $f\in L^\infty(A,\Rz)$ who are weakly $k$--times differentiable, and for
		which all derivatives up to order $k$ are also essentially bounded.\\
	$W^{k,p}_\loc(A,\Rz)$ & For $k\in\Nz$ and $1\leq p\leq\infty$ the set of all measurable functions such that for every $x\in A$
		there exists an open neighbourhood $U\in\Uc(x)$ such that $f|_U\in W^{k,p}(U,\Rz)$. \\
	$\Lip^k(A,\Rz)$ & For $k\in\Nz$ the set of all $k$--times differentiable functions, whose $k$th derivative is Lipschitz continuous. We write $\Lip(A,\Rz):=\Lip^0(A,\Rz)$ for the space of Lipschitz continuous functions. \\
	$\Lip^k_\loc(A,\Rz)$ & For $k\in\Nz$ the set of all measurable functions, such that for every $x\in A$ there exists
		$U\in\Uc(x)$ such that $f|_U\in\Lip^k(U,\Rz)$.\\
	$\Xc(E)$ & The space of all smooth vector fields on $E$.\\
	$\Xc^{k,\infty}(E)$ & For $k\in\Nz$, the space of vector--fiels who are essentially bounded on $E$, weakly $k$--times differentiable, and for
		which all derivatives up to order $k$ are also essentially bounded. We endow it with the classical $W^{k,\infty}$--norm.\\
	$\Dc(A,\Rz)$ & $\Cc^{\infty}_c(A)$.\\	
	$\Dc'(A,\Rz)$ & The topological dual of $\Dc(A,\Rz)$, the space of distributions on $A$.\\
	$\dot{\Bc}(M\times(0,T])$ & The set of functions in $\Cc^\infty(M\times(0,T])$ for which all derivatives tend to $0$ spatially at infinity.\\
	$\dot{\Bc}'(M\times(0,T])$ & The strong dual of $\dot{\Bc}(M\times(0,T])$, all distributions vanishing spatially at infinity.
\end{longtable}
\end{center}

\subsection{Optimal transport}
First we recall some well--known results of optimal transport in the typical setting of Polish spaces and their less well known extensions
to the more general Radon metric spaces. For a comprehensive introduction to optimal transport theory, we refer to \citet{figalli2021invitation} and \citet{villani2009optimal}.

Recall that a Radon metric space is a metric space $(M,d)$ such that every Borel probability measure on $M$ is a Radon measure.
\begin{defi}[Radon metric space]
	A metric space $(M,d)$ is called a Radon metric space, if every Borel probability measure $\mu$ on $M$ satisfies that for every Borel set $B \subset M$,
	and every $\ep > 0$, there exists a compact set $K_{\ep} \subset B$ such that $\mu(B \setminus K_{\ep}) < \ep$.
\end{defi}
Most fudamental results of optimal transport theory, which are usually stated for Polish spaces, remain true in Radon metric spaces. This is a strong generalization, 
since the Radon spaces are far more general objects than Polish spaces. The following result is proved in \citet[page 244]{fremlin2003measure4}.
\begin{theorem}
	A complete metric space is a Radon metric space if and only if its weigth is measure--free. 
\end{theorem}
The non--existence of metric spaces with measure--free weight is consistent with ZFC---see for example \cite{fremlin2006measure}. While it is still an open question whether
also the existence of metric spaces with measure--free weight is consistent with ZFC, or whether ZFC implies that all complete metric spaces are Radon metric spaces, all spaces 
we will consider in this work can be constructed in ZFC and are, therefore, Radon metric spaces.

\subsubsection{Fundamental results}

The fundamental result of optimal transport theory is the existence of optimal transport plans between two probability measures. 
\begin{theorem}[Existence of optimal transport plans]\label{th:existenceoptimaltransportplan}
	Let $(M,d)$ be a complete Radon metric space and $\mu, \nu \in \Pc_p(M)$. Then, there exists an optimal transport plan $\ga \in \Ga(\mu,\nu)$.
\end{theorem}
\begin{proof}
	\citet[Theorem 1.2.1]{bogachev2012monge}
\end{proof}
The $p$--Wasserstein space is a complete metric space if the underlying metric space is complete.
\begin{lem}\label{le:Wassersteincomplete}
	Let $(M,d)$ be a complete Radon metric space. Then, $\Wc_p(M)$ is a complete metric space.
\end{lem}
\begin{proof}
	\citet[Theorem 2.7]{mardare2018free}
\end{proof}
First, we establish some useful links between the Wasserstein--topology and other classical topologies on spaces of Borel--probability measures.
Prokhorov's theorem characterizes relatively compact sets of probability measures with respect to the narrow topology---see \citet[Vol. 2 Theorem 8.6.2]{bogachev2007measure}.
\begin{theorem}[Prokhorov]\label{th:prokhorov}
	Let $(M,d)$ be a Radon metric space. A set $\Kc\subseteq\Pc(M)$ is relatively compact with respect to the narrow topology
	if and only if it is tight.
\end{theorem}
\begin{lem}\label{lem:weakstarnarrow}
	Let $(M,g)$ be a Riemannian manifold, and let $(\mu_n)_{n\in\Nz}$ be a sequence in $\Pc(M)$ with
	$\mu_n\xrightharpoonup{*}\mu$ for $\mu\in\Pc(M)$. Then $\mu_n\xrightharpoonup{}\mu$.
\end{lem}
\begin{proof}
	\citet[Lemma 2.1.13]{figalli2021invitation}.
\end{proof}
Prokhorov's theorem allows us to directly compare the Wasserstein topology and the narrow toplogy.
\begin{theorem}\label{th:Wassersteinconvergence}
	Let $(M,d)$ be a Polish metric space, $x_0\in M$, and $1\leq p<\infty$. For a sequence $(\mu_k)_{k\in\Nz}$
	in $\Pc_p(M)$ and $\mu\in\Pc_p(M)$ the following are equivalent:
	\begin{enumerate}
		\item $\lim_{n\to\infty} W_p(\mu_n,\mu)=0$,
		\item $\mu_n\xrightharpoonup{}\mu$ and $\lim_{n\to\infty}\int_M d(x,x_0)^p\,\der\mu_n(x)=\int_M d(x,x_0)^p\,\der\mu(x)$.
	\end{enumerate}
	Furthermore, if $(M,d)$ is locally compact (or in particular a Riemannian manifold), then, the above are also equivalent to
	\begin{itemize}
		\item	$\mu_n\xrightharpoonup{*}\mu$ and $\lim_{n\to\infty}\int_M d(x,x_0)^p\,\der\mu_n(x)=\int_M d(x,x_0)^p\,\der\mu(x)$.
	\end{itemize}
\end{theorem}
\begin{proof}
	\citet[Theorem 7.1.5]{ambrosio2005gradient}
\end{proof}
\begin{co}
	Theorem~\ref{th:Wassersteinconvergence} also holds if $M$ is any Radon metric space.
\end{co}
\begin{proof}
	Let $(\mu_n)_{n \in \Nz}$ be a sequence in $\Pc_p(M)$ and $\mu \in \Pc_p(M)$, such that $\mu_n \xrightharpoonup{*} \mu$ and $\lim_{n\to\infty}\int_M d(x,x_0)^p \der\mu_n(x)=\int_M d(x,x_0)^p \der\mu(x)$.
	Since $M$ is a Radon metric space, there exist compact sets $K_{n,N} \subset M$ and
	$K_N \subset M$, such that for all $n,N \in \Nz$, 
	\begin{equation}
		\begin{aligned}
			&\mu_n(M \setminus K_{n,N}) < \frac{1}{N},\\
			&\mu(M \setminus K_N) < \frac{1}{N}.
		\end{aligned}
	\end{equation}	
	Define  $m \coloneqq \bigcup_{n, N \in \Nz} K_{n,N} \cup \bigcup_{N \in \Nz} K_N$. In particular, $\supp(\mu) \subset \ol{m}$
	and for all $n \in \Nz$, $\supp(\mu_n) \subset \ol{m}$. Since $M$ is a metric space, $\ol{m}$ is separable and complete and, therefore, a Polish space.  
	Now, the result follows from Theorem~\ref{th:Wassersteinconvergence} applied to the Polish space $\ol{m}$.
\end{proof}
The next results are some useful tools to determine or bound the Wasserstein distance between two probability measures.
\begin{theorem}[(Monge-)Kantorovich duality]\label{th:kantorovichduality}
	Let $(M,g)$ be a Polish space (or in particular a Riemannian manifold) and $\mu, \nu \in \Wc_1(M)$. Then,
	\begin{equation}
		W_1(\mu,\nu) = \sup\biggl\{\int_M f \der \mu - \int_M f \der \nu \,\biggl|\, f\in \Lip(M)\text{, $1$--Lipschitz}\biggr\}.
	\end{equation} 
\end{theorem}
\begin{proof}
	This is a special case of a more general result proved in \citet[Theorem 5.10]{villani2009optimal}.
\end{proof}
If we work on more general metric spaces, we recover a weaker version of the Monge--Kantorovich duality.
\begin{lem}[Weak (Monge-)Kantorovich duality]\label{th:weakkantorovichduality}
	Let $(M,d)$ be a metric space and $\mu, \nu \in \Wc_1(M)$. Then, for all $f \in \Lip(M)$, such that $f$ is $1$--Lipschitz,
	\begin{equation}
		 \int_M f \der \mu - \int_M f \der \nu \leq W_1(\mu,\nu).
	\end{equation} 
\end{lem}
\begin{proof}
	Denote by $\pi_1, \pi_2: M \times M \to M$ the projections on the first and the second coordinates respectively. Let $\ga$ be a transport plan from $\mu$ to $\nu$. Then, 
	\begin{equation}
		\begin{aligned}
			&\int_M f(x) \der \mu(x) - \int_M f(y) \der \nu(y)\\
			&\qquad = \int_M f(\pi_1(x,y)) \der \ga(x,y) - \int_M f(\pi_2(x,y)) \der \ga(x,y)\\
			&\leq \int_M \bigl|f(\pi_1(x,y)) - f(\pi_2(x,y))\bigr| \der \ga(x,y)\\
			&\leq \int_M d(x,y) \der \ga(x,y).\\
		\end{aligned}
	\end{equation}
	Since $\ga$ was chosen arbitrarily, the result follows.
\end{proof}

\subsubsection{Wasserstein embeddings}

Now, we study some basic properties of embeddings of Wasserstein spaces.
\begin{pr}\label{pr:wassersteinidentity}
	Let $(M,d)$ be a metric space. Let $1\leq q\leq p$. Then the embedding $\iota:\Wc_p(M)\to\Wc_q(M)$ is $1$--Lipschitz.
\end{pr}
\begin{proof}
	Let $\mu, \nu \in \Wc_p(M)$ and let $\ga \in \Ga^{\opt}(\mu,\nu)$ for the $p$--Wasserstein distance. Then, 
	Jensen's inequality yields
	\begin{equation}
		\begin{aligned}
		\Wc_q(\mu,\nu) &= \biggl(\int_{M\times M} d(x,y)^q \der\ga(x,y)\biggr)^{\frac{1}{q}} \leq \biggl(\int_{M\times M} d(x,y)^p \der\ga(x,y)\biggr)^{\frac{1}{p}} \\&= \Wc_p(\mu,\nu).
		\end{aligned}
	\end{equation} 
	Thus, the embedding is $1$--Lipschitz.
	% Let $\mu\in\Pc_p(M)$ then $\mu\in\Pc_q(M)$: We define $A:=\{y\in M\mid d(x_0,y)<1\}$
	% and have
	% $$
	% 	\begin{aligned}
	% 		\int_M d(x_0,x)^q\,\der\mu(x) &= \int_{A} d(x_0,x)^q\,\der\mu(x) + \int_{A^c} d(x_0,x)^q\,\der\mu(x)\\
	% 			&\leq \int_{A} 1\,\der\mu(x) + \int_{A^c} d(x_0,x)^p\,\der\mu(x) \\
	% 			&\leq 1 + \int_{M} d(x_0,x)^p\,\der\de\mu(x) < \infty.
	% 	\end{aligned}
	% $$
	% Now let $\mu,\nu\in\Pc_p(M)$. We define $B_{r,s}:=\{(x,y)\mid r\leq d(x,y)^p<s\}$ and get
	% $$
	% 	\begin{aligned}
	% 	W_p(\mu,\nu)^p &= \int_{M\x M} \!\!\!\!\!d(x,y)^p\,\der\ga(x,y) \geq \int_{B_{\be,\infty}} \!\!\!\!\!d(x,y)^p\,\der\ga(x,y)
	% 		\geq \int_{B_{\be,\infty}} \!\!\!\!\!\be\,\der\ga(x,y)\\ &= \be\ga(B_{\be,\infty}).
	% 	\end{aligned}
	% $$
	% Hence, $\ga(B_{\be,\infty})\leq \tfrac1\be W_p(\mu,\nu)^p$.
	
	% Now let $0<\ep<1$, and $\mu,\nu$ such that $W_p(\mu,\nu)<\ep^{\frac{p+q}p}$.
	% Then let $\ga\in\Ga^{\opt}(\mu,\nu;d^p)$, and we get 
	% $$
	% 	\begin{aligned}
	% 		W_q(\mu,\nu)^q &\leq \int_M d(x,y)^q\,\der\ga(x,y) \\
	% 			&= \int_{B_{0,\ep^p}} \!\!\!\!\!d(x,y)^q\,\der\ga(x,y) + \int_{B_{\ep^p,1}} \!\!\!\!\!d(x,y)^q\,\der\ga(x,y) + \int_{B_{1,\infty}} \!\!\!\!\!d(x,y)^q\,\der\ga(x,y) \\
	% 			&\leq \int_{B_{0,\ep^p}} \!\!\!\!\!(\ep^p)^{\frac qp}\,\der\ga(x,y) + \int_{B_{\ep^p,1}} \!\!\!\!\!\der\ga(x,y) + \int_{B_{1,\infty}} \!\!\!\!\!d(x,y)^p\,\der\ga(x,y) \\
	% 			&< \ep^q + \ep^q + \ep^{p+q} \leq 3\ep^q.
	% 	\end{aligned}
	% $$
	% Therefore, $W_q(\mu,\nu)<3^{\frac1q}\ep$.
\end{proof}

Another elementary result, which will be useful several times later on is that a Riemannian isometry induces an isometry on Wasserstein spaces.

\begin{lem}\label{le:Wisom}
Let $(M,g)$ and $(N,h)$ be two Riemannian manifolds and $f: M \to N$ a Riemannian isometry (i.e., by definition surjective). Then, the push-forward $\mu \mapsto f_* \mu$ under $f$ is an isometry from $\Wc_p(M)$ to $\Wc_p(N)$ for any $p$.
\end{lem}
\begin{proof}
Take $\mu$, $\nu \in \Wc_p(M)$ and $\ga \in \Ga(\mu,\nu)$ a transport plan between $\mu$ and $\nu$, i.e., a probability measure $\ga$ on $M\times M$ satisfying $(\pi_1)_* \ga = \mu$ and $(\pi_2)_* \ga = \nu$. Then, one readily shows that $(f,f)_* \ga$ is a transport plan between $f_* \mu$ and $f_* \nu$ and we have 
\begin{equation}
\begin{aligned}
\int_{M \times M} d(x,y)^p \der (f,f)_* \ga(x,y) &= \int_{M \times M} d(f(x),f(y))^p \der \ga(x,y) \\ &= \int_{M \times M} d(x,y)^p \der \ga(x,y).
\end{aligned}
\end{equation}
This implies $\Wc_p(\mu, \nu) = \Wc_p(f_* \mu, f_* \nu)$. Now, since $f$ is bijective by assumption, for every transport plan $\ol{\ga}$ between $f_* \mu$ and $f_* \nu$, $(f^{-1},f^{-1})_* \ga(x,y)$ is a transport plan between $\mu$ and $\nu$. We conclude with a computation similar to the above that $\Wc_p(f_* \mu, f_* \nu) = \Wc_p(\mu,\nu)$.
\end{proof}

\begin{lem}\label{le:intergeodesic}
	Let $(M,g)$ be a Riemannian manifold, and let $\mu,\nu\in\Pc_1(M)$. Then $t\mapsto t\mu+(1-t)\nu$ is a geodesic
	in $\Wc_1(M)$ connecting $\mu$ and $\nu$.
\end{lem}
\begin{proof}
	Let $\mu,\nu\in\Pc_1(M)$, and let $\la\in[0,1]$. We denote by $\De:M\to M\x M$, $x\mapsto(x,x)$. Let $\ga\in\Ga^{\opt}(\mu,\nu)$,
	then $\ol{\ga}_\la:=\la\De_*\mu+(1-\la)\ga\in\Ga(\mu,\la\mu+(1-\la)\nu)$, since
	$$
		\begin{aligned}
			(\pi_1)_*(\la\De_*\mu+(1-\la)\ga) &= \la\mu+(1-\la)\mu = \mu\\
			(\pi_2)_*(\la\De_*\mu+(1-\la)\ga) &= \la\mu+(1-\la)\nu.
		\end{aligned}
	$$
	Therefore,
	$$
		\begin{aligned}
			\Wc_1(\mu,\la\mu+(1-\la)\nu) &\leq \int_{M\x M}d(x,y)\,\der\ol{\ga}_\la(x,y) \\
				&= (1-\la)\int_{M\x M} d(x,y)\,\der\ga(x,y) \\
				&= (1-\la)\Wc_1(\mu,\nu) = \la \Wc_1(\mu,\mu)+(1-\la) \Wc_1(\mu,\nu).
		\end{aligned}
	$$
	In addition, $\Wc_1(\la\mu+(1-\la)\nu,\nu)\leq \la \Wc_1(\mu,\nu)+(1-\la)\Wc_1(\nu,\nu)$, and we have
	$$
		\begin{aligned}
			\Wc_1(\mu,\nu) &\leq \Wc_1(\mu,\la\mu+(1-\la)\nu)+\Wc_1(\la\mu+(1-\la)\nu,\nu)\\&\leq (1-\la)\Wc_1(\mu,\nu)+\la \Wc_1(\mu,\nu) = \Wc_1(\mu,\nu).
		\end{aligned}
	$$
	Therefore, $\Wc_1(\mu,\la\mu+(1-\la)\nu)=(1-\la)\Wc_1(\mu,\nu)$, so the straight line connects, and the rest follows from the
	triangle inequality.
\end{proof}

\subsubsection{Functionals on Wasserstein spaces}

Next, we study the properties of some particular functionals defined on Wasserstein spaces. We prove the following result in the space $\Wc_1(M)$.

\begin{lem}\label{le:W1abs}
Take $\mu: [0,T] \to \Wc_1(M)$ absolutely continuous. Then, for all $\ph \in \Cc^{\infty}(\Rz^n)$, the map $t \mapsto \int_M \ph \der \mu(t)$ is absolutely continuous.
\end{lem}
\begin{proof}
Since $t \mapsto \mu(t)$ is absolutely continuous in $\Wc_1(M)$, there exists $m \in L^1([0,T])$ such that for all $s < t \in [0,T]$, 
\begin{equation}
\Wc_p(\mu(t),\mu(s)) \leq \int_s^t m(t) \der t.
\end{equation} 
Take $\ph \in \Cc^{\infty}(\Rz^n)$. Since $\ph$ is compactly supported and smooth, there exists $L>0$ such that $\ph$ is $L$--Lipschitz. Let $\ga$ be an optimal transport plan between $\mu(s)$ and $\mu(t)$ in $\Wc_1(M)$. Then,
\begin{equation}
\begin{aligned}
\left|\int_M \ph \der\mu(t) - \int_M \ph \der\mu(s)\right| &\leq \int_{M\times M} |\ph(x) - \ph(y)|\der\ga \\
&\leq \int_{M\times M} L|x-y|\der\ga \\
&= L\Wc_1(\mu(s),\mu(t)) \leq \int_s^t Lm(t) \der t.
\end{aligned}
\end{equation}
Hence, $t \mapsto \int_M \ph \der \mu(t)$ is absolutely continuous.
\end{proof}

\begin{pr}\label{pr:VarLip}
	Let $p \geq 2$, $n\in \Nz$ and $(M,g)$ be a Riemannian manifold. The second moment
	\begin{equation}
		\Mc_2: \Wc_p(M) \to \Rz_+
	\end{equation}
	is locally Lipschitz with respect to the $p$--Wasserstein distance. Furthermore, the map
	\begin{equation}
		\begin{aligned}
			\si: \Wc_p(\Rz^n) &\to \Mc_n(\Rz)\\ 
			\mu \mapsto &\cov(\mu),
		\end{aligned}
	\end{equation}
	sending a $p$--integrable measure onto its Covariance matrix is locally Lipschitz with respect to the $p$--Wasserstein distance and
	the strong operator norm.
\end{pr}
% \begin{proof}
% 	Set $\ph: y \mapsto \| y\|^2_2$. Then, $\ph$ is $p$--convex. Let $x\in\Rz^n$ and set 
% 	\begin{equation}
% 		\ps_x: y \mapsto \ph(y) + \|x - y \|^p_2.
% 	\end{equation}
% 	A straightforward computation shows that $\ps_x$ has a unique minimum $\frac{y_0}{\| x \|_2} x$, for $y_0 > 0$ satisfying
% 	\begin{equation}\label{eq:legtransaux}
% 		\biggl( \frac{2}{p} y_0 \biggr)^{\frac{1}{p-1}} + y_0 = \|x\|_2.
% 	\end{equation}
% 	In particular $y_0 < \|x\|$, and it follows,
% 	\begin{equation}\label{eq:legtransaux2}
% 		\begin{aligned}
% 			&\inf_{y \in \Rz^2} \ps_x(y) = \ps_x(y_0) = y_0^2 + (\|x\|_2 - y_0)^p \\
% 			&\qquad = y_0^2 + \biggl( \frac{2}{p} y_0 \biggr)^{\frac{p}{p-1}} = y_0\biggl(y_0 + \biggl( \frac{2}{p}  \biggr)^{\frac{p}{p-1}}y_0^{\frac{1}{p-1}}\biggr)\\
% 			&\qquad \leq y_0\biggl(y_0 + \biggl( \frac{2}{p}  y_0 \biggr)^{\frac{1}{p-1}}\biggr) = y_0 \| x \|_2,
% 		\end{aligned}
% 	\end{equation}
% 	because $p\geq 2$. Hence, (\ref{eq:legtransaux}) yields
% 	\begin{equation}
% 		\inf_{y \in \Rz^2} \ps_x(y) \leq \|x\|_2^2. 
% 	\end{equation}
% 	Denote by $\ph[p]$ the $p$--Legendre--transform of $\ph$. With (\ref{eq:legtransaux2}), the Monge--Kantorovich--Rubinstein duality yields
% 	\begin{equation}
% 		\Wc_p(\mu, \nu) \geq  \int \ph \der \mu - \int \ph[p] \der \nu \geq \int \|x\|_2^2 \der \mu(x) - \int \|x\|_2^2 \der \nu(x),
% 	\end{equation}
% 	and, thus, finally,
% 	\begin{equation}
% 		\Wc_p(\mu, \nu) \geq \biggl| \int \|x\|_2^2 \der \mu(x) - \int \|x\|_2^2 \der \nu(x)\biggr|. 
% 	\end{equation}
% 	This immediately yields the result.
% \end{proof}
\begin{proof}
	Let $x_0 \in M$, $\mu, \nu \in \Wc_p(M)$ and let $\ga \in \Ga^{\opt}(\mu, \nu)$ be an optimal transport plan between $\mu$ and $\nu$. Then, we obtain
	\begin{equation}
		\begin{aligned}
			&\biggl|\Mc_2(\mu) - \Mc_2(\nu)\biggr| \leq \biggl| \int_{M \times M} d(x,x_0)^2 - d(y,x_0)^2 \der \ga(x,y) \biggr| \\
			&\qquad = \biggl| \int_{M \times M} (d(x,x_0) - d(y,x_0))(d(x,x_0) + d(y,x_0)) \der \ga(x,y) \biggr|\\
			&\qquad \leq \biggl(\int_{M \times M} (d(x,x_0) + d(y,x_0))^2 \der \ga(x,y)\biggr)^{\frac{1}{2}}\\&\qquad\qquad\qquad\qquad\qquad\qquad\cdot \biggl(\int_{M \times M} (d(x,x_0) - d(y,x_0))^2 \der \ga(x,y)\biggr)^{\frac{1}{2}}\\
			&\qquad \leq \sqrt{2}\bigl(\Mc_2(\mu) + \Mc_2(\nu)\bigr)^{\frac{1}{2}} \Wc_2(\mu, \nu).
		\end{aligned}
	\end{equation}  
	Let $\rho >0$. Since by Theorem~\ref{th:Wassersteinconvergence}, $\Mc_2: \Wc_p(M) \to \Rz$ is continuous, there exists $B > 0$ such that for all $\mu \in B(\de_{x_0}, \rho)$, $\Mc_2(\mu) \leq B$.
	In particular, for all $\mu, \nu \in B(\de_{x_0}, \rho)$, 
	\begin{equation}
		\biggl|\Mc_2(\mu) - \Mc_2(\nu)\biggr| \leq 2\sqrt{B}\Wc_2(\mu, \nu).
	\end{equation}
	Now, the result for the covariance matrix follows immediately.
\end{proof}
\begin{lem}\label{le:intlin}
Let $p\geq 1$, $F,H$ be two finite-dimensional vector--spaces, $\ph: F \to H$ a linear map and $\mu \in \Wc_p(F)$. Then,
\begin{equation}
\int_F \ph(x) \der \mu(x) = \ph\biggl( \int_F x \der\mu(x)\biggr).
\end{equation}
\end{lem}

\subsubsection{Compact subsets of Wasserstein spaces}

Here, we introduce the concept of tightness in the $p$--th moment, which, as we shall see, is equivalent to relative compactness. Then, we will prove some further interesting properties of families of measures in $\Wc_p(M)$ that are tight in the $p$--th moment.

\begin{defi}
\small
Let $(M,g)$ be a metric space. Fix $x_0 \in M$. A family $(\mu_i) \in \Wc_p(M)^I$ is called tight in the $p$--th moment if for all $\epsilon > 0$, there exists a compact set $K$ such that for all $i\in I$, $\int_{M\setminus K} d(x,x_0)^p \der \mu_i(x) \leq \epsilon$.
\end{defi}

\begin{lem}\label{le:q-tight}
Let $(M,d)$ be a metric space. Take $x_0 \in M$. Let $(\mu_i)_{i \in I}$ be a family of measures such that there exists $B > 0$ such that for all $i\in I$, $\Mc_p(\mu_i) \leq B^p$. Then for all $R > 0$ and all $i \in I$, 
\begin{equation}
\mu_i(B(x_0, R)) \geq 1 - \min\biggl(1,\frac{B^p}{R^p}\biggr).
\end{equation}
Furthermore, if $(M,d)$ is a Riemannian manifold, $(\mu_i)_{i \in I}$ is tight in the $q$--th moment for all $q < p$. More precisely, for all $R\geq R_{\ep}$, with
\begin{equation}
R_{\ep} \coloneqq \begin{cases}
   0 & B^q\frac p{p-q}\leq\ep \\
   \biggl(B^p\frac{p}{p-q}\biggr)^{\frac{1}{p-q}}\ep^{\frac{1}{q-p}} & \text{otherwise},
	\end{cases}
\end{equation}
we obtain
\begin{equation}
\int_{M \setminus B(x_0, R_{\ep})} d(x_0,x)^q \der \mu_i(x) \leq \ep.
\end{equation}
\end{lem}
\begin{proof}
The first point is a straightforward consequence of the Markov inequality.

Now, let $q<p$, $i\in I$, $\ep>0$ and $R>0$. Then, if $R\leq B$, obviously,
\begin{equation}
\int_{M \setminus B(x_0, R)} d(x_0,x)^q \der \mu_i(x) \leq B^q\frac p{p-q}.
\end{equation}
Hence, if $\ep> B^q\frac p{p-q}$, the integral is bounded by $\ep$ for all $R\geq 0$.
Otherwise,
\begin{equation}
\begin{aligned}
&\int_{M \setminus B(x_0, R)} d(x_0,x)^q \der \mu_i(x) \\
&\qquad = \int_{R^q}^{\infty} \mu_i(d(x_0,x)^q>\la) \der \la + \int_0^{R^q} \mu_i(M\setminus \ol{B(x_0,R^q)}) \der \la\\
&\qquad\leq \int_{R^q}^{\infty} \mu_i(M\setminus B(x_0,\la^{\frac{1}{q}})) \der \la + R^q\frac{B^p}{R^p}\\
&\qquad\leq \int_{R^q}^{\infty} \frac{B^p}{\la^{\frac{p}{q}}} \der \la + \frac{B^p}{R^{p-q}} = \frac{B^p}{R^{p-q}}\biggl(\frac1{\frac{p}{q}-1}+1\biggr).
\end{aligned}
\end{equation}
Hence, if 
\begin{equation}
R \geq \bigl(\frac{B^pp}{p-q}\bigr)^{\frac{1}{p-q}}\ep^{\frac{1}{q-p}},
\end{equation}
we obtain 
\begin{equation}
\int_{M \setminus B(x_0, R_{\ep})} d(x_0,x)^q \der \mu_i(x) \leq \ep.
\end{equation}
Therefore, the $(\mu_i)_{i\in I}$ are tight in the $q$--th moment.
\end{proof}

The following results only hold in proper spaces, like Riemannian manifolds.

\begin{lem}\label{le:tightconv}
Let $(M,g)$ be a Riemannian manifold. Let $x_0 \in M$. Let $(\mu_i) \in \Wc_p(M)^I$ be tight in the $p$--th moment. Then $(\mu_i)$ is tight. Furthermore, let $(i_n)_{n\in\Nz}$ be a sequence in $I$ such that $(\mu_{i_n})$ converges in the weak--*--topology. Then, $(\mu_{i_n})$ converges with respect to the $\Wc_p$--distance.
\end{lem}
\begin{proof}
Let $\eps > 0$ and take $K$ a compact set such that for all $i\in I$, 
\begin{equation}
\int_{M \setminus K} d(x,x_0)^p \der \mu_i(x) < \eps.
\end{equation}
Then, in particular $\int_{M \setminus (K \cup B(x_0,1))} d(x,x_0)^p \der \mu_i(x) < \eps$. Hence, we have 
\begin{equation}
\mu_i\bigl(M \setminus (K \cup B(x_0,1))\bigr) \leq \eps
\end{equation}
 and we conclude by compactness of $K \cup B(x_0,1)$ that $(\mu_i)$ is tight.

Now, suppose $(\mu_{i_n})$ is a sequence in $(\mu_i)$ that converges weak-*, hence it converges narrowly to some probability measure $\mu$ because the $(\mu_i)$ are tight by Prokhorov's Theorem \cite[Theorem 2.1.11]{figalli2021invitation}. Let $R > 0$ such that $K \subset B(x_0,R)$. We have $\int_{M \setminus B(x_0,R)} d(x_0,x)^p - R^p \der \mu_{i_n}(x) < \eps$ for all $n \in \Nz$. On the other hand, $d_R: x \mapsto \min(R^p, d(x_0,x)^p)$ is bounded and continuous. Hence,
\begin{equation}
\lim_{n \to \infty} \int_M d_R(x) \der \mu_{i_n}(x) = \int_M d_R(x) \der \mu(x).
\end{equation}
In particular, we conclude, that $(\Mc_p(\mu_{i_n}))$ is a Cauchy--sequence and, thus, converges to a limit $l$.
Now, for all $n \in \Nz$,
\begin{equation}
\Mc_p(\mu_{n_i}) = \int_0^{\infty} \mu_{i_n}(\{d(x_0,x)^p > \la\}) \der \la.
\end{equation}
But for all $\la > 0$, $\lim_{n \to \infty} \mu_{i_n}(\{d(x_0,x)^p > \la\}) = \mu(\{d(x_0,x)^p > \la\})$ by narrow convergence. Hence, we get $\Mc_p(\mu) \leq \lim_{n \to \infty} \Mc_p(\mu_{i_n}) < \infty$ from Fatou's Lemma and, thus, a posteriori
\begin{equation}
\lim_{n \to \infty} \Mc_p(\mu_{i_n}) = \Mc_p(\mu).
\end{equation}
Therefore, $(\mu_{i_n})$ converges to $\mu$ with respect to the $\Wc_p$--distance by Theorem~\ref{th:Wassersteinconvergence}.
\end{proof}

\begin{co}\label{co:compacttight}
	Let $(\mu_i) \in \Wc_p(M)^I$ be tight in the $p$--th moment. Then $(\mu_i)$ is relatively compact in $\Wc_p(M)$.
\end{co}
The converse is also true.
\begin{lem}\label{le:compacttight}
Let $(\mu_i) \in \Wc_p(M)^I$ be relatively compact in $\Wc_p(M)$. Then $(\mu_i)$, is tight in the $p$--th moment.
\end{lem}
\begin{proof}
Suppose there exists $\ep >0$ and a sequence $(\mu_n)_{n \in \Nz}$ in $(\mu_i)$, such that for all $n \in \Nz$, $\mu_n(B(x_0, N)) \leq 1 - \ep$. 

Since, $(\mu_i)$ is relatively compact, there exists a subsequence $(\nu_n)_{n\in\Nz}$ of $(\mu_n)_{n\in\Nz}$ converging to some $\nu \in \Wc_p(M)$. Define for $n \in \Nz$,
\begin{equation}
\be_n \coloneqq \frac{1}{\Mc_p(\nu_n)}d(x_0,x)^p\nu_n.
\end{equation}
Since, $(\nu_n)_{n\in\Nz}$ converges to $\nu$ in $\Wc_p(M)$, in particular $\Mc_p(\nu_n)$ converges to $\Mc_p(\nu)$. Thus, $(\be_n)_{n\in\Nz}$ narrowly converges to $\be \coloneqq \frac{1}{\Mc_p(\nu)}d(x_0,x)^p\nu$. So, in particular, $(\be_n)_{n\in\Nz}$ is tight. Since the $\Mc_p(\nu_n)$ are uniformly bounded, we conclude that $(\nu_n)_{n\in\Nz}$ is tight in the $p$--th moment.
\end{proof}

% \begin{pr}\label{pr:moments}
% 	Let $p\geq 2$ and $x_0 \in M$. The map 
% 	\begin{equation}
% 		\begin{aligned}
% 			\Mc_p: \Wc_p(M) &\to \Rz_+^*\\ 
% 			\mu \mapsto &\int_{M} d(x,x_0)^p \der \mu(x),
% 		\end{aligned}
% 	\end{equation}
% 	is $1$--Lipschitz with respect to the $p$--Wasserstein distance.
% \end{pr}
% \begin{proof}
% 	Let $\mu, \nu \in \Wc_p(M)$. Set $\ph: y \mapsto \| y\|^p_2$. Then, $\ph$ is $p$--convex. Let $x\in\Rz^n$ and set 
% 	\begin{equation}
% 		\ps_x: y \mapsto \ph(y) + \|x - y \|^p_2.
% 	\end{equation}
% 	A straightforward computation shows that $\ps_x$ has a unique minimum in $\frac{x}{2}$ and
% 	\begin{equation}
% 		\begin{aligned}
% 			\ps_x\biggl(\frac{x}{2}\biggr) = \frac{\|x\|^p_2}{2^{p-1}} \leq \|x\|^p_2.
% 		\end{aligned}
% 	\end{equation}
% 	Denote by $\ph[p]$ the $p$--Legendre--transform of $\ph$. Now, the Monge--Kantorovich--Rubinstein duality yields
% 	\begin{equation}
% 		\Wc_p(\mu, \nu)^p \geq  \int \ph \der \mu - \int \ph[p] \der \nu \geq \int \|x\|_2^p \der \mu(x) - \int \|x\|_2^p \der \nu(x),
% 	\end{equation}
% 	and, thus, finally,
% 	\begin{equation}
% 		\Wc_p(\mu, \nu)^p \geq \biggl| \int \|x\|_2^p \der \mu(x) - \int \|x\|_2^p \der \nu(x)\biggr|. 
% 	\end{equation}
% \end{proof}

\subsection{Riemannian Geometry}

One of the most prominent and useful results in Riemannian geometry is the Nash embedding theorem, which states that any finite-dimensional Riemannian manifold is essentially a submanifold of $\Rz^n$ for some $n \in \Nz$.
\begin{theorem}[Embedding theorem, Nash]\label{th:Nash}
Let $(M,g)$ be a finite dimensional Riemannian manifold. Then there exists $n\in\Nz$ and an isometric embedding $f:(M,g)\to(\Rz^n,g^{eucl})$. 
\end{theorem}
\begin{proof}
	The initial proof of the theorem can be found in \citet{nash1954c1} with an improved proof in \citet{gunther1989einbettungssatz}.
\end{proof}

\subsubsection{Extension of smooth vector--fields}

The Nash Embedding Theorem might be useful to generalize results from Euclidean spaces to more general Riemannian manifolds. However, most embedded images of manifolds in $\Rz^n$ are closed submanifolds of $\Rz^n$. So, we cannot always directly derive analytic results on $M$ from analytic results on $\Rz^n$ by embedding $M$ into $\Rz^n$. Instead we need a way to define some sort of open ``thicker'' version of the embedding of $M$---which still looks similar to $M$---and methods to smoothly extend tensors on $M$ to tensors on $\Rz^n$ supported in this extension of the embedding of $M$. This can be done the following way:

Let $(M,g)$ be a Riemannian manifold and $R\subseteq M$ a submanifold. The concepts we are about to use are introduced in \citet{michor2008topics}. Denote the normal bundle of $R$ in $M$ by $N^g R$ and 
the restriction of the Levi-Cività connection on $M$ to $N^g R$ by $\nabla^g$. Furthermore, let $s_0: R \to N^g R$ be the $0$--
section. Note, that any $(k,l)$--tensor $\Ph$ on $R$ can be canonically extended to a $(k,l)$--tensor $N^g\Ph$ on $N^g R$ by 
pointwise push-forward and--or pull-back of $\Ph$ under the horizontal lifts induced by $\nabla^g$ (and--or their inverses 
respectively). In particular, if $\ph \in \Cc^{\infty}$, then $N^g \ph$ is just the function obtained by constantly extending
$\ph$ on each fiber.

Now, the key result to transform these smooth extensions on $N^g R$ into smooth extensions on $M$ is the existence of a so-called tubular neighbourhood.

\begin{theorem}[Existence of a tubular neighbourhood]\label{th:extension}
Let $(M,g)$ be a Riemannian manifold, and let $R\subset M$ be a submanifold.  
Then, there exists an open neighbourhood $T \subset N^g R$ of $s_0(R)$ and $t: T \to M$, such that $t \circ s_0 = \id$ and $t$ is a diffeomorphism onto its image. Such a neighbourhood $T$ is called a \define{tubular neighbourhood} of $R$ in $M$.

Furthermore, $t$ can be chosen such that for every $p\in M$---under the identification of $T_{(p,0)} N^g R$ with $T_p M$---$d_p t: T_{(p,0)} N^g R \to T_p M$ is the identity. In that case, we call $t$ a normal tubular neighbourhood. We shall denote $T = (T_p)_{p\in R}$.
\end{theorem}
\begin{proof}
	\citet[pp.~109--118]{hirsch2012differential}.
\end{proof}

We fix a normal tubular neighbourhood $T$ of $R$ in $M$ and locally endow $N^g R$ with the metric $\ol{g}= t^* g$. Now, we are ready to state the extension theorem.

\begin{theorem}[Extension of tensors]
Let $\Ph$ be a tensor of class $\Cc^p$ on $R$, with $p\in\Nz$. Then, we can extend $T$ to a tensor $\ol{\Ph}$ of class $\Cc^p$ on $M$. 
\end{theorem}

\begin{proof}
Let $\eta$ be a smooth function such that $\eta(p) = 1$ for all $p\in R$ and $\eta$ is supported in $t(T)$  and $0\leq \eta \leq 1$. Then, the tensor $\eta t_* \Ph$ is a smooth extension of $\Ph$ on $M$ with support in $t(T)$. 
\end{proof}

We even have the stronger albeit slightly more technical result, which shall however prove to be useful later on.

\begin{lem}\label{le:extensionprops}
Let $(M,g)$ be a Riemannian manifold and $R\subset M$ a submanifold. 
If $X\in\Ga^p(TR)$ is a $\Cc^p$ vector field, then we can find a $\Cc^p$ extension $\ol{X}\in\Ga^p(TM)$ of $X$ such that for all $p \in R$ and all $v \in T_p$, $|X_{(p,v)}| = |\ol{X}_{(p,v)}|$. 

In particular, if $\ph \in \Cc_c^{\infty}(R)$ it can be normally extended to a function $\ol{\ph} \in \Cc_c^{\infty}(M)$, i.e., such that for all $p\in R$, the restriction of $\der_{p} \ol{\ph}$ to $N^g_p R$ is zero.

In particular, this implies that if $t \mapsto v_t$ is a curve of vector fields, such that for all $p \in R$, $t \mapsto |v_t(p)|_g$ is $L^1$--integrable on an open interval $I$ of $\Rz$, then for all $x \in M$, $t \mapsto |v_t(x)|_g$ is $L^1$.

Furthermore, if $X$ is $L$--Lipschitz on $R$, then for every compact set $K$, there exists a constant $a$ depending only on $K$ and an extension $\ol{X}$ such that $\ol{X}$ is $(aL)$--Lipschitz. 
\end{lem}
\begin{proof}
Let $\eta$ be a smooth function such that $\eta(p) = 1$ for all $p\in R$ and $\eta$ is supported in $t(T)$ and $0\leq \eta\leq 1$. 
We define the extension in a slightly more involved way than in the proof of Theorem~\ref{th:extension}. Let 
$$\ol{N^g} X_{(p,v)} \coloneqq \frac{|(X_p)|}{|(N^g X)_{(p,v)}|}(N^g X)_{(p,v)} \text{ and define } \ol{X} = \eta t_* \ol{N^g} X.$$ 

Then, one readily shows that $\ol{X}$ has the same regularity as $X$. For the second point, we define $\ol{\ph} \coloneqq \eta N^g \ph \circ f^{-1}$. The required property follows from the fact that $\der f_{(p,0)} = \id$ for all $p\in R$.

Now take $t \mapsto v_t$ such that $t \mapsto |v_t(p)|_g$ is $L^1$-- integrable on an open interval $I$ of $\Rz$ for all $p\in R$. Take $x \in M$.

We have $\ol{v_t}_x = \eta(x) \der_{t^{-1}(x)} t( \ol{N^g} (v_t)_{t^{-1}(x)})$, hence, 
\begin{equation}
\begin{aligned}
\int_I |\ol{v_t}(x)| \der t &= \int_I |\eta(x) \der_{t^{-1}(x)} t( \ol{N^g} v_t(t^{-1}(x)))| \der t\\
&= \int_I |\eta(x) \der_{t^{-1}(x)} t( \ol{N^g} v_t(t^{-1}(x)))| \der t \\ &\leq \int_I |\der_{t^{-1}(x)} t( \ol{N^g} v_t(t^{-1}(x)))| \der t \\ &= \int_I |v_t(p)| \der t \leq \infty,
\end{aligned}
\end{equation}
by definition of $\ol{N^g} X$.

Finally, we shall only sketch the proof of the last part, which is more technical. The idea is, that for every $(p,v), (q,w) \in N^g R$, the distance between $(p,v)$ and $(q,w)$ induced by $t^* g$ can locally be sufficiently well controlled by the distance between $p$ and $q$ compared to the variation of $\der_{(p,v)} t$. This follows in principle since $t$ is a diffeomorphism and $\der_p t = \id$ for all $p\in R$. Hence, by making $T$ sufficiently ``thin'', we get a Lipschitz constant for $\ol{X}$, which depends, for any compact subset $K$ of $M$, only on the supremum of the operator norm of $T t$ on $K$.
\end{proof}

\begin{re}
A Lipschitz vector field on $R$ can in general not be extended to a Lipschitz vector field on $T$. In fact, if we define just any metric $h$ on $N^g R$ which coincides with $g$ on $T_{(p,0)} N^g R$, then there does not exist an isometric tubular neighbourhood in general. 

Just imagine an isometric embedding of a non--compact surface with a curvature that takes arbitrarily high values, i.e., that locally looks like very small spheres. Since the tubular neighbourhood cannot intersect itself, it becomes arbitrarily ``thin'' in such points. Thus, we cannot control the derivative of the cutoff function $\eta$ introduced in the previous proof and hence we cannot control the Lipschitz constant of $\ol{X}$ either.

However, one can show, that if $R$ is of bounded geometry, i.e., the Riemann curvature is bounded in all derivatives, then Lipschitz tensor fields can be extended to Lipschitz tensors, because of the existence of a partition of unity with uniformly bounded derivatives.
\end{re}

\subsubsection{Smooth approximations of Lipschitz functions}

We prove that any Lipschitz function $f \in W^{1,\infty}_{loc}(M)$, on a compact manifold $M$, can be arbitrarily well approximated by $\Cc^1$ functions with the same Lipschitz constant whose second derivatives do not grow ``too fast''. We proceed by first proving the result on $\Rz^n$ and then embedding $M$ into some $\Rz^n$.

\begin{pr}\label{pr:smoothapprox}
Let $n \geq 1$ and $f \in W^{1,\infty}_{loc}(\Rz^n)$ be $1$--Lipschitz. Then there exists a constant $K > 0$ depending only on $n$ such that for all $\eps > 0$, there exists $f_{\eps} \in \Cc^{\infty}(M)$ $1$--Lipschitz with
\begin{equation}
\begin{aligned}
&\|f - f_{\eps}\|_{\infty} < \eps \\
&\|\nabla^2 f_{\eps}\|_{\infty} < \frac{K}{\eps}.
\end{aligned}
\end{equation}
\end{pr}
\begin{proof}
Define 
\begin{equation}
\begin{aligned}
g_{\eps}: \Rz^n &\to \Rz \\
x &\mapsto \frac{1}{(\sqrt{2\pi}\eps)^n}\exp\bigl(-\frac{\|x\|^2}{2\eps^2}\bigr),
\end{aligned}
\end{equation}
the density of the isotropic normal distribution of standard deviation $\eps$. Without loss of generality, we can assume $f$ to satisfy $f(0) = 0$. Since $f$ is $1$--Lipschitz, for all $x \in \Rz^n$, $f(x) \leq \|x\|$. Now, a straightforward computation shows that
$$
f_{\eps} \coloneqq f * g_{\eps}
$$
satisfies the desired property for some $K > 0$ that depends only on $n$.
\end{proof}

The analogous result on $M$ follows from a standard embedding argument.

\begin{co}\label{co:smoothapprox}
Let $n \geq 1$, $M$ be an $n$--dimensional manifold, with or without border, with bounded geometry---or in particular compact---, and $f \in W^{1,\infty}_{loc}(M)$ be $1$--Lipschitz. Then there exists a constant $K_M > 0$ depending only on $n$ and $M$ such that for all $\eps > 0$, there exists $f_{\eps} \in W^{2,\infty}_{loc}(\Rz^n)$ $K_M$--Lipschitz with
\begin{equation}
\begin{aligned}
&\|f - f_{\eps}\|_{\infty} < \eps \\
&\|\nabla^2 f_{\eps}\|_{\infty} < \frac{K_M}{\eps}.
\end{aligned}
\end{equation}
\end{co}
\begin{proof}
for compact manifolds, the Nash embedding Theorem~ yields the existence of $n \geq 1$ and an isometric embedding $i: M \to \Rz^n$. 

Now, $f \circ i^{-1}: i(M) \to \Rz$ is $1$--Lipschitz. Since $M$ is compact, analogously to the last point of Lemma~\ref{le:extensionprops}, there exists $a \geq 1$ --- depending only on $M$ --- such that $f \circ i^{-1}$ can be extended to an $a$--Lipschitz function $\ol{f}$ defined on $\Rz^n$.

Let $\eps > 0$. Next, let $\ol{f}_{\eps}$ be the smooth approximation of $\ol{f}$ given by Proposition~\ref{pr:smoothapprox}. Then $f_{\eps} \coloneqq \ol{f}_{\eps} \circ i$ satisfies $\|f_{\eps}\|_{W^{2,\infty}} \leq  \|\ol{f}_{\eps}\|_{W^{2,\infty}} \leq \frac{K_M}{\eps}$ and has, thus, the desired properties.

If we assume just bounded geometry for $M$, this follows from a local application of Proposition~\ref{pr:smoothapprox} using a $\Cc^2$--uniformly bounded partition of unity.
\end{proof}

\sectionmark{Parabolic differential structures}
\section{Parabolic differential structures on Wasserstein\\ spaces}\sectionmark{Parabolic differential structures}\label{ch:parstructwas}

\subsection{Tangent bundle and probability vector--fields}\label{sec:vectprob} 

Let us start with a formal setup and introduce some notation. Let $p > 2$, $(M,g)$ be a complete finite dimensional Riemannian manifold and $\Xc^{2,\infty}(M)$ the space of vector--fields on $M$ with $W^{2,\infty}$--regularity, endowed with the $W^{2,\infty}$--norm.
We will always assume completeness of the Riemannian manifold we work with in what follows.
Denote for all $x \in M$ the evaluation in $x$ by $\pi_x : \Xc^{2,\infty}(M) \to T_x M$, $X \mapsto X_x$. Furthermore, for all $X \in \Xc^{2,\infty}(M)$ and all $t \in [0,T]$, the flow of $X$ at time $t$, denoted by $\rho_t^X: M \mapsto M$, exists by the Theorem of Picard--Lindelöf.

We first define the tangent space of $\Wc_p(M)$ at $\mu \in \Wc_p(M)$.
\begin{defi}[Tangent bundle]\label{de:tangentbundle}
The tangent space $\T^d_{\mu} \Wc_p(M)$ of $\Wc_p(M)$ at $\mu \in \Wc_p(M)$ is defined as
\begin{equation}
\T^d_{\mu} \Wc_p(M) \coloneqq \Wc_p(\Xc^{2,\infty}(M), \|~\|_{W^{2,\infty}(M)}).
\end{equation} 
The elements of $\T^d_{\mu} \Wc_p(M)$ will be referred to as tangent vectors at $\mu$ or \emph{vector--field probabilites}.
\end{defi}
\begin{re}
The space $\Xc^{2,\infty}(M)$ is independent of the choice of the atlas on $M$ if $(M,g)$ is of bounded geometry. This is not true for a general Riemannian manifold $M$. 
However, whenever we do not assume $(M,g)$ to be of bounded geometry, we will always implicitly fix an atlas on $M$. In particular, for those results we prove on more general Riemannian manifolds, one only has to find
one atlas such that all tangent vectors considered are in $\Xc^{2,\infty}(M)$ with respect to this atlas.
\end{re}
\begin{re}
Note that this gives a trivial tangent bundle, since the tangent spaces do not depend on $\mu$.
\end{re}
\begin{re}
While $\Wc_p(\Xc^{2,\infty}(M))$ is not a vector--space, it is a subset of the space of bounded signed Borel--measures $\Mc_p(\Xc^{2,\infty}(M))$. 
\end{re}
\begin{re}
	Note that $\Xc^{2,\infty}(M)$ is not a Polish space. Hence, many standard results do not hold on $\Wc_p(\Xc^{2,\infty}(M))$, like the 
	existence of optimal transport plans or the strong Monge--Kantorovich duality of Definition~\ref{th:kantorovichduality}. However, we  
	shall not need those results for measures on $\Xc^{2,\infty}(M)$ in this work.
\end{re}
Now, the definition of a vector--field on $\Wc_p(M)$ is straightforward.
\begin{defi}[Probability vector--field]\label{de:probvectfield}
Let $V$ be a map
\begin{equation}
V: \Wc_p(M) \to \Wc_p(\Xc^{2,\infty}(M)).
\end{equation}
We call $V$ a \emph{probability vector--field}.
\end{defi}
\begin{re}
We only require $\T^d_{\mu} \Mc_p(M) \subset \Wc_p(\Xc^{0,\infty}(M))$ for the following definition, where $\Xc^{0,\infty}(M)$ is the space of vector--fields on $M$ with $W^{0,\infty}$--regularity, endowed with the $W^{0,\infty}$--norm.  
However, the restriction to $L^{\infty}$--regular vector--fields is necessary to ensure that we work on a metric space. Note that many of the following results have similar proofs if we define $V$ more generally as a map taking values in $\Mc_p(\Xc^{2,\infty}(M))$. For the sake of simplicity, we will not do so here.
\end{re}
We define the barycenter of $V \in \Wc_p(\Xc^{2,\infty}(M))$ as the map of its pointwise barycenters. This motivates the following definition.
\begin{defi}[Barycenter of a probability vector--field]\label{de:barX}
Denote for all $x\in M$ and all $\mu \in \Wc_p(M)$,
$$
	\ol{V}[\mu](x) \coloneqq \mathbb E((\pi_x)_*V[\mu]).
$$
Then $x \mapsto \ol{V}[\mu](x)$ is a well-defined vector field on $M$ and shall be denoted by $\ol{V}[\mu]$. Furthermore, we define $\ol{V}: \mu \mapsto \ol{V}[\mu]$, the barycenter of $V$.
\end{defi}
A priori, $\ol{V}$ has very little regularity for more general probability vector--fields with image in $V \in \Wc_p(\Xc^{0,\infty}(M))$. However, the assumption that $V$ takes values in $\Wc_p(\Xc^{2,\infty}(M))$ ensures that $\ol{V}$ is sufficiently well--behaved.
Indeed, we obtain the following result.
\begin{lem}\label{le:bary}
Let $V$ be a probability vector--field. 
\begin{enumerate}
	\item Then, for all $\mu \in \Wc_p(M)$, 
	\begin{itemize}
		\item $\ol{V}[\mu] \in W^{2,\infty}(M)$ and $\|\ol{V}[\mu]\|_{W^{2,\infty}(M)} \leq 3\Mc_1(V[\mu])$,
		\item  for all $t,s \in [0,T]$, $(\rh^{\ol{V}[\mu]}_t)_* \mu \in \Wc_p(M)$, and
		\begin{equation}
		\Wc_p\bigl((\rh^{\ol{V}[\mu]}_t)_* \mu, (\rh^{\ol{V}[\mu]}_s)_* \mu\bigr) \leq |t-s| \Mc_1(V[\mu]),
		\end{equation}
		\item the directional derivatives of $\ol V[\mu]$ are given by 
		\begin{equation}\label{eq:avgder}
		\frac{\der\ol V[\mu](x)}{\der x_i} = \int_{\Xct(M)} \pi_x\biggl(\frac{\der X}{\der x_i}\biggr)\der V[\mu](X)
		\end{equation}
		in local coordinates.
	\end{itemize}
	\item If $V$ is continuous (resp- uniformly continuous) with respect to the $\Wc_1$--norm on $\Wc_p(\Xc^{2,\infty}(M))$--- which is in particular the case if it is continuous in the $\Wc_p$--norm (see Proposition~\ref{pr:wassersteinidentity}), then $\mu \mapsto \ol{V}[\mu]$ is continuous (resp- uniformly continuous) with respect to the $\Cc^1$--norm and for $\mu, \nu \in \Wc_p(M)$,
	\begin{equation}
	\| \ol{V}[\mu] - \ol{V}[\nu]\|_{W^{1,\infty}(M)} \leq 2\Wc_1(V[\mu], V[\nu]).
	\end{equation}
\end{enumerate}
\end{lem}
\begin{proof}
\subsubsection*{Proof of 1.}
Take $x,y \in M$. Then, 
\begin{equation}
\begin{aligned}
\int_{T_x M} (x,v) \der (\pi_x)_* V[\mu] (x,v) &= \int_{\Xct(M)} \pi_x(X) \der V[\mu](X) \\ &\leq \int_{\Xct(M)} \|X\|_{\infty} \der V[\mu](X) < \Mc_1(V[\mu]).
\end{aligned}
\end{equation}
Hence, $\ol{V}[\mu]$ is well-defined and bounded. Furthermore, by Lemma~\ref{le:intlin},
\begin{equation}
\begin{aligned}
\biggl|P^y_x\biggl(&\int_{T_x M} (x,v) \der (\pi_x)_* V[\mu] (x,v)\biggr) - \int_{T_y M} (y,v) \der (\pi_y)_* V[\mu] (y,v)\biggr|_{g_y} \\ &\leq \int_{\Xct(M)} |P^y_x \pi_x(X) - \pi_y(X)|_{g_y} \der V[\mu](X) \\ &\leq \int_{\Xct(M)} \|X\|_{W^{1,\infty}(M)} d(x,y) \der V[\mu](X) \leq d(x,y) \Mc_1(V[\mu]).
\end{aligned}
\end{equation}
Therefore, $\ol{V}[\mu] \in W^{1,\infty}(M)$.  Furthermore, for all $\mu \in \Wc_p(M)$ and all $t\in[0,T]$,
$$
	\Wc_p\bigl(\mu, (\rh^{\ol{V}[\mu]}_t)_* \mu\bigr)^p \leq \int_M  d(x, \rh^{\ol{V}[\mu]}_t(x))^p \der \mu(x) \leq t^p\Mc_1(V[\mu])^p.
$$
Therefore,
\begin{equation}
\Wc_p\bigl(\mu, (\rh^{\ol{V}[\mu]}_t)_* \mu\bigr) \leq t \Mc_1(V[\mu]),
\end{equation}
and analogously, one shows 
\begin{equation}
\Wc_p\bigl((\rh^{\ol{V}[\mu]}_s)_* \mu, (\rh^{\ol{V}[\mu]}_t)_* \mu\bigr) \leq |t-s| \Mc_1(V[\mu]),
\end{equation}
for all $s,t \in [0,T]$.
Let us now prove the second part of the statement. Suppose first $(M,g) = \Rz^n$ and let $(e_i)_{i\in \llbracket 1, n \rrbracket}$ be the canonical base. Take $\mu \in \Wc_p(\Rz^n)$, $x \in \Rz^n$, $i\in \llbracket 1, n \rrbracket$ and $h > 0$. Then, there exists a bounded function $K: \Xct(M) \to \Rz$ such that
\begin{equation}
\begin{aligned}
\ol{V}[\mu](x + h e_i) &- \ol{V}[\mu](x) = \int_{\Xct(M)} \pi_{x + h e_i}(X) - \pi_x(X) \der V[\mu](X) \\
&= \int_{\Xct(M)} h\pi_x\biggl(\frac{\der X}{\der x_i}\biggr) + K(X)(h^2)\|X\|_{W^{2,\infty}(\Rz^n)} \der V[\mu](X).
\end{aligned}
\end{equation}
Now, $\der V[\mu]$ has bounded first moment, so 
\begin{equation}
\ol{V}[\mu](x + h e_i) - \ol{V}[\mu](x) - \int_{\Xct(M)} h\pi_x\biggl(\frac{\der X}{\der x_i}\biggr)\der V[\mu](X) =_{h \to 0} O(h^2).
\end{equation} 
Thus, $\frac{\der\ol V[\mu]}{\der x_i}(x) = \int_{\Xct(M)} \pi_x(\frac{\der X}{\der x_i})\der V[\mu](X)$. Now, we readily show that $\ol{V}[\mu] \in W^{2,\infty}(M)$ and it follows that $\|\ol{V}[\mu]\|_{W^{2,\infty}(M)} \leq 3\Mc(V[\mu])$.
Finally, an analogous proof in local coordinates yields the local result on a general Riemannian manifold $(M,g)$.
\subsubsection*{Proof of 2.}
For the second point, we use once again Monge--Kantorovich duality. Take $\mu, \nu \in \Wc_p(\Rz^n)$ and $x \in M$. Indeed, we have
\begin{equation}
\begin{aligned}
(\ol{V}[\mu] - \ol{V}[\nu])(x) &= \int_{\Xct(M)} \pi_x(X) \der(V[\mu] - V[\nu])(X) \\ &\leq \int_{\Xct(M)} \|X\|_{\infty} \der(V[\mu] - V[\nu])(X) \leq \Wc_1(V[\mu], V[\nu])
\end{aligned}
\end{equation}
and similarly, (\ref{eq:avgder}) yields
\begin{equation}
\begin{aligned}
|\nabla\ol{V}[\mu] - \nabla\ol{V}[\nu]|(x) &\leq \int_{\Xct(M)} \|X\|_{W^{1,\infty}(M)} \der(V[\mu] - V[\nu])(X) \\  &\leq \Wc_1(V[\mu], V[\nu]).
\end{aligned}
\end{equation}
These bounds do not depend on $x$, and since $V$ is continuous (resp- uniformly continuous) with respect to the $\Wc_p$--distance and, hence, in particular with respect to the $\Wc_1$--distance (Proposition~\ref{pr:wassersteinidentity}), we have shown that $\mu \mapsto \ol{V}[\mu]$ is continuous (resp- uniformly continuous) with respect to the $\Cc^1(M)$--norm.
\end{proof}

\begin{re}
	Note that we only need $p \geq 1$ to prove Lemma~\ref{le:bary}.
\end{re}

\subsection{The parabolic Wasserstein--flow equation (MDE)}

We shall now formally state the new definition of what it means for a curve of measures to be a solution of (\ref{eq:MDE}). A motivational paragraph will follow shortly after.

\begin{defi}[MDE]\label{def:MDE}
Let $[0,T] \to \Pc(M)$, $t \mapsto \mu(t)$ be a path of measures. We say that $\mu(t)$ satisfies the measure differential equation \emph{(\ref{eq:MDE})} if $\mu(0) =\mu_0$, $\mu$ is narrowly continuous and for all $s\in [0,T],$ and all $\ph$ smooth and compactly supported on $M$, we have
\begin{equation}\label{eq:MDEd}
	\tag{MDEd}
\begin{aligned}
\int_M \ph \der \mu(s) - \int_M \ph \der \mu(0) = \int_0^s \int_M \square_{\mu(t)}^V \ph \der \mu(t) \der t,
\end{aligned}
\end{equation} 
where $\square_{\mu(t)}^V \ph = \int_{\Xct(M)} \frac{1}{2}\Lc_{X-\ol{V}[\mu(t)]}^2(\ph) + \Lc_X(\ph) \der V[\mu(t)](X)$.
\end{defi}

With Lemma~\ref{le:eflowop}, for all $\mu \in \Wc_p(M)$, $\ol{V}[\mu]$ is $\Cc^1$, hence for all $X \in \Xc^{2,\infty}(M)$ and all $\ph \in \Cc^{\infty}_c(M)$, $\Lc_{X-\ol{V}[\mu(t)]}^2$ is indeed well-defined.

\begin{lem}\label{re:altdef}
With the notation from Definition~\ref{def:MDE}, (\ref{eq:MDEd}) reduces to 
\begin{equation}
\begin{aligned}
\int_M \ph \der \mu(s) &- \int_M \ph \der \mu(0) \\&= 
\frac{1}{2}\int_0^s \int_M \int_{\Xct(M)} \Lc_{X-\ol{V}[\mu(t)]}^2(\ph) \der V[\mu(t)] \der \mu(t) \der t \\ &\qquad+ \int_0^s \int_M \Lc_{\ol{V}[\mu(t)]}(\ph) \der \mu(t) \der s.
\end{aligned}
\end{equation} 
\end{lem}
\begin{proof}
By Lemma~\ref{le:intlin} and the definition of $\ol{V}[\mu(t)]$, we have for all $x\in M$, 
\begin{equation}
\begin{aligned}
\int_{\Xct(M)} \Lc_X(\ph)(x) \der V[\mu(t)](X) &= \int_{T_x M} \der_x \ph (v) \der (\pi_x)_* V[\mu(t)](v)\\ &=  \der_x \ph \biggl(\int_{T_x M} v \der (\pi_x)_* V[\mu(t)](v)\biggr)\\ &= \der_x \ph (\ol{V}[\mu(t)]).
\end{aligned}
\end{equation}
\end{proof}

Note that a given curve $t \mapsto \mu(t)$ can be solution of (\ref{eq:MDEd}) for several different probability vector--fields. In particular, one can without loss of generality assume that $V$ is symmetric.

\begin{defi}[Symmetrisation of a probability vector--field]
Let $V$ be a probability vector--field on $\Wc_p(M)$. Let 
\begin{equation}
\begin{aligned}
m_Y: \Xct(M) &\to \Xct(M) \\
X &\mapsto X - Y
\end{aligned}
\end{equation}
for all $Y \in \Xct(M)$ and 
\begin{equation}
\begin{aligned}
a: \Xct(M) &\to \Xct(M) \\
X &\mapsto -X.
\end{aligned}
\end{equation}
We define $SV$, the symmetrisation of $V$ as the probability vector--field
\begin{equation}
\begin{aligned}
SV: \Wc_p(M) &\to \Wc_p(\Xc^{2,\infty}(M)) \\
\mu &\mapsto \frac{1}{2}\biggl((m_{-\ol{V}[\mu]} \o a \o m_{\ol{V}[\mu]})_* V[\mu] + V[\mu] \biggr).
\end{aligned}
\end{equation}
\end{defi}

\begin{co}
Let $V$ be a probability vector--field and $t \mapsto \mu(t)$ be a narrowly continuous curve of probability measures that solves (\ref{eq:MDEd}) for $V$ with initial value $\mu_0$. Then,  $\mu(t)$ solves (\ref{eq:MDEd}) for $SV$.
\end{co}

\subsection{Properties of solutions of (MDE)}

Before solving (\ref{eq:MDEd}), let us first determine some basic properties of its solutions. To alleviate notation, we shall often write $\ol{X}[\mu]$ instead of $X - \ol{V}[\mu]$.
\begin{pr}\label{pr:W1Holder}
Assume $(M,g)$ is of bounded geometry. Let $B \in \Rz^*_+$ and $V$ be a probability vector--field  on $\Wc_p(M)$ with $p$--th moments uniformly bounded by $B$. Let $\mu: [0,T] \to \Wc_1(M)$ be a narrowly continuous curve of probability measures that solves (\ref{eq:MDEd}) for $V$ with initial value $\mu_0$. Then,  $t \mapsto \mu(t)$ is Hölder--continuous with exponent $\tfrac12$ in $\Wc_1(M)$.
\end{pr}

\begin{proof}
Let $\ph \in \Cc^{\infty}_c(M)$. There exist fixed constants $B', B''>0$ --- depending only on the dimension $n$ of $M$ and the uniform bound $B$ on the $p$--th moments of the $V[\mu]$ --- such that
$$
\begin{aligned}
	&\biggl|\int_M\int_{\Xct(M)} \frac{1}{2} \Lc^2_{\ol{X}[\mu(u)]} \ph (x) \der V[\mu (u)](X)\der\mu(u)(x)\biggr|\\
	&\qquad \leq B'(\| \nabla \ph \|_\infty + \| \nabla^2 \ph \|_\infty)
\end{aligned}
$$
and
$$
	\biggl|\int_M \Lc_{\ol{V}[\mu(u)]} \ph (x) \der \mu (u)(x)\biggr|\leq B''\|\nabla\ph\|_{\infty}.
$$
Hence, there exists a constant $C$ independent of $s$, $t$ and $\ph$, such that 
\begin{equation}\label{eq:Holderint}
\biggl|\int_M \ph \der (\mu(t) - \mu(s))\biggr| \leq C(\| \nabla \ph \|_\infty + \| \nabla^2 \ph \|_\infty)|t-s|.
\end{equation}
Now, let $1 > \ep > 0$. Let $s \in [0,T]$ and $t \in [0,T]$ such that $|t-s| \leq \ep^2$. Since $\mu(s),\mu(t) \in \Wc_1(M)$, by the Monge--Kantorovic duality there exists $f_{s,t}$, $1$--Lipschitz, such that 
\begin{equation}
\Wc_1(\mu(t), \mu(s)) \leq \biggl|\int_M f_{s,t} \der (\mu(t) - \mu(s))\biggr| + \ep.
\end{equation}
Hence, Corollary~\ref{co:smoothapprox} yields that there exist $K>0$ and a $K$--Lipschitz $\ph_{s,t} \in \Cc^{\infty}_c(M)$, such that 
\begin{equation}
\begin{aligned}
&\|f_{s,t} - \ph_{s,t}\|_{\infty} \leq \ep, \\
&\|\nabla^2 \ph_{s,t} \|_{\infty} \leq \frac{K}{\ep}. \\
\end{aligned}
\end{equation}
Now, (\ref{eq:Holderint}) becomes
\begin{equation}
\begin{aligned}
\Wc_1(\mu(t), \mu(s)) \leq (2 + (1 + K)C)\ep.
\end{aligned}
\end{equation}
Therefore, $t \mapsto \mu(t)$ is $\tfrac12$--H\"older--continuous.
\end{proof}
%\att{Counterexample for non--bounded curvature?? Space with positive curvature going to $\infty$, because volume of balls goes to zero, such that we cannot convolute}

\begin{pr}\label{pr:WpHolder}
Let $B \in \Rz^*_+$ and $V$ be a probability vector--field on $\Wc_p(M)$ with $p$--th moments uniformly bounded by $B$. Let $\mu: [0,T] \to \Wc_p(M)$ be a narrowly continuous curve of probability measures that solves (\ref{eq:MDEd}) for $V$ with initial value $\mu_0$. Then,  $\mu(t)$ is uniformly continuous in $\Wc_p(M)$. Furthermore, let $\ep > 0$ and $\Rf(\ep)>0$ be the minimal value such that for all $t \in [0,T], R \geq \Rf(\ep)$,
\begin{equation}
\int_{M \setminus B(x_0,R)} d(x_0,x)^p \der \mu(t) \leq \ep.
\end{equation}
Then, the modulus of continuity $\omega$ of $t \mapsto \mu(t)$ is the inverse of 
\begin{equation}
	\Om: \ep \mapsto \biggl(\frac{\ep}{\Rf(\ep)^{(p-1)}}\biggr)^2.
\end{equation}
Finally, if $M$ is of bounded geometry, $\mu$ is H\"older--continuous in $\Wc_q(M)$ for $q < p$ with exponent $\frac{p-q}{2(p-1)}$.
\end{pr}

\begin{proof}
\textbf{Step 1: Continuity}
\medskip

\noindent First, we prove by induction that $\mu(t)$ is continuous in $\Wc_k(M)$ for all $k \in \Nz$ with $1 \leq k \leq p$.
\begin{itemize}

\item For $k = 1$, the property follows directly from Proposition~\ref{pr:W1Holder} if $(M,g)$ is of bounded geometry. Otherwise, let $\ep >0$ and $s,t \in [0,T]$. Then, (\ref{eq:MDEd}) yields that there exists a constant $C_1$ independent of $t$ and $s$, such that for all $\ph \in \Cc^{\infty}_c(M)$, 
$$
	\begin{aligned}
		\biggl|\int_M\ph\der(\mu(t)-\mu(s))\biggr|\leq \int_s^t\int_M C_1(|\nabla\ph(x)|+|\nabla^2\ph(x)|)\der\mu(u)(x)\,du.
	\end{aligned}
$$
Choose $R>0$ such that
$$
	\int_{M\setminus B(x_0,R)}d(x,x_0)\der\mu(t)\leq\ep\quad\text{and}\quad
	\int_{M\setminus B(x_0,R)}d(x,x_0)\der\mu(s)\leq\ep.
$$
Let $0\leq\ph\in C^\infty_c(M)$ be such that $\ph(x)=d(x_0,x)$ in $B(x_0,R) \setminus B(x_0,1)$, and
$\|\nabla\ph\|_\infty\leq 1$, $\|\nabla^2\ph\|_\infty\leq 1$ and $d(x_0,x) \leq \ph(x) \leq 1$ for all $x \in B(x_0,1)$. Such a function
exists for all $R>0$. Then,
\begin{equation}
	\begin{aligned}
		\Mc_1(\mu(i)) &= \int_{M}d(x,x_0)\der\mu(i) \geq \int_{M}\ph\der\mu(i) - 1\\
			&\geq \int_{B(x_0,R)}d(x_0,x)\der\mu(i) - 1 \geq\Mc_1(\mu(i))-\eps -1.
	\end{aligned}
\end{equation}
for $i \in \{s,t\}$. Now,
$$
	\biggl|\int_M\ph\der(\mu(t)-\mu(0))\biggr|\leq 2C_1|t - s|.
$$
Therefore,
$$
	\Mc_1(\mu(t)-\mu(0))\leq 2C_1|t|+2\ep + 2\leq 2C_1T+2\ep +2,
$$
and the $\bigl(\Mc_1(\mu(t))\bigr)_{t}$ are uniformly bounded.

\item Suppose we have proven for some $1 \leq k \leq p$ that there exists a constant $K_k$, such that for all $t \in [0,T]$, $\Mc_k(\mu(t)) \leq K_k$. 

Let $\ep >0$ and $s,t \in [0,T]$. Again, choose $R>0$ such that
$$
\begin{aligned}
	&\int_{M\setminus B(x_0,R)}d(x,x_0)^{k+1}\der\mu(s)\leq\ep \\
	&\int_{M\setminus B(x_0,R)}d(x,x_0)^{k+1}\der\mu(t)\leq\ep.
\end{aligned}
$$ 

Now, let $0\leq\ph\in C^2_c(M)$ be such that $\ph(x)=d(x_0,x)^{k+1}$ for $x \in B(x_0,R)$, $\ph(x) \leq d(x_0,x)^{k+1}$ for $x \in M \setminus B(x_0,R)$ and
$\|\nabla \ph\|_\infty\leq R^{k}$ and $\|\nabla^2 \ph\|_\infty\leq R^{k - 1}$. Such a function exists for all $R>0$. Note that we do not have to regularize the function around $x_0$, since $k+1 \geq 2$.

Then, for all $u \in [s,t]$,
\begin{equation}
	\begin{aligned}
		&\int_M C_1(|\nabla\ph(x)|+|\nabla^2\ph(x)|)\der\mu(u)(x)\\ &\qquad \leq C \bigl(\Mc_{k}(\mu(u)) + \Mc_{k-1}(\mu(u))\bigr)\leq C(K_k + K_{k-1}), 
	\end{aligned}
\end{equation}
for some constant $C>0$, which does not depend on $\ep$, $N$ or $u$. Thus, we conclude that 
$$
	\Mc_{k+1}(\mu(t)-\mu(s))\leq C(K_k + K_{k-1})|t-s|+2\ep,
$$
and, since $\ep$ was chosen arbitrary, 
$$
	\Mc_{k+1}(\mu(t)-\mu(s))\leq C(K_k + K_{k-1})|t-s|.
$$
The narrow continuity of $t \mapsto \mu(t)$ in combination with continuity of $t \mapsto \Mc_{k+1}(\mu(t))$ implies continuity in $\Wc_{k+1}(M)$. In particular,
the $\bigl(\Mc_{k+1}(\mu(t))\bigr)_{t}$ are uniformly bounded. 

In particular, the initialisation and the first induction step yield that $t \mapsto \mu(t)$ is continuous in $\Wc_2(M)$, it is continuous in $\Wc_k(M)$ for $1\leq k <2$ with Proposition~\ref{pr:wassersteinidentity}.

 \item Finally, suppose the result has been proven for all $k \in \Nz$ with $1 \leq k \leq p$. Then, in particular, $t \mapsto \mu(t)$ is continuous in $\Wc_{p-1}(M)$. Hence, we just repeat the above argument to obtain the continuity in $\Wc_p(M)$.
\end{itemize}
\medskip
\textbf{Step 2: Uniform continuity}
\medskip

Next, we prove that $\mu: [0,T] \to \Wc_p(M)$ is uniformly continuous in $\Wc_p(M)$. 

\noindent Note that $[0,T]$ is compact and $\mu$ is continuous with respect to the $\Wc_p$--distance. 
Therefore, $(\mu(t))_{t \in [0,T]}$ is a compact subset of $\Wc_p(M)$. Hence, by Lemma~\ref{le:compacttight}, it is tight in the $p$--th moment. Let $\ep >0$ and $t,s \in [0,T]$. Then, one readily shows that there exists a constant $C>0$, such that
$$
\Wc_p(\mu(t),\mu(s)) \leq 2\ep + \Rf(\ep)^{p-1}\Wc_1(\mu(t),\mu(s)) \leq 2\ep + \Rf(\ep)^{p-1}C\sqrt{|t-s|}
$$ 
by Lemma~\ref{pr:W1Holder}. Hence, if 
$$
|t-s| \leq \frac{\ep^2}{C^2\Rf(\ep)^{2(p-1)}},
$$
we deduce $\Wc_p(\mu(t),\mu(s)) \leq 3\ep$ and, thus, the uniform continuity.

\medskip
\noindent\textbf{Step 3: H\"older--continuity in $\Wc_q(M)$ for $q < p$}
\medskip

Finally, suppose $M$ is of bounded geometry. Let $q < p$ and let $B > 0$ such that for all $t \in [0,T]$, $\Mc_p(\mu(t)) \leq B^p$. Then, by Lemma~\ref{le:q-tight}, for $R_{\ep} = \bigl( B^p\frac{p}{p-q}\bigr)^{\frac{1}{p-q}}\ep^{\frac{1}{q-p}}$,  we have $\int_{M \setminus B(x_0,R_{\ep})} d(x_0,x)^q \der \mu(t) \leq \ep$. By Corollary~\ref{co:smoothapprox}, there exists $C>0$, such that
$$
\Wc_q(\mu(t),\mu(s)) \leq 2\ep + R_{\ep}^{q-1}\Wc_1(\mu(t),\mu(s)) \leq 2\ep + C\ep^{\frac{q-1}{q-p}}\sqrt{|t-s|}.
$$ 
Therefore, if $|t-s| \leq \ep^{\frac{2(p-1)}{p-q}}$, $\Wc_q(\mu(t),\mu(s)) \leq (2 + C)\ep$. Therefore, $\mu(t)$ is Hölder--continuous in $\Wc_q(M)$.

\end{proof}

\begin{re}
	The above proof can be made more concise by using a Grönwall argument instead of an induction proof. However, the above proof also shows that the divergence of the 
	$p$--th moments is polynomial. This result cannot be recovered using Grönwall.
\end{re}

\subsection{Motivation}\label{subs:motiv}

Our aim is to define a new differential structure on the Wasserstein space $\Wc_p(M)$ --- i.e. define what it means for a curve of measures to be smooth and corresponding tangent spaces---such that at least some reasonable diffusion processes are sufficiently regular.

\begin{defi}[Averaged flow operators]\label{de:avflowop}
Let $(M,g)$ be a Riemannian manifold. We define the \emph{averaged flow dispersion operators associated to $V$} by
\begin{equation}
\begin{aligned}
f^{V[\mu]}_t: \Mc_b(M) &\to \Mc_b(M) \\
\nu &\mapsto \int_{\Xct(M)} (\rh^{X-\ol{V}[\mu]}_{\sqrt t})_*\nu \der V[\mu](X)
\end{aligned}
\end{equation}
and
\begin{equation}
\begin{aligned}
f^V_t: \Mc_b(M) &\to \Mc_b(M) \\
\nu &\mapsto \int_{\Xct(M)} (\rh^{X-\ol{V}[\nu]}_{\sqrt t})_*\nu \der V[\nu](X)
\end{aligned}
\end{equation}
for all $\mu \in \Wc_p(M)$, $t>0$. Furthermore, define the \emph{expectational flow operator}
\begin{equation}
\begin{aligned}
e^{V[\mu]}_t: \Mc_b(M) &\to \Mc_b(M) \\
\nu &\mapsto (\rh^{\ol{V}[\mu]}_t)_* \nu
\end{aligned}
\end{equation}
and
\begin{equation}
\begin{aligned}
e^V_t: \Mc_b(M) &\to \Mc_b(M) \\
\nu &\mapsto (\rh^{\ol{V}[\nu]}_t)_* \nu
\end{aligned}
\end{equation}
for all $\mu \in \Wc_p(M)$, $t>0$.
\end{defi}

Now the idea is the following. Let $t\in [0,T]$ and $n\in\Nz$. Let $\ta = \frac{T}{n}$ and define for all $k \in \llbracket 0,n-1 \rrbracket$,
$$
\mu^{\ta}(k\ta) = (e^V_{\ta}\circ f^V_{\ta})^k(\mu_0)
$$
and define $\mu^{\ta}(t)$ by convex interpolation for any other $t \in [\,0,T]\,$. Now, if for $\ta \to 0$, $\mu^{\ta}$ converges pointwise to some $\nu: [0,T] \to \Pc(M)$, with respect to the $\Wc_p$--norm, then, we say that $\nu$ satisfies ($\ref{eq:MDE}$) and is smooth.

\begin{re}
Note that $f^V_t$ is well--defined. Indeed, take $\nu \in \Mc_b(M)$ and $f\in \Cc_b(M)$. Then, 
\begin{equation}
(x,X) \mapsto f \circ \rh^{X-\ol{V}[\nu]}_{\sqrt t}(x)
\end{equation}
is continuous and bounded, hence, 
\begin{equation}
\begin{aligned}
&\int_{M \times \Xct(M)} f \circ \rh^{X-\ol{V}[\nu]}_{\sqrt t}(x)  \der V[\nu]\otimes \nu(X, x) \\
&\qquad = \int_{\Xct(M)} \int_M f \circ \rh^{X-\ol{V}[\nu]}_{\sqrt t}(x) \der \nu(x) \der V[\nu](X).
\end{aligned}
\end{equation}
Thus, $\int_{\Xct(M)} \bigl(\rh^{X-\ol{V}[\nu]}_{\sqrt t}\bigr)_*\nu \der V[\nu](X)$, defined by
\begin{equation}
\begin{aligned}
&\int_M f(x) \der \biggl(\int_{\Xct(M)} \bigl(\rh^{X-\ol{V}[\nu]}_{\sqrt t}\bigr)_*\nu \der V[\nu](X)(x)\biggr)\\
 &\qquad = \int_{\Xct(M)} \int_M f \circ \rh^{X-\ol{V}[\nu]}_{\sqrt t}(x) \der \nu(x) \der V[\nu](X), 
\end{aligned}
\end{equation}
is well-defined.
\end{re}

\begin{re}
For motivation purposes, we keep notation short here. In the further course of this work, we shall introduce general subdivisions of $[0,T]$ with maximal step length $\ta$, define a $\mu^{\ta}$ for each of these subdivisions. Naturally we would want all these approximations to converge to a common limit for $\ta \to 0$. We shall see that this is indeed the case with suitable assumptions on $V$ and $(M,g)$.
\end{re}

This implicit definition of a solution of ($\ref{eq:MDE}$) might be intuitively easy to grasp, but is hard to handle from a mathematical point of view. To obtain the explicit integral equation (\ref{eq:MDEd}) from the idea outlined above, we argue as follows.

Formally, we want to get from $t$ to $\der t$ by 
\begin{equation}
\mu(t + \der t) =\int_{\Xct(M)} \bigl(\rh^{\ol{X}[\mu(t)]}_{\sqrt {\der t}}\bigr)_* \bigl(\rh^{\ol{V}[\mu(t)]}_{\der t}\bigr)_*\mu(t) \der V[\mu(t)](X)
\end{equation} 
or equivalently 
\begin{equation}
\begin{aligned}
\frac{\mu(t + \der t) - \mu(t)}{\der t} &= \int_{\Xct(M)} \frac{\bigl(\rh^{\ol{X}[\mu(t)]}_{\sqrt {\der t}}\bigr)_* \bigl(\rh^{\ol{V}[\mu(t)]}_{\der t}\bigr)_*\mu(t) - \mu(t)}{\der t} \der V[\mu(t)](X) \\
&= \int_{\Xct(M)} \frac{\bigl(\rh^{\ol{X}[\mu(t)]}_{\sqrt{ \der t}}\bigr)_* \bigl(\rh^{\ol{V}[\mu(t)]}_{\der t}\bigr)_*\mu(t) - \bigl(\rh^{\ol{V}[\mu(t)]}_{\der t}\bigr)_*\mu(t)}{\der t} \\&\qquad+ \frac{\bigl(\rh^{\ol{V}[\mu(t)]}_{\der t}\bigr)_*\mu(t) - \mu(t)}{\der t} \der V[\mu(t)](X).
\end{aligned}
\end{equation} 
Now, integration over $M$ gives 
\begin{equation}
\begin{aligned}
\int_M &\ph \der\biggl(\frac{\mu(t + \der t) - \mu(t)}{\der t}\biggr) \\&= 
\int_M \ph \der \left(\int_{\Xct(M)}\!\!\!\!\!\!\!\!\!\!\!\!\! \frac{\bigl(\rh^{\ol{X}[\mu(t)]}_{\sqrt {\der t}}\bigr)_* \bigl(\rh^{\ol{V}[\mu(t)]}_{\der t}\bigr)_*\mu(t) - \bigl(\rh^{\ol{V}[\mu(t)]}_{\der t}\bigr)_*\mu(t)}{\der t} \der V[\mu(t)](X)\right)\\ &\qquad\qquad+ \int_M \ph \der \biggl(\frac{\bigl(\rh^{\ol{V}[\mu(t)]}_{\der t}\bigr)_*\mu(t) - \mu(t)}{\der t}\biggr).
\end{aligned} 
\end{equation} 
By assuming that we can use Fubini---of course we would have to properly justify this---we obtain 
\begin{equation}
\begin{aligned}
\int_M \ph \der\biggl(\frac{\mu(t + \der t) - \mu(t)}{\der t}\biggr) &= 
\int_M \frac{1}{\der t} \int_{\Xct(M)} \ph \circ \rh^{\ol{X}[\mu(t)]}_{\sqrt {\der t}} \\&\qquad\qquad- \ph \der V[\mu(t)](X) \der \mu(t) \\ &+ \frac{\der}{\der t}\int_M \ph \circ \rh^{\ol{V}[\mu(t)]}_{\der t} - \ph \der \mu(t).
\end{aligned} 
\end{equation} 
With this equation it becomes clear why we have set the time step in the first integral to $\sqrt{\der t}$ instead of $t$: the expression $\int_{\Xct(M)} \ph \circ \rh^{\ol{X}[\mu(t)]}_{\sqrt {\der t}} - \ph \der V[\mu(t)](X)$ looks like the discrete version of some elliptic operator of second order. In fact, this is no surprise, since these operators usually appear when it comes to model diffusion phenomena. Of course, many operators of second order admit this particular discrete version. However, here the most logical choice is the one in Definition~\ref{def:MDE}. Making the expression in the integral on the left-hand side continuous is straightforward, the second integral on the right-hand side would rather become $\int_M \Lc_{\ol{V}[\mu(t)]}(\ph) \der \mu(t)$, see Lemma~\ref{re:altdef}.

\begin{re}
Finally, note that if we had chosen to take time steps of length $\der t$ instead of $\sqrt{\der t}$ under pushforwards of $\rh^{\ol{X}[\mu(t)]}$, the first term of the RHS in Definition~\ref{def:MDE} would have become $\int_M \int_{\Xct(M)} \Lc_{\ol{X}[\mu(t)]} (\ph) \der V[\mu(t)] \der \mu(t)$. Hence, Definition~\ref{def:MDE} would become 
\begin{equation}
\frac{\der}{\der t} \int_M \ph \der \mu(t) = \int_M \int_{\Xct(M)} \Lc_{X}\ph \der V[\mu(t)](X) \der \mu(t),
\end{equation}
which is essentially (\ref{eq:piccMDE}) defined in \cite{piccoli2019measure}. 
\end{re}

\subsubsection{Examples}

An important indicator that the integral equation we have obtained is indeed interesting comes from the following basic examples.
\begin{ex}\label{ex:Wienerpro}
If we take $M = \Rz$ and $V(\mu) = \eta$ for all $\mu \in \Pc(M)$, where $\eta$ gives weight $\frac{1}{2}$ to the vector--fields $X_1$ and $X_{-1}$ that are constant $1$ and $-1$, respectively, then the Wiener process, i.e., $\mu(t) = \frac{1}{\sqrt{2\pi t}}e^{-\frac{x^2}{2t}} \der x$ is a solution of (\ref{eq:MDE}) with $\mu_0 = \de_0$. This is a good sign, as the Wiener process is obtained as a limit of random walks with step-size going to $0$.

Analogously if we set $\eta$ to be the measure that gives weight $- \frac{1}{2}$ to the vector--fields $X_1$ and $X_{-1}$, then for any $T$, the backward Wiener process with $\mu(t) = \de_0$ for $t > T$ is a solution of (\ref{eq:MDE}) with $\mu_0 = f_T(x)dx$ with $f_t: x \mapsto \frac{1}{\sqrt{2\pi t}} \exp{-\frac{x^2}{2t}}$.
\end{ex}

\begin{ex}
Now, with the same notation as above, but with $\eta$ giving equal weights to $X_{-1}$ and $X_2$---i.e., with non-zero pointwise expectation---we get the Wiener process with a right-shift of constant speed $1$ as a solution.
\end{ex}
\begin{ex}
Let us now give a more sophisticated example by choosing $M = \Rz^2$. Otherwise we stick to the notation of the previous examples, this time giving equal weights of $\frac{1}{3}$ to 
the roots of $X^3 -1$, which we denote by $j_1$, $j_2$ and $j_3$. For $\mu_0 = \de_0$ we expect to get some Gaussian diffusion with a rotational symmetry, i.e., whose covariance-matrix is a multiple of the identity matrix. Indeed, in this case one readily checks that the operator $\frac{1}{3}(\frac{\d^2}{\d j_1^2}+\frac{\d^2}{\d j_2^2}+\frac{\d^2}{\d j_3^2}) = \frac{1}{2} \De$, where $\De$ denotes the Laplace operator---which gives a solution of the desired form. The detailed computation goes 
\begin{equation}
\begin{aligned}
\frac{\d^2}{\d j_1^2} &= \biggl(-\frac{1}{2}\frac{\d}{\d x} + \frac{\sqrt 3}{2}\frac{\d}{\d y}\biggr)\biggl(-\frac{1}{2}\frac{\d}{\d x} + \frac{\sqrt 3}{2}\frac{\d}{\d y}\biggr) \\ &= \frac{1}{4}\frac{\d^2}{\d x^2} + \frac{3}{4}\frac{\d^2}{\d y^2} - \frac{\sqrt 3 }{4}\biggl(\frac{\d^2}{\d y\d x} + \frac{\d^2}{\d x\d y}\biggr) \\ \frac{\d^2}{\d j_2^2} &= \biggl(-\frac{1}{2}\frac{\d}{\d x} - \frac{\sqrt 3}{2}\frac{\d}{\d y}\biggr)\biggl(-\frac{1}{2}\frac{\d}{\d x} - \frac{\sqrt 3}{2}\frac{\d}{\d y}\biggr) \\ &= \frac{1}{4}\frac{\d^2}{\d x^2} + \frac{3}{4}\frac{\d^2}{\d y^2} + \frac{\sqrt 3 }{4}\biggl(\frac{\d^2}{\d y\d x} + \frac{\d^2}{\d x\d y}\biggr) \\ \frac{\d^2}{\d j_3^2} &= \frac{\d^2}{\d x^2},
\end{aligned}
\end{equation}
which yields 
\begin{equation}
\begin{aligned}
\frac{1}{3}\biggl(\frac{\d^2}{\d j_1^2}+\frac{\d^2}{\d j_2^2}+\frac{\d^2}{\d j_3^2}\biggr) = \frac{1}{3}\biggl(\frac{3}{2}\frac{\d^2}{\d x^2} +\frac{3}{2}\frac{\d^2}{\d y^2}\biggr) = \frac{1}{2}\biggl(\frac{\d^2}{\d x^2} + \frac{\d^2}{\d y^2}\biggr).
\end{aligned}
\end{equation}
\end{ex}

\subsection{The differential operator $\square_{\mu}^V$}\label{sbs:square}

Let us now study some basic properties of the differential operator $\square_{\mu}^V$. First, we now make a stronger assumption on $V$, namely $(\pi_x)_* V[\mu] \in \Wc_2(M)$ for all $x \in M$ and $\mu \in \Pc(M)$ --- or directly $V[\mu] \in \Wc_2(\Xc^{2,\infty}(M))$. Then $(\pi_x)_* V[\mu]$ is in particular an integrable signed Borel--measure. Let $\nu_+$ and $\nu_-$ be the unique pair of positive $2$--integrable Borel--measures such that $(\pi_x)_* V[\mu] = \nu_+ - \nu_-$. Furthermore, choose a smooth frame $(X_1,\dots,X_n)$ on $M$ and take $X \in \Xct(M)$. Then for $X = X_i \frac{\d}{\d x_i}\in \Xct(M)$, in Einstein--convention,
\begin{equation}\label{eq:LDclassic}
\Lc_X^2 = X_i\frac{\d X_j}{\d x_i}\frac{\d }{\d x_j} + X_i X_j\frac{\d^2 }{\d x_i\d x_j}.
\end{equation}
We are particularly interested in the term of second order of $\square_{\mu(t)}^V$. Note that the second order term $ X_i X_j\frac{\d^2 }{\d x_i\d x_j}$ is tensorial in $X$, so it suffices to study it pointwise on tangent spaces. Since $(\pi_x)_* V[\mu] \in \Wc_2(M)$, for all $\ph \in \Cc^{\infty}_c(M)$, $\int_{T_x M} X_i X_j\frac{\d^2 \ph}{\d x_i\d x_j}(x) \der(\pi_x)_* V[\mu](X)$ converges, and
\begin{equation}\label{eq:ellterm}
\begin{aligned}
\int_{T_x M} &X_i X_j\frac{\d^2 \ph}{\d x_i\d x_j}(x) \der(\pi_x)_* V[\mu](X) \\&=\int_{T_x M} X_i X_j\frac{\d^2 \ph}{\d x_i\d x_j}(x) \der \nu_+(X) \\ &\qquad-\int_{T_x M} X_i X_j\frac{\d^2 \ph}{\d x_i\d x_j}(x) \der \nu_-(X) \\
&= \frac{\d^2 \ph}{\d x_i\d x_j}(x) \biggl(\int_{T_x M} X_i X_j \der \nu_+(X) -\int_{T_x M} X_i X_j \der \nu_-(X) \biggr),
\end{aligned} 
\end{equation}
where $\pi_i$ denotes the projection on $x_i$. Indeed, for all $i,j$ and all $X\in T_x M$, $X_i X_j \leq \|X\|_2^2$ and by assumption $X \mapsto \|X\|_2^2$ is $\nu_+$-- and $\nu_-$--integrable. Indeed, $\int_{T_x M} \|X\|_2^2 \der \nu_+(X)$ is just the $2$-moment of $\nu_+(X)$. Now, for all $i,j$, $\int_{T_x M} X_i X_j \der \nu_+(X) = \langle\pi_i, \pi_j\rangle_{L^2_{\nu_+}}$. Hence, $\bigl( \int_{T_x M} X_i X_j \der \nu_+(X) \bigr)_{i,j}$ is a Gram--ma\-trix and, thus, symmetric and semi-elliptic. With the same reasoning, we see that $\bigl( \int_{T_x M} X_i X_j \der \nu_-(X) \bigr)_{i,j}$ is symmetric and semi-elliptic. This yields a decomposition of $\square_{\mu}^V = \De_{\mu(t)}^{+,V} - \De_{\mu(t)}^{-,V}$ into two point--wise semi-elliptic operators, i.e., equation (\ref{eq:MDE}) is of the form
\begin{equation}
\frac{\der}{\der t}\int_M \ph \der \mu(t) = \int_M \De_{\mu(t)}^{+,V} \ph \der \mu(t) - \int_M \De_{\mu(t)}^{-,V} \ph \der \mu(t).
\end{equation}

\begin{re}
If $V$ takes values in the positive Borel--measures or, in particular, some $\Wc_p(\Xc^{2,\infty}(M))$, $\De_{\mu}^{-,V}$ vanishes for all $\mu$ and equation (\ref{eq:MDEd}) resembles a semi-parabolic equation. We shall see that, in that case, the assumptions on $\square_{\mu}^V$ that are necessary to guarantee the existence of a solution of (\ref{eq:MDEd}) are very similar to those we know from the classical solution theory of semi-parabolic PDEs.
\end{re}

\begin{defi}\label{de:coeffs}
	Let $p>2$, $v \in \Wc_p(\Xc^{2,\infty}(M))$, $i,j \in \llbracket 0,n-1 \rrbracket$ and define for all $x \in M$
	\begin{equation}
		\begin{aligned}
			a_{i,j}^v(x) &\coloneqq \int_{T_x M} X_i(x) X_j(x) \der v(X)\\
			b_j^v(x) &\coloneqq \int_{T_x M} X_i(x) \frac{\d X_j}{\d x_i}(x) \der v(X).
		\end{aligned}
	\end{equation}
\end{defi}

\begin{theorem}\label{th:ellipticapprox}
	Let $p\geq 2$, $(M,g)$ be a Riemannian manifold of bounded geometry and $V$ be a probability vector--field, such that there exists $B > 0$, such that for all $\mu \in \Wc_p(M)$, $\Mc_p(V[\mu]) \leq B$. Let $\ep > 0$. Then, there exists a probability vector--field  $W$ such that
	\begin{itemize}
		\item $\Wc_p^{\infty}(V,W) \leq \ep^{\frac{1}{p}}$.
		\item $\square^W_{\mu}$ is $\Kc\ep$--elliptic for all $\mu \in \Wc_p(M)$, for a constant $\Kc>0$ depending only on $B$ and $M$.
		\item for all $\mu,\nu \in \Wc_p(M)$, $\Wc_p\bigl(W[\mu], W[\nu]\bigr) \leq \Wc_p\bigl(V[\mu], V[\nu]\bigr)$.
		\item $W$ has uniformly bounded $p$--th moments.
	\end{itemize}
	 
\end{theorem}
\begin{proof}
	Let $N\in\Nz$ such that there exists a Riemannian embedding $i: M \to \Rz^N$. Such an $N$ always exists by Nash's embedding theorem. Let $(x_j)_{j\in \llbracket 0, N-1 \rrbracket}$ be the canonical basis of $\Rz^N$. Define for all $j\in \llbracket 0, N-1 \rrbracket$ the vector--field $X_j$ by 
	\begin{equation}
		\begin{aligned}
			&X_j: M \to TM\\
			&x \mapsto i^{-1}\bigl(\pi_{i(T_x M)} x_j\bigr),
		\end{aligned}
	\end{equation}
	where $\pi_{i(T_x M)}$ denotes the orthogonal projection on the space $i(T_x M)$. Then, $\bigl(X_j\bigr)_{j\in \llbracket 0, N-1 \rrbracket}$ is a family of smooth vector--fields and  
	$\bigl(X_j(x)\bigr)_{j\in \llbracket 0, N-1 \rrbracket}$ is a generating system of $T_x M$ for all $x\in M$. By construction, for all $j \in \llbracket 0, N-1 \rrbracket$ and all $x \in M$, $\|X_j(x)\| \leq 1$. Furthermore, since $(M,g)$ is of bounded geoemtry, there exists $C(M) > 0$, such that for all $j \in \llbracket 0, N-1 \rrbracket$, $\|X_j\|_{W^{2,\infty}} \leq C(M)$.

	Now, define $W$ by
	\begin{equation}
		\begin{aligned}
			W: \Wc_p(M) &\to \Wc_p(\Xc^{2,\infty}(M)) \\
			\mu &\mapsto (1 - \ep)V[\mu] + \frac{\ep}{N} \sum_{j = 0}^N \de_{X_j}.
		\end{aligned}
	\end{equation}
	Then, for all $\mu \in \Wc_p$, 
	\begin{equation}
		\begin{aligned}
			\Wc_p(V[\mu], W[\mu]) &\leq \Wc_p(V[\mu], (1 - \ep)V[\mu] + \ep \de_0) \\&\qquad\qquad\qquad\qquad+ \Wc_p((1 - \ep)V[\mu] + \ep \de_0, W[\mu])\\
			&\leq (B^{\frac{1}{p}}+ C(M))\ep^{\frac{1}{p}}.
		\end{aligned}
	\end{equation}
	Finally, it follows immediately from (\ref{eq:ellterm}) that $\square^W_{\mu}$ is $\frac{\ep}{N}$--elliptic for all $\mu \in \Wc_p(M)$.

	The last two points are immediate.
\end{proof}

\subsection{Existence and construction of solutions of (MDE)}

We shall now prove the existence of solutions of (\ref{eq:MDEd}) by a discrete--to--continuous argument, as sketched in Section~\ref{subs:motiv}. This proof requires that $V[\mu]$ has finite $p$-th moment for some fixed $p>2$ and for all $\mu \in \Wc_p(M)$ with respect to the $\|~\|_{W^{2,\infty}}$--norm. Therefore, we have to make a stronger assumption than before, namely $p > 2$.

\begin{theorem}\label{th:exfirst}
Let $T>0$. Let $V$ be a uniformly continuous probability vector--field on $\Wc_p(M)$ for $p > 2$. Assume that $V$ has uniformly bounded $p$--th moments, e.g., there exists a constant $B>0$ such that for all $\mu \in \Wc_p(M)$, $\Mc_p(V[\mu]) < B$. 
\begin{enumerate}
\item Assume furthermore that $(M,g)$ is of bounded geometry and that there exists $R>0$ such that for all $\mu \in \Wc_p(M)$, $V[\mu] \in \Wc_{p+2}(M)$ and $\Mc_{p+2}(V[\mu]) \leq R^{p+2}$. Then, (\ref{eq:MDEd}) has a $\frac{1}{p}$--Hölder--continuous solution in $\Wc_p(M)$ on $[0,T]$. 
\item Assume only that there exists $q < p$ such that $V$ is continuous for the $\Wc_q$--distance. Then, (\ref{eq:MDEd}) has a continuous solution in $\Wc_p(M)$ on $[0,T]$. 

\end{enumerate}

\end{theorem}
 
We will prove Theorem~\ref{th:exfirst} as a corollary of the stronger Theorem~\ref{th:gencentlim}, which is in some sense a generalisation of the well-known Central Limit Theorem. It affirms the existence of solutions to (\ref{eq:MDEd}) in a constructive way. 

\subsubsection{Average Flow Approximation Series (AFAS)}

An explicit solution of (\ref{eq:MDEd}) is obtained as a limit of an explicit Euler--scheme---which we will call Average Flow Approximation Series (AFAS).   

\begin{defi}[Average flow approximation series]\label{de:AFAS}
Let $T>0$, $p \geq 2$, $\mu_0 \in \Wc_p(M)$. Let $\de>0$ and $\Pf(\de)$ be a partition of $[0,T]$ with maximal step--length $\pf(\Pf(\de)) \leq \de$. Let $N \coloneqq |\Pf(\de)|$ and index the elements of $\Pf(\de)$ in increasing order. Define
$$
\mu^{\Pf(\de)}(0) = \mu_0,
$$
and for all $l \in \llbracket 0,N-1 \rrbracket$,
$$
\mu^{\Pf(\de)}(x_{l+1}) = (e^V_{x_{l+1} - x_l}\circ f^V_{x_{l+1} - x_l})(\mu^{\Pf(\de)}(x_l)).
$$
Furthermore, define $\mu^{\Pf(\de)}(t)$ by 
$$
\mu^{\Pf(\de)}(t) = f^V_{2(t - x_l)}(\mu^{\Pf(\de)}(x_l)),
$$
for all $t \in \bigl[x_l, \frac{x_{l+1} + x_l}{2}\bigr]$ and 
$$
\mu^{\Pf(\de)}(t) = \biggl(e^V_{2\bigl(t - \frac{x_{l+1} + x_l}{2}\bigr)}\circ f^V_{x_{l+1} - x_l}\biggr)\bigl(\mu^{\Pf(\de)}(x_l)\bigr),
$$
for all $t \in \bigl[\frac{x_{l+1} + x_l}{2}, x_{l+1}\bigr]$.
\end{defi}
As proposed in Section~\ref{subs:motiv}, we call a solution of (\ref{eq:MDEd}) smooth if it is a limit of AFAS.
\begin{defi}[Smooth solution]\label{de:smoothsol}
 Let $p\geq 2$ and $T>0$. A narrowly continuous curve $\nu: [0,T] \to \Wc_p(M)$ is called a \emph{smooth solution} of (\ref{eq:MDEd}) if it solves (\ref{eq:MDEd}) and there exists a sequence $(\Pf_n)_{n\in\Nz}$ of partitions of $[0,T]$ with maximal step--length $\pf(\Pf_n) \to 0$ such that for all $n\in\Nz$, $\mu^{\Pf_n}$ converges pointwise narrowly to $\nu$.
\end{defi}

\subsubsection{Averaged flow operators---Hölder--regularity in time}

First we show that AFAS are indeed curves in $\Wc_p(M)$.
\begin{lem}\label{le:eflowop}
Let $T > 0$. For all $t,s \in [0,T]$, $e^V_t(\mu) \in \Wc_p(M)$, and
\begin{equation}
\Wc_p(e^V_s(\mu), e^V_t(\mu)) \leq |t-s| \Mc_1(V[\mu]).
\end{equation}
\end{lem}
\begin{proof}
This is an immediate corollary of Lemma~\ref{le:bary}.
\end{proof}

\begin{lem}\label{le:fflowop}
For all $\mu \in \Wc_p(M)$ and all $t,s\in[0,T]$, $f^V_t(\mu) \in \Wc_p(M)$. In particular, for all $N$, $\mu^N(t) \in \Wc_p(M)$. Furthermore, we have
\begin{equation}
\Wc_p(f^V_s(\mu),f^V_t(\mu)) \leq \sqrt{|t-s|}\bigl(\|\ol{V}[\mu]\|_{\infty} + \Mc_p(V[\mu])^{\frac{1}{p}}\bigr).
\end{equation}
\end{lem}
\begin{proof}
Let $X \in \Xct(M)$ and $\mu \in \Wc_p(M)$. To simplify notation, we shall write $\ol{X}[\mu] \coloneqq X - \ol{V}[\mu]$, and we shall denote the transport plan $(\id, \rh^X_{\sqrt{t}})_* \mu$ just by $\rh^X_{\sqrt{t}}$. Then, $\int_{\Xct(M)} \rh^{\ol{X}[\mu]}_{\sqrt{t}} \der V[\mu](X)$ is a transport plan from $\mu$ to $f^V_t(\mu)$. 

Now, we use this particular transport plan to give an upper bound of the $p$--Wasserstein distance. First, we write down the definitions formally without justifying integrability and assuming that we can use the
Theorem of Fubini--Tonelli. We have
\begin{equation}
\begin{aligned}
\Wc_p(\mu, f^V_t(\mu))^p &\leq \int_{M \times M} d(x,y)^p \der\biggl( \int_{\Xct(M)} \rh^{\ol{X}[\mu]}_{\sqrt{t}} \der V[\mu](X) \biggr)(x,y) \\
&= \int_0^{\infty} \biggl( \int_{\Xct(M)} \rh^{\ol{X}[\mu]}_{\sqrt{t}} \der V[\mu](X) \biggr) (\{ d(x,y)^p > \la \}) \der \la \\
&= \int_0^{\infty} \int_{\Xct(M)} \rh^{\ol{X}[\mu]}_{\sqrt{t}} (\{ d(x,y)^p > \la \}) \der V[\mu](X)  \der \la \\
&= \int_{\Xct(M)} \int_0^{\infty}  \rh^{\ol{X}[\mu]}_{\sqrt{t}} (\{ d(x,y)^p > \la \}) \der \la \der V[\mu](X)\\
&= \int_{\Xct(M)} \int_{M}  d(x, \rh^{\ol{X}[\mu]}_{\sqrt{t}}(x))^p \der \mu(x) \der V[\mu](X)\\
&= \int_{M} \int_{\Xct(M)}  d(x, \rh^{\ol{X}[\mu]}_{\sqrt{t}}(x))^p \der V[\mu](X) \der \mu(x) .
\end{aligned}
\end{equation}
Now, since $V[\mu]$ has support in $W^{1,\infty}(M)$,  for all $x \in M$, 
\begin{equation}
d\bigl(x, \rh^{\ol{X}[\mu]}_{\sqrt{t}}(x)\bigr) \leq \int_0 ^{\sqrt{t}} | \ol{X}[\mu](\rh^X_t(x))|_g \der t \leq \sqrt{t}(\|\ol{V}[\mu]\|_{\infty} + \|X\|_{\infty}).
\end{equation}
Hence, since $V[\mu]$ has finite $p$--th moment, we have 
\begin{equation}
\begin{aligned}
\int_{\Xct(M)}  &d(x, \rh^X_{\sqrt{t}}(x))^p \der V[\mu](X) \\&\leq (\sqrt{t})^p \int_{\Xct(M)} (\|\ol{V}[\mu]\|_{\infty} + \|X\|_{\infty})^p \der V[\mu](X) < \infty,
\end{aligned}
\end{equation}
and Fubini--Tonelli yields that all the steps of the computation above are justified. Finally,
\begin{equation}
\Wc_p(\mu, f^V_t(\mu)) \leq \sqrt{t}\biggl(\int_{\Xct(M)} (\|\ol{V}[\mu]\|_{\infty} + \|X\|_{\infty})^p \der V[\mu](X)\biggr)^{\frac{1}{p}}.
\end{equation}
Analogously, we prove that for all $t,s \in [0,T]$,
\begin{equation} 
\begin{aligned}
\Wc_p(f^V_s(\mu), f^V_t(\mu)) &\leq |\sqrt{t} - \sqrt{s}|\biggl(\int_{\Xct(M)} (\|\ol{V}[\mu]\|_{\infty} + \|X\|_{\infty})^p \der V[\mu](X)\biggr)^{\frac{1}{p}}\\ 
&\leq \sqrt{|t-s|}\biggl(\int_{\Xct(M)} (\|\ol{V}[\mu]\|_{\infty} + \|X\|_{\infty})^p \der V[\mu](X)\biggr)^{\frac{1}{p}} \\
&\leq \sqrt{|t-s|}\bigl(\|\ol{V}[\mu]\|_{\infty} + \Mc_p(V[\mu])^{\frac{1}{p}}\bigr) \\
\end{aligned}
\end{equation}

Now using Lemma~\ref{le:eflowop} it follows by immediate induction that the $\mu^N$ are all curves in $\Wc_p(M)$.
\end{proof}

\subsubsection{Averaged flow operators---Lipschitz measure--dependence}

To complete the preparations for our first main result, we enlarge our repertoire of estimates concerning AFAS by some sharper and more technical results than Lemma~\ref{le:fflowop}. These will allow us to control the speed of divergence of two AFAS with different initial values and different step lengths. However, 
stronger assumptions on the vector--field $V$ and the manifold $(M,g)$ are required for those results.
\begin{lem}\label{le:fflowop2}
	Assume $(M,g) = (\Rz^n, g^{eucl})$. Let $V$ be a probability vector--field. Assume that there exists $R > 0$ such that for all $\mu \in \Wc_p(M)$, $V[\mu] \in \Wc_{p+2}(\Xc^{2,\infty}(M))$ and $\Mc_{p+2}(V[\mu]) \leq R^{p+2}$.  Let $t > 0$ and $\mu, \nu, v_1, v_2 \in \Wc_p(M)$. Then, if $t\leq1$, there exists a constant $\Kc > 0$ depending only on $p$, such that
	\begin{equation}\label{eq:estfrandVp}
		\begin{aligned}
			&\Wc_p(f^{V[v_1]}_{t^2}(\mu), f^{V[v_2]}_{t^2}(\nu))^p\\
			&\qquad \leq \Wc_p(\mu, \nu)^p + 2^{p-2}p\biggl(9(p-1)\Wc_p(\mu,\nu)^{p-2}\Wc_p(V[v_1],V[v_2])^2\\
			&\qquad + pR^2\Wc_p(\mu,\nu)^p + 6R\Wc_p(\mu,\nu)^{p-1}\Wc_p(V[v_1],V[v_2]) \biggr)t^2\\
			&\qquad + \Kc\max(1, R^{p+2})\biggl( t^{p+2} + \Wc_p(V[v_1],V[v_2])^2 t^p + \Wc_p(V[v_1],V[v_2]) t^{p+1}\\ &\qquad + \Wc_p(\mu, \nu)^{p - 1}t^3 + \Wc_p(\mu, \nu)^{p-2}t^4 + \Wc_p(\mu, \nu)^2t^p + \Wc_p(\mu, \nu)t^{p+1}\\
			&\qquad + \Wc_p(\mu,\nu)^{p-2}\Wc_p(V[v_1],V[v_2])t^3 + \Wc_p(\mu, \nu)\Wc_p(V[v_1],V[v_2])t^p\biggr).
		\end{aligned}
	\end{equation}
	and,
	\begin{equation}\label{eq:esterandVp}
		\begin{aligned}
			&\Wc_p(e^{V[v_1]}_t(\mu), e^{V[v_2]}_t(\nu))^p\\
			&\qquad \leq \Wc_p(\mu, \nu)^p + p2^{p-2}\biggl(R\Wc_p(\mu,\nu)^p + 3\Wc_p(\mu,\nu)^{p-1}\Wc_p(V[v_1],V[v_2])\biggr)t\\
			&\qquad + \Kc\max(1, R^{p+1})\biggl(t^{p+1} + \Wc_p(V[v_1],V[v_2])t^p + \Wc_p(\mu, \nu)^{p-2}t^3\\
			&\qquad + \Wc_p(\mu, \nu)^{p-1}t^2 + \Wc_p(\mu,\nu)^2t^{p-1} + \Wc_p(\mu,\nu)t^p\\ &\qquad + \Wc_p(\mu, \nu)\Wc_p(V[v_1],V[v_2])t^{p-1} + \Wc_p(\mu,\nu)^{p-2}\Wc_p(V[v_1],V[v_2])t^2\biggr).
		\end{aligned}
	\end{equation}
	Assume furthermore that $V$ is $L$--Lipschitz with respect to the $\Wc_p$--distance for some $L>0$. In particular, if $v_1 = \mu$ and $v_2 = \nu$, then
	\begin{equation}\label{eq:estfrightVp}
		\begin{aligned}
			&\Wc_p(f^{V}_{t^2}(\mu), f^{V}_{t^2}(\nu))^p\\
			&\qquad \leq \Wc_p(\mu, \nu)^p + 2^{p-2}p\biggl(9(p-1)L^2 + pR^2 + 6RL \biggr)\Wc_p(\mu,\nu)^pt^2\\
			&\qquad + \Kc\max(1, R^{p+2})\biggl( 1  + \Wc_p(\mu, \nu)^{p - 1} + \Wc_p(\mu, \nu)^{p-2} \\
			&\qquad+ (L^2 + L + 1)\Wc_p(\mu, \nu)^2 + (1 + L)\Wc_p(\mu, \nu) \biggr)t^{\min(3,p)}.
		\end{aligned}
	\end{equation}
	and,
	\begin{equation}\label{eq:esterightVp}
		\begin{aligned}
			&\Wc_p(e^{V}_t(\mu), e^{V}_t(\nu))^p\\
			&\qquad \leq \Wc_p(\mu, \nu)^p + p2^{p-2}\bigl(R + 3L\bigr)\Wc_p(\mu,\nu)^{p}t\\
			&\qquad + \Kc\max(1, R^{p+1})\biggl(1 + (1 + L)\Wc_p(\mu,\nu) + \Wc_p(\mu, \nu)^{p-2}\\
			&\qquad + (1+L)\Wc_p(\mu, \nu)^{p-1} + (1 + L)\Wc_p(\mu,\nu)^2 \biggr)t^{\min(p-1, 2)}.
		\end{aligned}
	\end{equation}
	More precisely, for $p=2$,
	\begin{equation}\label{eq:estfrandV2}
		\begin{aligned}
			&\Wc_2(f^{V[v_1]}_{t^2}(\mu), f^{V[v_2]}_{t^2}(\nu))^2 \leq \biggl(\Wc_2(\mu, \nu)^2\\
			&\qquad + 2t^2\biggl(9\Wc_2(V[v_1], V[v_2])^2 + 3R\Wc_2(\mu,\nu) \Wc_2(V[v_1], V[v_2])\biggr)\\
			&\qquad + 4R^2\biggl(R\Wc_2(\mu, \nu) + 3\Wc_2(V[v_1], V[v_2])\biggr)\frac{t^3}{6} + \frac{2R^4}{3}t^4 \biggr)e^{8R^2t^2}.
		\end{aligned}
	\end{equation}
	and, in particular, if $\mu = v_1$ and $\nu = v_2$, then,
	\begin{equation}\label{eq:esterandV2}
		\begin{aligned}
			&\Wc_2(f^{V}_{t^2}(\mu), f^{V}_{t^2}(\nu))^2 \leq \biggl(\Wc_2(\mu, \nu)^2  + 2(9L^2 + 3LR)\Wc_2(\mu, \nu)^2t^2\\
			&\qquad + 4R^2(R+3L)\Wc_2(\mu, \nu)\frac{t^3}{6} + \frac{2R^4}{3}t^4 \biggr)e^{8R^2t^2}.
		\end{aligned}
	\end{equation}
	Furthermore,
	\begin{equation}\label{eq:estfrightV2}
		\begin{aligned}
			&\Wc_2(e^{V[v_1]}_t(\mu), e^{V[v_2]}_t(\nu))^2 \leq \biggl(\Wc_2(\mu, \nu)^2 + 6\Wc_2(\mu,\nu) \Wc_2(V[v_1], V[v_2])t\\
			&\qquad + 6R\Wc_2(V[v_1], V[v_2])t^2\biggr)e^{2Rt},\\
		\end{aligned}
	\end{equation}
	and, in particular, if $\mu = v_1$ and $\nu = v_2$, then,
	\begin{equation}\label{eq:esterightV2}
		\begin{aligned}
			&\Wc_2(e^{V}_t(\mu), e^{V}_t(\nu))^2 \leq \biggl(\Wc_2(\mu, \nu)^2 + 6L\Wc_2(\mu,\nu)^2t + 6RL\Wc_2(\mu, \nu)t^2\biggr)e^{2Rt}.\\
		\end{aligned}
	\end{equation}
\end{lem}
\begin{proof}
	Using Theorem~\ref{th:existenceoptimaltransportplan}, let $\ga$ be an optimal transport plan from $\mu$ to $\nu$ and $\Ga$ be an optimal transport plan from $\tau^{-\ol{V}[v_1]}_*V[v_1]$ to $\tau^{-\ol{V}[v_2]}_*V[v_2]$. 
	First, note that,
	\begin{equation}
		\begin{aligned}
			&\Wc_p(\tau^{-\ol{V}[v_1]}_*V[v_1], \tau^{-\ol{V}[v_2]}_*V[v_2])\\ &\qquad \leq \Wc_p(V[v_1], V[v_2]) + \| \ol{V}[v_1]- \ol{V}[v_2] \|_{W^{2,\infty}(M)} \\
			&\qquad \leq 3 \Wc_p(V[v_1], V[v_2]),
		\end{aligned}
	\end{equation}
	by Lemma~\ref{le:bary}.
	For any two vector--fields $X,Y \in \Xc^{2,\infty}(M)$, denote $\rho^{X,Y}_t \coloneqq (\rho^{X}_t,\rho^{Y}_t)$. Then, 
	one readily shows that $\bigl(\rho^{X,Y}_t\bigr)_* \ga$ is a transport plan between $\bigl(\rho^{X}_t\bigr)_* \mu$ and $\bigl(\rho^{Y}_t\bigr)_* \nu$.

	Define $\int_{\Xc^{2,\infty}(M) \times \Xc^{2,\infty}(M)} \bigl(\rho^{X,Y}_t\bigr)_* \ga \der\Ga(X,Y)$ by 
	\begin{equation}\label{eq:gentransplan}
		\begin{aligned}
			&\int_{M \times M} \ph(x,y) \der\biggl(\int_{\Xc^{2,\infty}(M) \times \Xc^{2,\infty}(M)} \bigl(\rho^{X,Y}_t\bigr)_* \ga \der\Ga(X,Y)\biggr)\\
			&\qquad = \int_{\Xc^{2,\infty}(M) \times \Xc^{2,\infty}(M)} \biggl(\int_{M \times M} \ph(x,y) \der\bigl(\rho^{X,Y}_t\bigr)_* \ga(x,y)\biggr) \der\Ga(X,Y)\\
			&\qquad = \int_{\Xc^{2,\infty}(M) \times \Xc^{2,\infty}(M)} \biggl(\int_{M \times M} \ph(\rho^X_t(x),\rho^Y_t(y)) \der \ga(x,y)\biggr) \der\Ga(X,Y),
		\end{aligned}
	\end{equation}
	for all $\ph \in \Cc_c(M)$. Then, $\int_{\Xc^{2,\infty}(M) \times \Xc^{2,\infty}(M)} \bigl(\rho^{X,Y}_t\bigr)_* \ga \der\Ga(X,Y)$ is a transport plan from
	$f^{V[v_1]}_{t^2}(\mu)$ to $f^{V[v_2]}_{t^2}(\nu)$.
	We obtain for all $X,Y\in \Xc^{2,\infty}(M)$, all $x,y \in M$ and all $s >0$,
	\begin{equation}
		\begin{aligned}
			\frac{\der}{\der t}\biggl(t \mapsto \rho^X_t(x)\biggr)(s) &= X(\rho^X_s(x)), \\
			\frac{\der^2}{\der t^2}\biggl(t \mapsto \rho^X_t(x)\biggr)(s) &= \der X(\rho^X_s(x))\bigl(X(\rho^X_s(x))\bigr),\\
			\frac{\der^3}{\der t^3}\biggl(t \mapsto \rho^X_t(x)\biggr)(s) &= \der^2 X(\rho^X_s(x))\bigl(X(\rho^X_s(x))\bigr)\bigl(X(\rho^X_s(x))\bigr)\\
			& + \der X(\rho^X_s(x))\biggl(\der X(\rho^X_s(x))\bigl(\der X(\rho^X_s(x))\bigr)\biggr)\\
		\end{aligned}
	\end{equation}
	for a.e.\ $s>0$. Thus,
	\begin{equation}\label{eq:derauxbigproof}
		\begin{aligned}
			&\frac{\der}{\der t}\biggl(t \mapsto \bigl\| \rho^X_t(x) - \rho^Y_t(y)\bigr\|^p\biggr)(s)\\
			&\quad = p \bigl\|\rho^X_s(x) - \rho^Y_s(y)\bigr\|^{p-2}\biggl<\rho^X_s(x) - \rho^Y_s(y),\frac{\der}{\der t}\biggl(t \mapsto \rho^X_t(x) - \rho^Y_t(y)\biggr)(s)\biggr>, \\
			&\frac{\der^2}{\der t^2}\biggl(t \mapsto \bigl\|\rho^X_t(x) - \rho^Y_t(y)\bigr\|^p\biggr)(s) = p(p-2) \bigl\|\rho^X_s(x) - \rho^Y_s(y)\bigr\|^{p-2}\\
                        &\qquad\qquad\qquad\cdot\biggl<\frac{\rho^X_s(x) - \rho^Y_s(y)}{\|\rho^X_s(x) - \rho^Y_s(y)\|},\frac{\der}{\der t}\biggl(t \mapsto \rho^X_t(x) - \rho^Y_t(y)\biggr)(s)\biggr>^2 \\
			&\quad + p \bigl\|\rho^X_s(x) - \rho^Y_s(y)\bigr\|^{p-2}\biggl\|\frac{\der}{\der t}\biggl(t \mapsto \rho^X_t(x) - \rho^Y_t(y)\biggr)(s)\biggr\|^2\\
			&\quad + p \bigl\|\rho^X_s(x) - \rho^Y_s(y)\bigr\|^{p-2}\biggl<\rho^X_s(x) - \rho^Y_s(y),\frac{\der^2}{\der t^2}\biggl(t \mapsto \rho^X_t(x) - \rho^Y_t(y)\biggr)(s)\biggr>.
		\end{aligned}
	\end{equation}
	This yields for $t > 0$
	% \begin{equation}
	% 	\begin{aligned}
	% 		&\bigl\|\rho^X_t(x) - \rho^Y_t(y)\bigr\|^p = \|x - y\|^p + p\|x-y\|^{p-1}\bigl<x - y, X(x) - Y(y)\bigr>t \\
	% 		&\qquad + \frac{1}{2}\biggl(p(p-2)\|x-y\|^{p-2}\bigl<\frac{x-y}{\|x-y\|},X(x) - Y(y)\bigr>^2 + p\|x-y\|^{p-2}\|X(x) - Y(y)\|^2 \\
	% 		&\qquad +  p\|x-y\|^{p-2}\bigl<x-y, \der X(x)(X(x)) - \der Y(y)(Y(y))\bigr>\biggr)t^2 \\
	% 		&\qquad + \int_0^t (t-s)\biggl[ \frac{\der^2}{\der t^2}\biggl(t \mapsto \bigl\|\rho^X_t(x) - \rho^Y_t(y)\bigr\|^p\biggr)(s) - \frac{\der^2}{\der t^2}\biggl(t \mapsto \bigl\|\rho^X_t(x) - \rho^Y_t(y)\bigr\|^p\biggr)(0)\biggr] \der s . 
	% 	\end{aligned}
	% \end{equation}
	\begin{equation}\label{eq:rhoest}
		\begin{aligned}
			&\bigl\|\rho^X_t(x) - \rho^Y_t(y)\bigr\|^p = \|x - y\|^p + p\|x-y\|^{p-2}\bigl<x - y, X(x) - Y(y)\bigr>t \\
			&\qquad + \int_0^t (t-s)\frac{\der^2}{\der t^2}\biggl(t \mapsto \bigl\|\rho^X_t(x) - \rho^Y_t(y)\bigr\|^p\biggr)(s)  \der s.
		\end{aligned}
	\end{equation}
	In particular, for $p=2$ we obtain
	\begin{equation}
		\begin{aligned}
			&\bigl\|\rho^X_t(x) - \rho^Y_t(y)\bigr\|^2 = \|x - y\|^2 + 2\bigl<x - y, X(x) - Y(y)\bigr>t \\
			&\qquad + 2\int_0^t (t-s)\biggl[\biggl\|\frac{\der}{\der t}\biggl(t \mapsto \rho^X_t(x) - \rho^Y_t(y)\biggr)(s) \biggr\|^2\\
			&\qquad + \biggl<\rho^X_s(x) - \rho^Y_s(y),\frac{\der^2}{\der t^2}\biggl(t \mapsto \rho^X_t(x) - \rho^Y_t(y)\biggr)(s)\biggr>\biggr]\der s.
		\end{aligned}
	\end{equation}
	First, note that, using Lemma~\ref{le:intlin}, for all $x,y \in M$
	\begin{equation}
		\begin{aligned}
			&\int_{\Xc^{2,\infty}(M) \times \Xc^{2,\infty}(M)} p\|x-y\|^{p-1}\bigl<x - y, X(x) - Y(y)\bigr> \der \Ga(X,Y)\\
			&\qquad = \int_{\Xc^{2,\infty}(M)} p\|x-y\|^{p-1}\bigl<x - y, X(x)\bigr> \der \tau^{-\ol{V}[\mu]}_*V[\mu](X)\\
			&\qquad - \int_{\Xc^{2,\infty}(M)} p\|x-y\|^{p-1}\bigl<x - y, Y(y)\bigr> \der \tau^{-\ol{V}[\nu]}_*V[\nu](Y) = 0.
		\end{aligned}
	\end{equation}
	Furthermore, for all $x,y \in M$ and all $X,Y \in \Xc^{2,\infty}(M)$, we have
	\begin{equation}
		\begin{aligned}
			&\| \der X(x)(X(x)) - \der Y(y)(Y(y))\| = \|\der X(x)(X(x)) - \der X(y)(X(y))\|\\
			&\quad\qquad + \| \der X(y)(X(y)) - \der Y(y)(Y(y))\|\\
			&\quad \leq 2\|X\|_{W^{2,\infty}(M)}^2 \|x-y\| + 2\bigl(\|X\|_{W^{2,\infty}(M)} + \|Y\|_{W^{2,\infty}(M)}\bigr)\|X - Y\|_{W^{2,\infty}(M)}.
		\end{aligned}
	\end{equation}
	Combining \eqref{eq:derauxbigproof} with the estimates
	\begin{equation}
		\begin{aligned}
			&\|\rho^X_s(x) - \rho^Y_s(y)\| \leq \|x-y\| + s\bigl(\|X\|_{W^{2,\infty}(M)} + \|Y\|_{W^{2,\infty}(M)}\bigr),\\
			&\biggl\|\frac{\der}{\der t}\biggl(t \mapsto \rho^X_t(x) - \rho^Y_t(y)\biggr)(s) \biggr\| \leq \|X\|_{W^{2,\infty}(M)}\|x -y\| + \| X-Y\|_{W^{2,\infty}(M)}\\ &\qquad + s\bigl(\|X\|_{W^{2,\infty}(M)}^2 + \|Y\|_{W^{2,\infty}(M)}^2\bigr),\\
			&\biggl\|\frac{\der^2}{\der t^2}\biggl(t \mapsto \rho^X_t(x) - \rho^Y_t(y)\biggr)(s) \biggr\| \leq 2\|X\|_{W^{2,\infty}(M)}^2\|x-y\|\\ &\qquad + 2\bigl(\|X\|_{W^{2,\infty}(M)} + \|Y\|_{W^{2,\infty}(M)}\bigr)\|X - Y\|_{W^{2,\infty}(M)}\\ &\qquad + 2s\bigl(\|X\|_{W^{2,\infty}(M)}^3 + \|Y\|_{W^{2,\infty}(M)}^3\bigr),
		\end{aligned}
	\end{equation}
	and subsequently using the inequality of Cauchy--Schwarz, we roughly estimate the integral term in (\ref{eq:rhoest}) to obtain the existence of $C>0$ such that, for $t < 1$,
	\begin{equation}\label{eq:intermest}
		\begin{aligned}
			&\int_{M \times M} \|x-y\|^p \der\biggl(\int_{\Xc^{2,\infty}(M) \times \Xc^{2,\infty}(M)} \bigl(\rho^{X,Y}_t\bigr)_* \ga \der\Ga(X,Y)\biggr)
			 \leq \Wc_p(\mu, \nu)^p\\ &\qquad+ 2^{p-1}\int_0^t (t-s)\int_{\bigl(\Xc^{2,\infty}(M)\bigr)^2 } \int_{M \times M} p(p-1)\|x-y\|^{p-2}\|X-Y\|_{W^{2,\infty}(M)}^2\\
			&\qquad + 2pR\|x-y\|^{p-1}\|X-Y\|_{W^{2,\infty}(M)} + p^2R^2\|x-y\|^p\\
			&\qquad + C\max(1, R^{p+2})\biggl(s^p + \|X-Y\|_{W^{2,\infty}(M)}^2s^{p-2} + \|X-Y\|_{W^{2,\infty}(M)}s^{p-1}\\
			&\qquad + \|x-y\|^{p-1}s + \|x-y\|^{p-2}s^2 + \|x-y\|^2s^{p-2}\\
			&\qquad + \|x-y\|s^{p-1} + \|x-y\|\|X-Y\|_{W^{2,\infty}(M)}s^{p-2}\\
			&\qquad + \|x-y\|^{p-2}\|X-Y\|_{W^{2,\infty}(M)}s\biggr) \der \ga(x,y) \der\Ga(X,Y) \der s.
		\end{aligned}
	\end{equation} 
	Now, for all $q<p$ and $\mu, \nu \in \Wc_p(M)$, by Jensen's inequality, we obtain for an optimal transport plan $g$ from $\mu$ to $\nu$ for the $p$--Wasserstein distance, 
	\begin{equation}
		\Wc_p(\mu,\nu)^p = \int_{M \times M} d(x,y)^p \der g(x,y) \geq  \biggl(\int_{M \times M} d(x,y)^q \der g(x,y)\biggr)^{\frac{p}{q}}.
	\end{equation}
	Hence, there exists a constant $\Kc>0$ such that
	\begin{equation}
		\begin{aligned}
			&\Wc_p(f^{V[v_1]}_{t^2}(\mu), f^{V[v_2]}_{t^2}(\nu))^p\\
			&\qquad \leq \Wc_p(\mu, \nu)^p + 2^{p-2}p\biggl(9(p-1)\Wc_p(\mu,\nu)^{p-2}\Wc_p(V[v_1],V[v_2])^2\\
			&\qquad + pR^2\Wc_p(\mu,\nu)^p + 6R\Wc_p(\mu,\nu)^{p-1}\Wc_p(V[v_1],V[v_2]) \biggr)t^2\\
			&\qquad + \Kc\max(1, R^{p+2})\biggl( t^{p+2} + \Wc_p(V[v_1],V[v_2])^2t^p  + \Wc_p(V[v_1],V[v_2])t^{p+1}\\ &\qquad + \Wc_p(\mu, \nu)^{p - 1}t^3 + \Wc_p(\mu, \nu)^{p-2}t^4 + \Wc_p(\mu, \nu)^2t^p + \Wc_p(\mu, \nu)t^{p+1} \\
			&\qquad + \Wc_p(\mu, \nu)\Wc_p(V[v_1],V[v_2])t^{p} + + \Wc_p(\mu, \nu)^{p-2}\Wc_p(V[v_1],V[v_2])t^3\biggr).
		\end{aligned}
	\end{equation}
	In particular, if $V$ is $L$--Lipschitz with respect to the $\Wc_p$--distance, 
	\begin{equation}
		\begin{aligned}
			&\Wc_p(f^{V}_{t^2}(\mu), f^{V}_{t^2}(\nu))^p\\
			&\qquad \leq \Wc_p(\mu, \nu)^p + 2^{p-2}p\biggl(9(p-1)L^2 + pR^2 + 6RL \biggr)\Wc_p(\mu,\nu)^pt^2\\
			&\qquad + \Kc\max(1, R^{p+2})\biggl( 1  + \Wc_p(\mu, \nu)^{p - 1} + \Wc_p(\mu, \nu)^{p-2} \\
			&\qquad+ (L^2 + L + 1)\Wc_p(\mu, \nu)^2 + (1 + L)\Wc_p(\mu, \nu) \biggr)t^{\min(3,p)}.
		\end{aligned}
	\end{equation}
	%Now, suppose $p<3$. Then, for all $s > 0$,
	%\begin{equation}
	%	\bigl(\rho^X_s(x) - \rho^Y_s(y)\bigr)^{p-2} - \bigl(x - y\bigr)^{p-2} \leq s^{p-2}\|X - Y\|_{W^{2,\infty}(M)}^{p-2}.
	%\end{equation}
	The estimates for $\Wc_p(e^{V[v_1]}_t(\mu), e^{V[v_2]}_t(\nu))^p$ follow analogously.
	Now, we work with $p=2$ and give a more precise estimate of the integral term in (\ref{eq:rhoest}). We get,
	\begin{equation}
		\begin{aligned}
			&\biggl<\rho^X_s(x) - \rho^Y_s(y),\frac{\der^2}{\der t^2}\biggl(t \mapsto \rho^X_t(x) - \rho^Y_t(y)\biggr)(s)\biggr> \\
			&\qquad = \biggl<\rho^X_s(x) - \rho^Y_s(y), \der X(\rho^X_s(x))\bigl(X(\rho^X_s(x))\bigr) - \der X(\rho^X_s(y))\bigl(X(\rho^X_s(y))\bigr)\biggr>\\
			&\qquad + \biggl<\rho^X_s(x) - \rho^Y_s(y), \der X(\rho^X_s(y))\bigl(X(\rho^X_s(y))\bigr) - \der Y(\rho^Y_s(y))\bigl(Y(\rho^Y_s(y))\bigr)\biggr>\\
			&\qquad = \biggl<\rho^X_s(x) - \rho^Y_s(y), \der X(\rho^X_s(x))\bigl(X(\rho^X_s(x))\bigr) - \der X(\rho^X_s(y))\bigl(X(\rho^X_s(y))\bigr)\biggr>\\
			&\qquad + \biggl< x - y + \int_0^s X\bigl(\rho^X_r(x)\bigr) - Y\bigl(\rho^Y_r(y)\bigr) \der r, \der X(y)\bigl(X(y)\bigr) - \der Y(y)\bigl(Y(y)\bigr)\\
			&\qquad  + \int_0^s \der^2 X(\rho^X_r(y))\bigl(X(\rho^X_r(y))\bigr)\bigl(X(\rho^X_r(y))\bigr)\\
			&\qquad + \der X(\rho^X_r(y))\biggl(\der X(\rho^X_r(y))\bigl(\der X(\rho^X_r(y))\bigr)\biggr) \\
			&\qquad - \der^2 Y(\rho^Y_r(y))\bigl(Y(\rho^Y_r(y))\bigr)\bigl(Y(\rho^Y_r(y))\bigr)\\
			&\qquad + \der Y(\rho^Y_r(y))\biggl(\der Y(\rho^Y_r(y))\bigl(\der Y(\rho^Y_r(y))\bigr)\biggr) \der r \biggr>.
		\end{aligned}
	\end{equation}
	Hence, using (\ref{eq:gentransplan}) with $d^p$, Fubini--Tonelli and the previous computations, we obtain
	\begin{equation*}
		\begin{aligned}
			& \int_{\Xc^{2,\infty}(M) \times \Xc^{2,\infty}(M)} \biggl(\int_{M \times M} \bigl\|\rho^X_t(x)- \rho^Y_t(y) \bigr\|^2 \der \ga(x,y)\biggr) \der\Ga(X,Y)\\
			& \leq \Wc_2(\mu, \nu)^2  + 2\int_0^t(t-s) \int_{\Xc^{2,\infty}(M) \times \Xc^{2,\infty}(M)} \int_{M \times M} 2\biggl(\|X-Y\|^2_{W^{2,\infty(M)}}\\
			&\qquad + 2\|X\|^2_{W^{2,\infty(M)}}\bigl\|\rho^X_s(x)- \rho^Y_s(y) \bigr\|^2\biggr) + 2R\|x-y\|\|X-Y\|_{W^{2,\infty}(M)}\\
			&\qquad  + 4R^2\|X-Y\|_{W^{2,\infty}(M)}s + 4R^3\|x-y\|s + 8R^4s^2 \der \ga(x,y)\der\Ga(X,Y) \der s\\
			&\leq \Wc_2(\mu, \nu)^2  + 2t^2\biggl(9\Wc_2(V[v_1], V[v_2])^2\\
			&\qquad + R\int_{M \times M} \|x -y\| \der \ga(x,y) \int_{\bigl(\Xc^{2,\infty}(M)\bigr)^2} \|X -Y\|_{W^{2,\infty}(M)}\der \Ga(X,Y)\biggr)\\
			&\qquad + 4R^2\biggl(R\int_{M \times M} \|x -y\| \der \ga(x,y)\\
			&\qquad + \int_{\Xc^{2,\infty}(M) \times \Xc^{2,\infty}(M)} \|X -Y\|_{W^{2,\infty}(M)}\der \Ga(X,Y)\biggr)\frac{t^3}{6} + \frac{2R^4}{3}t^4 \\
			&\qquad +  \int_0^t(t-s)\!\! \int_{\bigl(\Xc^{2,\infty}(M)\bigr)^2}\!\! \int_{M^2}\!\!  8 R^2\bigl\|\rho^X_s(x)- \rho^Y_s(y) \bigr\|^2\der \ga(x,y)\der\Ga(X,Y) \der s \\
			% ARTIFICIAL PAGE BREAK \\
%		\end{aligned}
%	\end{equation*}
%			% ARTIFICIAL PAGE BREAK \\
%	\begin{equation*}
%		\begin{aligned}
			&\leq \Wc_2(\mu, \nu)^2  + 2t^2\biggl(9\Wc_2(V[v_1], V[v_2])^2 + 3R\Wc_2(\mu,\nu) \Wc_2(V[v_1], V[v_2])\biggr)\\
			&\qquad + 4R^2\biggl(R\Wc_2(\mu, \nu) + 3\Wc_2(V[v_1], V[v_2])\biggr)\frac{t^3}{6} + \frac{2R^4}{3}t^4 \\
			&\qquad +  \int_0^t(t-s)\!\!  \int_{\bigl(\Xc^{2,\infty}(M)\bigr)^2}\!\! \int_{M^2}\!\! 8 R^2\bigl\|\rho^X_s(x)- \rho^Y_s(y) \bigr\|^2\der \ga(x,y)\der\Ga(X,Y) \der s,\\
		\end{aligned}
	\end{equation*}
	where the last inequality follows from Jensen's inequality. Finally, it follows by Grönwall's inequality and the Lipschitz--property of $V$ that
	\begin{equation}
		\begin{aligned}
			&\Wc_2(f^{V[v_1]}_{t^2}(\mu), f^{V[v_2]}_{t^2}(\nu))^2 \leq \biggl(\Wc_2(\mu, \nu)^2  + 2t^2\biggl(9L^2\Wc_2(v_1, v_2)^2\\
			&\qquad + 3LR\Wc_2(\mu,\nu) \Wc_2(v_1, v_2)\biggr) + 4R^2\biggl(R\Wc_2(\mu, \nu) + 3L\Wc_2(v_1, v_2)\biggr)\frac{t^3}{6}\\
			&\qquad + \frac{2R^4}{3}t^4 \biggr)e^{8R^2t^2}.
		\end{aligned}
	\end{equation}
	The estimates for $\Wc_2(e^{V[v_1]}_t(\mu), e^{V[v_2]}_t(\nu))^2$ follow analogously.
\end{proof}
An analogon of Lemma~\ref{le:fflowop2} also holds on all Riemannian manifolds with bounded geometry. 
\begin{lem}\label{le:fflowopman}
	Let $(M,g)$ be a Riemannian manifold with bounded geometry. With the assumptions of Lemma~\ref{le:fflowop2}, the estimates 
	(\ref{eq:estfrandVp}), (\ref{eq:esterandVp}), (\ref{eq:estfrandV2}) and (\ref{eq:esterandV2}) also hold, with an additional multiplicative constant $C(M) > 0$ depending only on the manifold $M$, with $C(\Rz^n) = 1$.
\end{lem}
\begin{proof}
    We proceed similarly as in the proof of Lemma~\ref{le:fflowop2} in local charts. Let $x,y \in M$ and $X,Y \in \Xc^{2,\infty}(M)$ and $t>0$. If there exists a unique smooth unit--speed geodesic $\ga$ from $\rho^X_t(x)$ to $\rho^Y_t(y)$, a technical computation yields
	\begin{equation}
		\begin{aligned}
			&\frac{\der}{\der t}\biggl(t \mapsto d\bigl(\rho^X_t(x), \rho^Y_t(y)\bigr)^p\biggr)(s)\\
			&\quad = p d\bigl(\rho^X_t(x), \rho^Y_t(y)\bigr)^{p-1}\biggl[g(\rho^X_t(x))\biggl(X(\rho^X_t(x)), \frac{\der}{\der t}\ga(0)\biggr)\\
			&\qquad - g(\rho^Y_t(y))\biggl(Y(\rho^Y_t(y)), \frac{\der}{\der t}\ga\bigl(d(\rho^X_t(x), \rho^Y_t(y))\bigr)\biggr)\biggr], \\
		\end{aligned}
	\end{equation}
	such that, in particular, if there exists a unique minimizing geodesic between $x$ and $y$,
	\begin{equation}
		\begin{aligned}
			&\frac{\der^2}{\der t^2}\biggl(t \mapsto d\bigl(\rho^X_t(x), \rho^Y_t(y)\bigr)^p\biggr)(0)\\
			&\quad = p(p-2) d\bigl(x, y\bigr)^{p-2}\biggl[g(x)\biggl(X(x),\frac{\der}{\der t}\ga(0)\biggr)\\
			&\qquad \qquad - g(y)\biggl(Y(y),\frac{\der}{\der t}\ga(d(x,y))\biggr)\biggr]^2 \\
			&\quad + p d\bigl(x, y\bigr)^{p-2}d^2(x,y)\biggl((x,y) \mapsto \frac{1}{2}d(x,y)^2\biggr)\biggl[X(x),Y(y),X(x),Y(y)\biggr]\\
			&\quad - p d\bigl(x, y\bigr)^{p-2}\biggl[g(x)\biggl(\nabla_X X(x),\frac{\der}{\der t}\ga(0)\biggr)\\
			&\qquad \qquad  - g(y)\biggl(\nabla_Y Y(y),\frac{\der}{\der t}\ga(d(x,y))\biggr)\biggr]^2.
		\end{aligned}
	\end{equation}
	Using the bounded geometry of $M$ to unifomly bound the curvature terms appearing in the second derivative above, we directly obtain a similar estimate as the expression inside the integral in \eqref{eq:intermest}, if $M$ is uniquely geodesic.

	If there are multiple minimizing geodesics between $x$ and $y$, the Hessian of $d^2: (z,w) \mapsto d(z,w)^2$ might not be well--defined in $(x,y)$. However, $d^2$ is still $W^{1,\infty}$ on any bounded tubular neighbourhood of the diagonal of $M \times M$. 
	Thus, its Clarke generalised Hessian (see \citet{rockafellar1998variational}) is well--defined and uniformly bounded on such a neighbourhood, again by the bounded geometry of $M$. A slightly more involved computation then yields an analogous estimate as the one we obtain for uniquely geodesic manifolds.
\end{proof}
Lemma~\ref{le:fflowop2} and Lemma~\ref{le:fflowopman} are phrased for probability vector--fields for convenience. We obtain a similar result for average flow--operators associated to any vector--field probabilities in a completely analogous way.
\begin{co}\label{co:fflowop2}
	Let $(M,g)$ be of bounded geometry. Let  $v,w \in \Wc_{p+2}(\Xc^{2,\infty}(M))$ and $R>0$ with $\Mc_{p+2}(\mu), \Mc_{p+2}(\nu) \leq R^{p+2}$. Then, for all $\mu, \nu$, the estimates 
	(\ref{eq:estfrandVp}), (\ref{eq:esterandVp}), (\ref{eq:estfrandV2}) and (\ref{eq:esterandV2}) also hold.
\end{co}
\begin{proof}
	The proof is completely analogous to the proof of Lemma~\ref{le:fflowopman}.
\end{proof}

\subsubsection{(AFAS)---Hölder--regularity in time}

In particular, we deduce a stronger regularity result for (AFAS) in manifolds with bounded geometry.
\begin{co}\label{co:pHolder}
	Let $T>0$, $p \geq 2$, $\mu_0 \in \Wc_p(M)$. Let $\Pf$ be a partition of $[0,T]$. Let $V$ be a probability vector--field. Assume that there exists $R > 0$ such that for all $\mu \in \Wc_p(M)$, $V[\mu] \in \Wc_{p+2}(\Rz^n)$ and $\Mc_{p+2}(V[\mu]) \leq R^{p+2}$. Then:
	\begin{itemize}
	 \item The Average Flow Approximation Series $\mu^{\Pf}$ with initial value $\mu_0$ is locally Hölder continuous with respect to the $p$--Wasserstein distance with exponent $\frac{1}{p}$.
	 \item More precisely, there exists $\de >0$ such that, if the path--length $\pf(\Pf) \leq \de$, for all $t,s \in [0,T]$, with $|t-s| \leq \frac{1}{2^{p-3}pR\bigl(5R(2p+1) + 4\bigr)}$, we have 
	\begin{equation}
		\begin{aligned}
			&\Wc_p(\mu^{\Pf}_t, \mu^{\Pf}_s) \leq C(M)\biggl(2^{p-3}pR\bigl(5R(2p+1) + 4\bigr)\biggr)^{\frac{1}{p}}|t-s|^{\frac{1}{p}},\\
		\end{aligned}
	\end{equation} 
	where $C(M) > 0$ is a constant depending only on the manifold $M$, with $C(\Rz^n) = 1$.
	\end{itemize}
\end{co}
\begin{proof}
	The statement follows by induction using (\ref{eq:estfrandVp}) in Lemma~\ref{le:fflowopman}. Let $t \in [0,T]$. We start with $\mu = \nu = \mu_t$ and in every step, we set $v_1 = \mu$ and $V[v_2] = \de_0$. Note that for all $\mu \in \Wc_p(M)$, 
	\begin{equation}
		\begin{aligned}
			\Wc_p(V[\mu], \de_0) & \leq R.\\
		\end{aligned}
	\end{equation}
\end{proof}
Furthermore, we obtain Lipschitz regularity for the $p$--th moment of an (AFAS) in manifolds with bounded geometry.
\begin{co}\label{co:pmomunifLip}
	Let $(M,g)$ be of bounded geometry, $T>0$, $p \geq 2$, $\mu_0 \in \Wc_p(M)$. Let $\Pf$ be a partition of $[0,T]$, $V$ a probability vector--field. Assume that there exists $R > 0$ such that for all $\mu \in \Wc_p(M)$, $V[\mu] \in \Wc_{p+1}(\Xc^{2,\infty}(M))$ and $\Mc_{p+1}(V[\mu]) \leq R^{p+1}$. Then:
	\begin{itemize}
	    \item The function $t \mapsto \Mc_p(\mu^{\Pf}_t)$ is Lipschitz with respect to the $p$--Wasserstein distance.
	    \item  More precisely, for all $\ep>0$, there exists $\de >0$ such that, if the path--length $\pf(\Pf) \leq \de$, for all $t,s \in [0,T]$, we have 
	\begin{equation}\label{eq:pmomunifbd}
		\begin{aligned}
			&\bigl|\Mc_p(\mu^{\Pf}_t) \bigr| \leq C(M)\bigl(\Mc_p(\mu_0) + p2^{p-1}R\bigl(2pR + 1\bigr)t\bigr)e^{\bigl(p2^{p-1}R\bigl(2pR + 1\bigr) + \ep\bigr)t}\\ &\qquad + \ep t,\\
		\end{aligned}
	\end{equation} 
	and, in particular, there exists a constant $\Hc$ depending on $\Mc_p(\mu_0)$, $T$, $R$, $p$ and $M$, such that
	\begin{equation}
		\begin{aligned}
			&\bigl|\Mc_p(\mu^{\Pf}_t) - \Mc_p(\mu^{\Pf}_s)\bigr| \leq \Hc|t-s|,
		\end{aligned}
	\end{equation} 
	where $C(M) > 0$ is a constant depending only on the manifold $M$, with $C(\Rz^n) = 1$.
	\end{itemize} 
\end{co}
\begin{proof}
	For $(M,g) = (\Rz^n, g^{eucl})$, we proceed as in the proof of Lemma~\ref{le:fflowop2}. Let $t,R >0$. A lengthy computation shows that for all $\mu \in \Wc_p(M)$,  $V \in \Wc_{p+1}(\Xc^{2,\infty}(M))$ with $\Mc_{p+1}(M) \leq R^{p+1}$, 
	\begin{equation}\label{eq:mompest}
		\begin{aligned}
			&\biggl| \Mc_p\bigl(f^V_t(\mu)\bigr) - \Mc_p(\mu) \biggr| \\ 
			&\qquad \leq p2^{p-2}\biggl((p-1)\Mc_{p-2}(\mu)\Mc_2(V[\mu]) + \Mc_{p-1}(\mu)\Mc_2(V[\mu])\biggr)t^2\\
			&\qquad + \frac{2^{p-2}}{p-1}\biggl((p-1)\Mc_p(V[\mu]) + \Mc_1(\mu)\Mc(V[\mu])\biggr)t^p\\
			&\qquad + \frac{2^{p-2}}{p+1}\Mc_{p+1}(V[\mu])s^{p+1} + \frac{p2^{p-2}}{6}\Mc_{p-2}(\mu)\Mc_3(V[\mu])t^3.
		\end{aligned}
	\end{equation}
	Similarly, we obtain an estimate for $\biggl| \Mc_p\bigl(e^V_t(\mu)\bigr) - \Mc_p(\mu) \biggr|$. 
	
	Now, the first statement follows by induction using \eqref{eq:mompest} in every step, and a subsequent application of Grönwall's inequality.

	To obtain the second estimate, we use \eqref{eq:pmomunifbd} on $\Rz^n$ to obtain, in particular, the existence of a constant $\Hc'$ depending only 
	on $\Mc_p(\mu_0)$, $T$, $R$ and $p$, and $\de>0$, such that if the path--length $\pf(\Pf) \leq \de$, then, for all $t \in [0,T]$, $\bigl|\Mc_p(\mu^{\Pf}_t) \bigr| \leq \Hc'$.
	Combining this uniform upper bound of the moments with \eqref{eq:mompest}, the second statement readily follows by induction.

	Finally, we obtain the estimates for an arbitrary manifold with bounded geometry by generalising the above proof analogously to the generalisation of the proof of Lemma~\ref{le:fflowop2} in Lemma~\ref{le:fflowopman}.
\end{proof}

Proceeding as in Lemma~\ref{le:fflowopman}, we obtain uniform upper bounds for the difference of the $p$--th and the $q$--th moments of an AFAS.
\begin{lem}\label{le:momentpq}
	Let $(M,g)$ be of bounded geometry. Let $T>0$, $p \geq q \geq 2$, $\mu_0 \in \Wc_p(M)$. Let $\Pf$ be a partition of $[0,T]$. Let $V$ be a probability vector--field. Assume that there exists $R > 0$ such that for all $\mu \in \Wc_p(M)$, $V[\mu] \in \Wc_{p+2}(\Xc^{2,\infty}(M))$ and $\Mc_{p+2}(V[\mu]) \leq R^{p+2}$. Then, there exists $\de >0$ and a constant $\Kc > 0$ depending only on $\Mc_p(\mu_0),T, p, R$ and $M$, such that, if $\pf(\Pf) \leq \de$, for all $t \in [0,T]$,
	\begin{equation}
		\biggl|\Mc_p(\mu^{\Pf}_t) - \Mc_q(\mu^{\Pf}_t)\biggr| \leq \biggl|\Mc_p(\mu_0) - \Mc_q(\mu_0)\biggr| + \Kc(p-q)t. 
	\end{equation}
\end{lem}
\begin{proof}
	For $(M,g) = (\Rz^n, g^{eucl})$, the result follows from a computation similar to the one in the proof of Lemma~\ref{le:fflowop2}. Indeed,
	\begin{equation}
		\begin{aligned}
			&\bigl\|\rho^X_t(x) - x_0\bigr\|^p - \bigl\|\rho^X_t(x) - x_0\bigr\|^q = \|x - x_0\|^p - \|x - x_0\|^q\\
			&\qquad + \bigl(p\|x-x_0\|^{p-2} - q\|x-x_0\|^{q-2}\bigr)\bigl<x - x_0, X(x)\bigr>t \\
			&\qquad + \int_0^t (t-s)\frac{\der^2}{\der t^2}\biggl(t \mapsto \bigl\|\rho^X_t(x) - x_0\bigr\|^p - \bigl\|\rho^X_t(x) - x_0\bigr\|^q\biggr)(s)  \der s,
		\end{aligned}
	\end{equation}
	where 
	\begin{equation}
		\begin{aligned}
			&\frac{\der^2}{\der t^2}\biggl(t \mapsto \bigl\|\rho^X_t(x) - x_0\bigr\|^p - \bigl\|\rho^X_t(x) - x_0\bigr\|^q\biggr)(s)\\
			&\qquad = \biggl(p(p-2) \bigl\|\rho^X_s(x) - x_0\bigr\|^{p-2} - q(q-2) \bigl\|\rho^X_s(x) - x_0\bigr\|^{q-2}\biggr)\\ &\qquad \cdot \biggl<\frac{\rho^X_s(x) - x_0}{\|\rho^X_s(x) - x_0\|}, \frac{\der}{\der t}\biggl(t \mapsto \rho^X_t(x)\biggr)(s)\biggr>^2 \\
			&\qquad + \biggl(p \bigl\|\rho^X_s(x) - x_0\bigr\|^{p-2} - q\bigl\|\rho^X_s(x) - x_0\bigr\|^{q-2}\biggr)\biggl[\biggl\|\frac{\der}{\der t}\biggl(t \mapsto \rho^X_t(x) \biggr)(s)\biggr\|^2\\
			&\qquad + \biggl<\rho^X_s(x) - x_0,\frac{\der^2}{\der t^2}\biggl(t \mapsto \rho^X_t(x) \biggr)(s)\biggr>\biggr].
		\end{aligned}
	\end{equation}
	Note that for all $p\geq q \geq 2$
	\begin{equation}
		\begin{aligned}
			&\biggl|\bigl\|\rho^X_s(x) - x_0\bigr\|^{p-2} - \bigl\|\rho^X_s(x) - x_0\bigr\|^{q-2}\biggr|\\
			&\qquad \leq (p-q)\biggl|\log\bigl(\bigl\|\rho^X_s(x) - x_0\bigr\|\bigr)\biggr|\\
			&\qquad\qquad \cdot \max\biggl(\bigl\|\rho^X_s(x) - x_0\bigr\|^{p-2}, \bigl\|\rho^X_s(x) - x_0\bigr\|^{q-2}\biggr). \\
		\end{aligned}
	\end{equation}
	Let $\min(p-2,1) > \de > 0$, then for all $p \geq q \geq 2$,
	\begin{equation}
		\begin{aligned}
			&\biggl|\log\bigl(\bigl\|\rho^X_s(x) - x_0\bigr\|\bigr)\biggr|\bigl\|\rho^X_s(x) - x_0\bigr\|^{q-2}\\ &\qquad \leq \bigl\|\rho^X_s(x) - x_0\bigr\|^{q-2 + \de} + \bigl\|\rho^X_s(x) - x_0\bigr\|^{q-2 - \de}.
		\end{aligned}
	\end{equation}
	Also note that for all $q<p$, $\Mc_q(\mu) \leq \Mc_p(\mu) + 1$. Therefore, after integration, for all $\mu \in \Wc_p(M)$, 
	\begin{equation}
		\begin{aligned}
			&\biggl|\int_M \int_{\Xc^{2,\infty}(M)} \bigl\|\rho^X_t(x) - x_0\bigr\|^p - \bigl\|\rho^X_t(x) - x_0\bigr\|^q  \der V[\mu](X)\der \mu(x)\biggr|\\
			&\qquad \leq \biggl|\Mc_p(\mu) - \Mc_q(\mu)\biggr| \\
                        &\qquad\quad+ 2^{p-2}\bigl(8p^2R^2 + 6R^2p + 3R^2p\bigr)\bigl(\Mc_p(\mu) + 1\bigr)(p-q)t^2 \\ 
			&\qquad\quad + \frac{2^{2(p-1)}\bigl(8p^2R^2 + 6R^2p + 3R^2p\bigr)R^2}{6}(p-q)t^3.
		\end{aligned}
	\end{equation}
	Now, the result follows by induction on the steps in $\Pf$ using \eqref{eq:pmomunifbd}.

	Finally, we obtain the estimates for an arbitrary manifold with bounded geometry by generalising the above proof analogously to the generalisation of the proof of Lemma~\ref{le:fflowop2} in Lemma~\ref{le:fflowopman}.
\end{proof}
\subsubsection{First Main Result: Generalised Central Limit Theorem---weak version}

Now, we are ready for the proof of our first main result, the generalised central limit theorem. 

\begin{theorem}[Generalised Central Limit Theorem---weak version]\label{th:gencentlim}\strut\newline
Let $T>0$ and $\mu_0 \in \Wc_p(M)$. Let $\bigl( \Pf(\frac{1}{N}) \bigr)_{N \in \Nz}$ be a family of partitions of $[0,T]$. Let $V$ be a uniformly continuous probability vector--field on $\Wc_p(M)$ for $p > 2$. Assume that $V$ has uniformly bounded $p$--th moments, e.g., there exists a constant $B>0$ such that for all $\mu \in \Wc_p(M)$, $\Mc_p(V[\mu]) < B$. 

\begin{enumerate}
	\item Assume that $(M,g)$ is of bounded geometry. Assume furthermore that there exists $R > 0$ such that for all $\mu \in \Wc_p(M)$, $V[\mu] \in \Wc_{p+2}(\Xc^{2,\infty}(M))$ and $\Mc_{p+2}(V[\mu]) \leq R^{p+2}$. Then, there exists a subsequence $\biggl(\mu^{\Pf(\frac{1}{N_k})}\biggr)_{k \in \Nz}$ of the Average Flow Approximation Series $\biggl(\mu^{\Pf(\frac{1}{N})}\biggr)_{N \in \Nz}$ with initial value $\mu_0$ (see Definition~\ref{de:AFAS}) which uniformly converges with respect to the $p$--Wasserstein distance to a $\frac{1}{p}$--Hölder--continuous smooth solution $\mu: [0,T] \to \Wc_p(M)$ of (\ref{eq:MDEd}) with initial value $\mu_0$.
	
	\item Now, we only assume that there exists $q < p$ such that $V$ is continuous for the $\Wc_q$--distance.
	Then, there exists a subsequence $\biggl(\mu^{\Pf(\frac{1}{N_k})}\biggr)_{k \in \Nz}$ of the Average Flow Approximation Series with initial value $\mu_0$ which converges to a smooth solution $\mu: [0,T] \to \Wc_p(M)$ of (\ref{eq:MDEd}) with initial value $\mu_0$. 

	The convergence of $\mu^{\Pf(\frac{1}{N_k})}$ to $\mu$ is pointwise in the $\Wc_{q}$--distance and $\mu(t)$ is continuous in $\Wc_{p}(M)$. 

\end{enumerate}

\end{theorem}

\begin{proof}
Let $p > 2$. Take $N \in \Nz$, and let $(x_i)_{i \in \llbracket 0, I_N \rrbracket}\coloneqq \Pf(\frac{1}{N})$ indexed in increasing order, To alleviate notation, we shall write $\mu^N \coloneqq \mu^{\Pf(\frac{1}{N})}$. The proof consists of three main steps.

\subsubsection*{Step 1:}

First, for all $N \in \Nz$, we give an estimate of the distance $\Wc_p(\mu^N(s), \mu^N(t))$ for $|t-s|$ small. 

Take $N \in \Nz$ and $l \in \llbracket 0, I_N - 1 \rrbracket$. Suppose $\frac{1}{2N} \leq 1$. Then for all $t,s \in \bigl[\frac{x_l + x_{l+1}}{2}, x_{l+1}\bigr]$, we estimate $|t-s| \leq \sqrt{|t-s|}$. Note that since $V$ has $p$--th moments bounded by $B>0$, it has in particular first and second moments bounded by $B + 1$. Then, Lemma~\ref{le:eflowop} yields that 
\begin{equation}
\begin{aligned}
\Wc_p(\mu^N(t), \mu^N(s)) &\leq 2\Mc_1\biggl(V\biggl[\mu^N\bigl(\frac{x_{l+1} + x_l}{2}\bigr)\biggr]\biggr)\sqrt{|t-s|} \\
&\leq 2(B + 1)\sqrt{|t-s|}
\end{aligned}
\end{equation}
 Furthermore, by Lemma~\ref{le:fflowop}, for all $t,s \in \bigl[x_l, \frac{x_l + x_{l+1}}{2}\bigr]$, 
\begin{equation}
\begin{aligned}
\Wc_p(\mu^N(t), \mu^N(s)) &\leq 2\sqrt{|t-s|}\biggl(\|\ol{V}[\mu^N(x_l)]\|_{\infty} + \Mc_p(V[\mu^N(x_l)])^{\frac{1}{p}}\biggr) \\
&\leq 2\sqrt{|t-s|}\bigl(3(B+1) + B^{\frac{1}{p}}\bigr)
\end{aligned}
\end{equation}
Hence, there exists $\Bc > 0$ such that, whenever $|t-s| \leq \frac{1}{2N}$,  we have 
\begin{equation}
\Wc_p(\mu^N(t), \mu^N(s)) \leq \Bc \sqrt{|t-s|}.
\end{equation} 

\subsubsection*{Step 2:}

In the second and most technical step, we prove that the Average Flow Approximation Series from Definition~\ref{de:AFAS} satisfies (\ref{eq:MDEd}) in the limit. 

Take $N \in \Nz$, $s \in [0,T]$ and denote by $l^N(s) \in \llbracket 0, I_N-1 \rrbracket$ the maximal integer satisfying $x_{l^N(s)} \leq s$ and $\ph \in \Cc_c^{\infty}(M)$.  Then, 
\begin{equation}
\int_M \ph \der (\mu^N(x_{l^N(s)}) - \mu^N(0)) = \sum_{i = 0}^{l^N(s)-1} \int_M \ph \der (\mu^N(x_{i+1}) - \mu^N(x_i)).
\end{equation}
Now, 
\begin{equation}
\begin{aligned}
\int_M &\ph \der (\mu^N(x_{i+1}) - \mu^N(x_i)) \\ &= \int_M \ph \der (e^V_f(f^V_{x_{i+1} - x_i}(\mu^N(x_i))) - f^V_{x_{i+1} - x_i}(\mu^N(x_i)))\\&\qquad + \int_M \ph \der (f^V_{x_{i+1} - x_i}(\mu^N(x_i)) - \mu^N(x_i)) \\
&= \int_M \ph\biggl(\rh^{\ol{V}[f^V_{x_{i+1} - x_i}(\mu^N(x_i))]}_{x_{i+1} - x_i}(x)\biggr) - \ph(x) \der f^V_{x_{i+1} - x_i}(\mu^N(x_i)) \\
&\qquad+ \int_M \int_{\Xct(M)} \ph\biggl(\rh^{\ol{X}[\mu^N(x_i)]}_{\sqrt{x_{i+1} - x_i}}(x)\biggr) - \ph(x) \der V[\mu^N (x_i)](X) \der\mu^N (x_i)(x).
\end{aligned}
\end{equation}
Now, note that for all $X \in \Xct(M)$, $d_{\infty}(\rh^X_t, \id) \leq t \| X\|_{\infty}$. Hence $\ph \circ \rh^X_t$ is compactly supported. Furthermore, there exists a function $K: \Xct(M) \to \Rz$ bounded by $\ol{B}$ such that
\begin{equation}\label{eq:Taylor-dev}
\begin{aligned}
\ph(\rh^X_t(x)) - \ph(x) &= \ph(\rh^X_t(x)) - \ph(\rh^X_0(x)) \\
&= t \Lc_X \ph(x) + \tfrac{1}{2} t^2 \Lc^2_X \ph (x) + K(X)\|\ph\|_{\Cc^3}t^{\min(3,p)}\|X\|_{W^{2,\infty}}^{\min(3,p)}.
\end{aligned}
\end{equation}
Since $\Mc_p(V[\mu^N (x_i)]) < B$, for all $x\in M$,
\begin{equation}
\begin{aligned}\label{eq:intermstep}
&\int_{\Xct(M)} \ph(\rh^{\ol{X}[\mu^N(x_i)]}_{\sqrt{x_{i+1} - x_i}}(x)) - \ph(x) \der V[\mu^N (x_i)](X)  \\
&= \int_{\Xct(M)} \sqrt{x_{i+1} - x_i} \Lc_{\ol{X}[\mu^N(x_i)]} \ph(x) + \tfrac{1}{2} (x_{i+1} - x_i) \Lc^2_{\ol{X}[\mu^N(x_i)]} \ph (x) \\&\qquad+ K(X)\|\ph\|_{\Cc^3}(x_{i+1} - x_i)^{\frac{\min(3,p)}{2}}\|{\ol{X}[\mu^N(x_i)]}\|_{W^{2,\infty}}^{\min(3,p)}\der V[\mu^N (x_i)](X) \\
&=  \int_{\Xct(M)} \sqrt{x_{i+1} - x_i} \Lc_{\ol{X}[\mu^N(x_i)]} \ph(x) \\&\qquad\qquad+ \tfrac{1}{2} (x_{i+1} - x_i) \Lc^2_{\ol{X}[\mu^N(x_i)]} \ph (x) \der V[\mu^N (x_i)](X) \\
&\qquad+ (x_{i+1} - x_i)^{\frac{\min(3,p)}{2}}\!\! \int_{\Xct(M)}\!\!\!\!\!\!\!\!K(X)\|\ph\|_{\Cc^3}\|{\ol{X}[\mu^N(x_i)]}\|_{W^{2,\infty}}^{\min(3,p)} \der V[\mu^N (x_i)](X).
\end{aligned}
\end{equation}
Now, 
\begin{equation}
\begin{aligned}
\int_{\Xct(M)} &\Lc_{\ol{X}[\mu^N(x_i)]} \ph(x) \der V[\mu^N (x_i)](X) \\&= \int_{\Xct(M)} \der_x\ph(\ol{X}[\mu^N(x_i)](x)) \der V[\mu^N (x_i)](X)\\ &= \int_{\Xct(M)} \der_x\ph(v - \ol{V}[\mu^N(x_i)](x)) \der (\pi_x)_* V[\mu^N (x_i)](v)\\ &= \der_x\ph\biggl(\int_{\Xct(M)} v \der (\pi_x)_* V[\mu^N (x_i)](v) - \ol{V}[\mu^N(x_i)](x)\biggr) = 0,
\end{aligned}
\end{equation}
by Lemma~\ref{le:intlin} and the definition of $\ol{V}[\mu^N(x_i)]$. 
Hence, the first term in (\ref{eq:intermstep}) vanishes and after summation, we get
\begin{equation}
\begin{aligned}
&\biggl|\int_M \ph \der (\mu^N(\mu^N(x_{l^N(s)})) - \mu^N(0)) \\
&-\sum_{i = 0}^{l^N(s)-1} \biggl[\int_{x_i}^{x_{i+1}}\int_M \int_{\Xct(M)} \tfrac{1}{2} \Lc^2_{\ol{X}[\mu^N(x_i)]} \ph (x) \der V[\mu^N (x_i)](X) \der\mu^N (x_i)(x) \der t\\
&+ \int_{x_i}^{x_{i+1}} \int_M \Lc_{\ol{V}\bigl[\mu^N\bigl(\frac{x_i + x_{i+1}}{2}\bigr)\bigr]}(\ph) \der\mu^N \bigl(\frac{x_i + x_{i+1}}{2}\bigr)(x) \der t\biggr]\biggr|\\
&\qquad \leq  C \|\ph\|_{\Cc^3}x_{i_{l^N(s)}}(\max_{i}\bigl(x_{i+1} - x_i\bigr))^{\frac{\min(1, p-2)}{2}} \leq C \|\ph\|_{\Cc^3}x_{i_{l^N(s)}}\frac{1}{N^{\frac{\min(1, p-2)}{2}}}.
\end{aligned}
\end{equation}
for some constant $C > 0$. We rewrite
\begin{equation}\label{eq:pieceeq}
\begin{aligned}
&\sum_{i = 0}^{l^N(s)-1} \biggl[\int_{x_i}^{x_{i+1}}\int_M \int_{\Xct(M)} \tfrac{1}{2} \Lc^2_{\ol{X}[\mu^N(x_i)]} \ph (x) \der V[\mu^N (x_i)](X) \der\mu^N (x_i)(x) \der t\\
&\qquad+ \int_{x_i}^{x_{i+1}} \int_M \Lc_{\ol{V}\bigl[\mu^N\bigl(\frac{x_i + x_{i+1}}{2}\bigr)\bigr]}(\ph) \der\mu^N \bigl(\frac{x_i + x_{i+1}}{2}\bigr)(x) \der t\biggr]\\
&= \int_0^{x_{i_{l^N(s)}}} \sum_{i = 0}^{l^N(s)-1} \int_M \int_{\Xct(M)}   \1_{[x_i, x_{i+1}]}(t) \bigl(\tfrac{1}{2} \Lc^2_{\ol{X}[\mu^N(x_i)]} \ph (x) \bigr)\\&\qquad\qquad\qquad\qquad\qquad\qquad\qquad\qquad\qquad \der V[\mu^N (x_i)](X) \der\mu^N (x_i)(x) \der t\\
&+ \int_0^{x_{i_{l^N(s)}}} \sum_{i = 0}^{l^N(s)-1}\!\! \int_{M}  \1_{[x_i, x_{i+1}]}(t) \biggl(\Lc_{\ol{V}\bigl[\mu^N\bigl(\frac{x_i + x_{i+1}}{2}\bigr)\bigr]}(\ph) \biggr)  \der\mu^N \bigl(\frac{x_i + x_{i+1}}{2}\bigr)(x) \der t.
\end{aligned}
\end{equation}
Finally, we will estimate the difference of (\ref{eq:pieceeq}) with 
\begin{equation}
\begin{aligned}
\sum_{i = 0}^{l^N(s)-1} &\biggl[\int_{x_i}^{x_{i+1}}\int_M \int_{\Xct(M)} \tfrac{1}{2} \Lc^2_{\ol{X}[\mu^N(t)]} \ph (x) \der V[\mu^N (t)](X) \der\mu^N (t)(x) \der t \\ 
&+ \int_{x_i}^{x_{i+1}}\int_M \Lc_{\ol{V}[\mu^N(t)]} \der\mu^N (t)(x) \der t\biggr]\\
&= \int_0^{x_{i_{l^N(s)}}} \int_M \int_{\Xct(M)} \tfrac{1}{2} \Lc^2_{\ol{X}[\mu^N(t)]} \ph (x) \der V[\mu^N (t)](X) \der\mu^N (t)(x) \der t \\
&+ \int_0^{x_{i_{l^N(s)}}} \int_M \Lc_{\ol{V}[\mu^N(t)]} \der\mu^N (t)(x) \der t.
\end{aligned}
\end{equation}
We will give a pointwise estimate of the difference and then argue by dominated convergence. In order to keep the length of the proof reasonable, we will not detail all of those straightforward computations. 

%First, by assumption, the $V[\mu]_{\mu \in \Wc_p(M)}$ have uniformly bounded $p$--th moments. Hence, 
%\begin{equation}\label{eq:innerintbound}
%\begin{aligned}
%\biggl|\int_{\Xct(M)}& \tfrac{1}{2} \Lc^2_{\ol{X}[\mu^N(t)]} \ph (x) \der (V[\mu^N(t)] - V[\mu^N (i\de^t_N)])(X) \biggr| \\ 
%&\leq \int_{\Xct(M)} \tfrac{1}{2} \| \ol{X}[\mu^N(t)] \|_{W^{1,\infty}(M)}^2 (\| \nabla \ph \|_\infty + \| \nabla^2 \ph \|_\infty)\\
%&\qquad\qquad\cdot \der |V[\mu^N(t)] - V[\mu^N (i\de^t_N)]|(X) \\
%&\leq \tfrac{1}{2} (\| \nabla \ph \|_\infty + \| \nabla^2 \ph \|_\infty)
%\bigl(\Mc_2(V[\mu^N(t)])\\
%&\qquad\qquad+\Mc_2(V[\mu^N (i\de^t_N)])+12\Mc_1(V[\mu^N(t)]+V[\mu^N (i\de^t_N)])^2\bigr) \\
%&\leq K_1(\| \nabla \ph \|_\infty + \| \nabla^2 \ph \|_\infty),
%
%&\leq \| \ph \|_{\Cc^2(M)}^2\bigl(\Mc_2(V[\mu^N(t)] - V[\mu^N (i\de^t_N)]) \\&\qquad\qquad+ \|\ol{V}[\mu^N(t)]\|_{W^{1,\infty}(M)}\Mc_1(V[\mu^N(t)] - V[\mu^N (i\de^t_N)]\bigr) \\
%&\leq \| \ph \|_{\Cc^2(M)}^2\bigl(\Mc_2(V[\mu^N(t)] - V[\mu^N (i\de^t_N)]) \\&\qquad\qquad+ 3(B+1)\Mc_1(V[\mu^N(t)] - V[\mu^N (i\de^t_N)]\bigr).
%\end{aligned}
%\end{equation}
%for some $K_1>0$.

Let $\ep > 0$. Since the $p$--th moments of $V$ are uniformly bounded, they are tight in the $2$nd moment by Lemma~\ref{le:q-tight}. Hence, there exists $R>0$ such that for all $\mu \in \Wc_p(M)$ and for $\Xct_{\geq R}(M):=\Xct(M)\setminus B(0,R)$, 
\begin{equation}
\int_{\Xc^{2,\infty}_{\geq R}(M)} \|X\|_{W^{2,\infty}}^2\,\der V[\mu(t)](X) < \ep.
\end{equation}
Now for all $x \in M$ the function $\ol{\ph}(X):=X\mapsto \tfrac12\Lc^2_{\ol{X}[\mu]} \ph (x)$ is $C R(\|\nabla\ph\|_\infty+\|\nabla^2\ph\|_\infty)$--Lipschitz on $B(0,R)$ for some constant $C$ independent of $\mu$, $R$, $x$, and $\ph$. Let $\ps$ be the radial extension of $\ol{\ph}(X)|_{B(0,R)}$. Then $\ps$ is Lipschitz with the
same Lipschitz constant as $\ol{\ph}$ on $B(0,R)$.
W.l.o.g.\ $R>1$.
\begin{equation}\label{eq:innerintbound0}
\begin{aligned}
\biggl|&\int_{\Xct_{\geq R}(M)} \tfrac{1}{2} \Lc^2_{\ol{X}[\mu^N(t)]} \ph (x)-\ps(X) \der (V[\mu^N(t)] - V[\mu^N (x_i)])(X) \biggr| \\ 
	&\,\leq \int_{\Xct_{\geq R}(M)}\bigl|\tfrac{1}{2} \Lc^2_{\ol{X}[\mu^N(t)]} \ph (x)-\ps(X)\bigr|\der(V[\mu^N(t)] + V[\mu^N (x_i)])(X) \\
	&\,\leq \int_{\Xct_{\geq R}(M)}\tfrac{1}{2} \|X\|_{W^{2,\infty}}^2(\|\nabla\ph\|_\infty+\|\nabla^2\ph\|_\infty)\der(V[\mu^N(t)] + V[\mu^N (x_i)])(X)\\
	&\,<\frac{1}{2}(\|\nabla\ph\|_\infty+\|\nabla^2\ph\|_\infty)\ep.
\end{aligned}
\end{equation}
On the other hand, for all $x\in M$, by the Monge--Kantorovich-duality~\ref{th:weakkantorovichduality},
\begin{equation}\label{eq:innerintbound1}
\begin{aligned}
\biggl|\int_{\Xct(M)} &\ps(X) \der (V[\mu^N(t)] - V[\mu^N (x_i)])(X) \biggr| \\
	&\leq C R(\|\nabla\ph\|_\infty+\|\nabla^2\ph\|_\infty) \Wc_1\bigl(V[\mu^N(t)], V[\mu^N (x_i)]\bigr).
\end{aligned}
\end{equation}
Now, $V$ is uniformly continuous, and for all $N \in \Nz$, $\Wc_p(\mu^N(t), \mu^N{x_{i_l^N(t)}}) \leq \Bc \sqrt{x_{i_l^N(t) + 1} -  x_{i_l^N(t)}}$. Hence, $$\Wc_p(V[\mu^N(t)],V[\mu^N{x_{i_l^N(t)}}]) \leq \widetilde{\om}^V(\Bc \sqrt{x_{i_l^N(t) + 1} -  x_{i_l^N(t)}}),$$ where $\widetilde{\om}^V$ is the modulus of continuity of $V$. Now, Lemma~\ref{pr:wassersteinidentity} and (\ref{eq:innerintbound1}) imply that
\begin{equation}\label{eq:innerintbound2}
\begin{aligned}
\biggl|\int_{\Xct(M)} \ps(X) \der (V[\mu^N(t)] &- V[\mu^N (x_{i_l^N(t)})])(X) \biggr| \\
	&\leq C R(\|\nabla\ph\|_\infty+\|\nabla^2\ph\|_\infty) \widetilde \om^V(\frac{\Bc}{\sqrt{N}}).
\end{aligned}
\end{equation}
Thus, we conclude from (\ref{eq:innerintbound0}) and (\ref{eq:innerintbound2}) that there exists $N_0\in\Nz$ such 
that for all $x\in M$ and all $N\geq N_0$
$$
\begin{multlined}
	\biggl|\int_{\Xct(M)} \tfrac{1}{2} \Lc^2_{\ol{X}[\mu^N(t)]} \ph (x) \der (V[\mu^N(t)] - V[\mu^N (x_{i_l^N(t)})])(X) \biggr|\\<\ep(\|\nabla\ph\|_\infty+\|\nabla^2\ph\|_\infty).
\end{multlined}
$$
Thus, the function
$$x\mapsto\int_{\Xct(M)}\tfrac{1}{2} \Lc^2_{\ol{X}[\mu^N(t)]} \ph (x) \der (V[\mu^N(t)] - V[\mu^N (x_{i_l^N(t)})])(X)$$
uniformly converges to $0$ for $N\to\infty$.
%\begin{equation}
%\begin{aligned}
%\| \ph \|_{\Cc^2(M)}^2&\bigl(\Mc_2(V[\mu^N(t)] - V[\mu^N (i\de^t_N)]) + 3(B+1)\Mc_1(V[\mu^N(t)] - V[\mu^N (i\de^t_N)]\bigr) \\ &\leq K\| \ph \|_{\Cc^2(M)}^2\om^V(\sqrt{\de^t_N})^{\frac{1}{p}}.
%\end{aligned}
%\end{equation}
This implies that there exists a constant $K_1>0$ such that
\begin{equation}\label{eq:diff1}
\begin{aligned}
&\biggl|\int_0^{x_{i_{l^N(s)}}} \sum_{i = 0}^{I_N-1} \int_M \int_{\Xct(M)} \1_{[x_i, x_{i+1}]}(t) \frac{1}{2} \Lc^2_{\ol{X}[\mu^N(t)]} \ph (x)\\&\qquad\qquad \der (V[\mu^N (x_i)](X) -  V[\mu^N (t)])(X) \der\mu^N (x_i)(x) \der t\biggr|\\ &\qquad\qquad\qquad\qquad\qquad\qquad\qquad\leq K_1(\| \nabla \ph \|_\infty + \| \nabla^2 \ph \|_\infty)|s|\om^V(\frac{1}{\sqrt{N}}),
\end{aligned}
\end{equation}
for a suitable function $\om^V=_0o(1)$.

Let $x,y\in M$. Then
\begin{equation}\label{eq:diff2.A}
\begin{aligned}
\biggl|&\int_{\Xct(M)} \tfrac{1}{2}\bigl( \Lc^2_{\ol{X}[\mu^N(t)]} \ph (x) -\Lc^2_{\ol{X}[\mu^N(t)]} \ph (y)\bigr) \der V[\mu^N (t)](X)\biggr|\\
	&\quad\leq C'\int_{\Xct(M)} d(x,y)\|X\|_{W^{2,\infty}}^2(\| \nabla \ph \|_\infty + \| \nabla^2 \ph \|_\infty+\|\nabla^3\ph\|_\infty)\\
	&\qquad\qquad\qquad\qquad\qquad\qquad\qquad\qquad\qquad\cdot\der V[\mu^N (t)](X)\\
	&\quad\leq C'(\| \nabla \ph \|_\infty + \| \nabla^2 \ph \|_\infty+\|\nabla^3\ph\|_\infty)(B + 1)d(x,y).
\end{aligned}
\end{equation}

Hence, $$x \mapsto \int_{\Xct(M)} \frac{1}{2} \Lc^2_{\ol{X}[\mu^N(t)]} \ph (x) \der V[\mu^N (t)](X)$$ is Lipschitz with constant $K_2(\| \nabla \ph \|_\infty + \| \nabla^2 \ph \|_\infty+\|\nabla^3\ph\|_\infty)$ for some $K_2 > 0$. Again, Lemma~\ref{le:q-tight}, the Monge--Kantorovic duality  and the proof of Proposition~\ref{pr:wassersteinidentity} yield 
\begin{equation}\label{eq:diff2}
\begin{aligned}
\biggl|\int_0^{x_{i_{l^N(s)}}} &\sum_{i = 0}^{l-1} \int_M \int_{\Xct(M)}   \1_{[x_i, x_{i+1}]}(t)\frac{1}{2} \Lc^2_{\ol{X}[\mu^N(t)]} \ph (x) \\
	&\cdot \der V[\mu^N (t)](X)\der(\mu^N (x_i)- \mu^N(t))(x)  \der t\biggr| \\
	&\qquad\leq K_2\Bc(\| \nabla \ph \|_\infty + \| \nabla^2 \ph \|_\infty+\|\nabla^3\ph\|_\infty)\frac{|s|}{\sqrt{N}}.
\end{aligned}
\end{equation}
 Finally, (\ref{eq:LDclassic}) yields for all $i$ and all $t\in [x_i,x_{i+1}]$, 
\begin{equation}
|\Lc^2_{\ol{V}[\mu^N(x_i)]}\ph - \Lc^2_{\ol{V}[\mu^N(t)]}\ph| \leq 3(B+1)\|\ol{V}[\mu^N(x_i)] - \ol{V}[\mu^N(t)]\|_{W^{1,\infty}(M)} \|\ph\|_{\Cc^2}.
\end{equation}
 Thus, the last point of Lemma~\ref{le:eflowop} and a similar argument as those above yield that there exists $K_3 > 0$ such that 
\begin{equation}\label{eq:diff3}
\begin{aligned}
\biggl|\int_0^{x_{i_{l^N(s)}}} &\sum_{i = 0}^{l-1} \int_M \int_{\Xct(M)}   \1_{[x_i, x_{i+1}]}(t) \frac{1}{2}( \Lc^2_{\ol{X}[\mu^N(t)]} -  \Lc^2_{\ol{X}[\mu^N(x_i)]})\ph (x)\\&\cdot \der V[\mu^N (x_i)](X) \der\mu^N (x_i)(x) \der t\biggr|\\
	&\quad\leq K_3(\| \nabla \ph \|_\infty + \| \nabla^2 \ph \|_\infty)\widetilde{\om}^V(\frac{\Bc}{\sqrt{N}})|s|.
\end{aligned}
\end{equation}
Define $\om: x\mapsto x + \widetilde{\om}^V(x) + \om^V(x)$. Now, since $\lim_{N \to \infty} x_{i_l^N(s)} = s$ and the integrand is uniformly bounded, putting together (\ref{eq:diff1}), (\ref{eq:diff2}) and (\ref{eq:diff3}) we conclude that there exist $K_4 > 0$ such that
\begin{equation}
\begin{aligned}
\biggl|&\int_0^s \int_M \int_{\Xct(M)} \frac{1}{2} \Lc^2_{\ol{X}[\mu^N(t)]} \ph (x) \der V[\mu^N (t)](X) \der\mu^N (t)(x) \der t \\ &\qquad-\sum_{i = 0}^{l^N(s)-1} \int_{x_i}^{x_{i+1}}\int_M \int_{\Xct(M)} \frac{1}{2} \Lc^2_{\ol{X}[\mu^N(x_i)]} \ph (x)\\&\qquad\qquad\qquad\qquad\qquad\qquad \der V[\mu^N (x_i)](X) \der\mu^N (x_i)(x) \der t\biggr| \\
&\qquad\qquad\qquad\qquad\leq K_4 (\| \nabla \ph \|_\infty + \| \nabla^2 \ph \|_\infty+\|\nabla^3\ph\|_\infty)\om(\frac{\Bc}{\sqrt{N}})|s|.
\end{aligned}
\end{equation}
An easier, albeit fairly similar estimate shows that there exists $K_5 > 0$ such that
\begin{equation}
\begin{aligned}
&\biggl|\sum_{i = 0}^{l^N(s)-1} \int_{x_i}^{x_{i+1}} \int_M \Lc_{\ol{V}[\mu^N\bigl(\frac{x_i + x_{i+1}}{2}\bigr)]}(\ph) \der\mu^N \bigl(\frac{x_i + x_{i+1}}{2}\bigr)(x) \der t \\&\qquad\qquad- \int_0^s \int_M \Lc_{\ol{X}_{V[\mu(t)]}}(\ph) \der \mu(t) \der t\biggr| \\
&\qquad\qquad\qquad\qquad\qquad\qquad\qquad\leq K_5 (\| \nabla \ph \|_\infty + \| \nabla^2 \ph \|_\infty)\frac{\Bc}{\sqrt{N}}|s|.
\end{aligned}
\end{equation}

Therefore, since $\Wc_p(\mu^N(s), \mu^N(x_{i_l^N(s)})) \leq \frac{\Bc}{\sqrt{N}}$ and $\lim_{N \to \infty} x_{i_l^N(s)} = s$, there exists $\Kc > 0$ such that
\begin{equation}\label{eq:limsol}
\begin{aligned}
\biggl|\int_M \ph &\der (\mu^N(s) - \mu^N(0)) \\&- \int_0^s \int_M \int_{\Xct(M)} \frac{1}{2} \Lc^2_{\ol{X}[\mu^N(t)]} \ph (x) \der V[\mu^N (t)](X) \der\mu^N (t)(x) \der t \\&-\int_0^s \int_M \Lc_{\ol{V}[\mu^N(t)]}(\ph) \der \mu^N(t) \der t\biggr| \\
&\qquad\leq \Kc (\| \nabla \ph \|_\infty + \| \nabla^2 \ph \|_\infty+\|\nabla^3\ph\|_\infty)\Om(\tfrac{1}{N})|s|.
\end{aligned}
\end{equation}
for some nondecreasing continuous $\Om$ with $\Om(0) = 0$, which concludes the second step.

\subsubsection*{Conclusion 1:}

We make the stronger assumptions of the first point. 

Corollary~\ref{co:pHolder} implies that the $\bigl(\mu^N\bigr)_{N \in \Nz}$ are uniformly $\frac{1}{p}$--Hölder continuous in the $p$--Wasserstein distance. Hence, in particular, Lemma~\ref{pr:wassersteinidentity} yields that the $\bigl(\mu^N\bigr)_{N \in \Nz}$ are uniformly $\frac{1}{p}$--Hölder continuous in the $q$--Wasserstein distance for $1\leq q < p$. Furthermore, 
Lemma~\ref{le:q-tight} implies that the $\bigl(\mu^N\bigr)_{N \in \Nz}$ are tight in the $q$--th moment. Thus, by Corollary~\ref{co:compacttight}, the $\bigl(\mu^N\bigr)_{N \in \Nz}$ are relatively compact in $\Wc_q(M)$. Therefore, the Theorem of Arzelà--Ascoli implies the existence of a subsequence $\mu^{M_k}$ which converges uniformly to a Hölder--continuous curve $\mu(t)$ with exponent $\frac{1}{p}$ in the $q$--Wasserstein distance.
In particular, by successive extractions, we obtain a subsequence $\bigl(\mu^{N_k}\bigr)_{k \in \Nz}$ satisfying $\Wc_{p -\frac{1}{k}}(\mu^{N_k}, \mu)_{\infty} \leq \frac{1}{k}$.

By Corollary~\ref{co:pmomunifLip}, the curves $t \mapsto \Mc_p(\mu^{M_k})$ are uniformly Lipschitz. In particular, they are uniformly bounded. Thus, the Theorem of Arzelà--Ascoli implies that there exists a subsequence $\mu^{N_l} \coloneqq \mu^{M_{k_l}}$ and a Lipschitz curve $t \mapsto M_p(t)$, such that $\biggl( t \mapsto \Mc_p\bigl(\mu^{N_l}_t\bigr)\biggr)$ uniformly converges to $t \mapsto M_p(t)$. Now, for all $k \in \Nz$,
\begin{equation}
\Mc_p(\mu^{N_l}(t)) = \int_0^{\infty} (\mu^{N_l}(t))(\{d(x_0,x)^p \geq \la\}) \der \la,
\end{equation}
and by narrow convergence, for all $\la > 0$, 
\begin{equation}
\lim_{k \to \infty} (\mu^{N_l}(t))(\{d(x_0,x)^p \geq \la\}) = (\mu(t))(\{d(x_0,x)^p \geq \la\}).
\end{equation}
Thus, Fatou's Lemma implies $\Mc_p(\mu(t)) \leq \liminf_{k \to \infty} \Mc_p(\mu^{N_l}(t))$ and, hence, for all $t \in [0,T]$, $\mu(t) \in \Wc_p(M)$ and
$\Mc_p(\mu(t)) \leq M_p(t)$.

Let $\ep > 0$ and $B > 0$ such that for all $t \in [0,T]$, $M_p(t) \leq B$. Lemma~\ref{le:momentpq} yields that there exists $L \in \Nz$, such that for all $l_1, l \geq L$, $$\biggl\|\biggl(t \mapsto  \bigl|\Mc_p(\mu^{N_l}(t)) - \Mc_{p - \frac{1}{l_1}}(\mu^{N_l}(t))\bigr|\biggr)\biggr\|_{\infty} \leq \ep,$$ and $\bigl| B^{1 - \frac{1}{Lp}} - B\bigr| \leq \ep$. W.l.o.g. $\frac{1}{L} \leq \ep$ and for all $l \geq L$, $$\biggl\|\biggl(t \mapsto  \bigl|\Mc_p(\mu^{N_l}(t)) - M_p(t)\bigr|\biggr)\biggr\|_{\infty}\leq \ep.$$
Let $t \in [0,T]$. There exists $L_1$ such that for all $l \geq L_1$, 
\begin{equation}
    \biggl|\Mc_{p - \frac{1}{L}}(\mu(t)) - \Mc_{p - \frac{1}{L}}(\mu^{N_l}(t))\biggr| \leq \ep.
\end{equation}
Then, it follows for all $l \geq \max(L,L_1)$ that
\begin{equation}
	\begin{aligned}
		\Mc_p(\mu^{N_l}(t)) &\geq M_p(t) - \ep  \geq \Mc_p(\mu(t)) - \ep \\&\geq \Mc_{p - \frac{1}{L}}(\mu(t))^{\frac{p}{p - \frac{1}{L}}} - \ep\\
		& \geq \Mc_{p - \frac{1}{L}}(\mu(t)) - 2\ep \geq \Mc_{p - \frac{1}{L}}(\mu^{N_l}(t)) - 3\ep \\&\geq \Mc_{p}(\mu^{N_l}(t)) - 4\ep.
	\end{aligned}
\end{equation}
Hence, for all $t \in [0,T]$, $\lim_{l \to \infty} \Mc_p(\mu^{N_l}(t)) = \Mc_p(\mu(t))$ and, thus, $\mu^{N_l}(t)$ converges to $\mu(t)$ with respect to the $p$--Wasserstein distance by Theorem~\ref{th:Wassersteinconvergence}. In particular, $\bigl(\mu^{N_l}(t)\bigr)$ is relatively compact in $\Wc_p(M)$. Using Arzelà--Ascoli one more time implies, therefore, that there exists a subsequence $\bigl(\mu^{{\Nc}_l}\bigr)$ of $\bigl(\mu^{N_l}\bigr)$, which uniformly converges to $\mu$ in the $p$--Wasserstein distance and $t \mapsto \mu(t)$ is Hölder--continuous with exponent $\frac{1}{p}$ in the $p$--Wasserstein distance.  Hence, $V[\mu^{{\Nc}_l}(t)]$ converges to $V[\mu(t)]$ with respect to $\Wc_p$ and $\ol{V}[\mu^{{\Nc}_l}(t)]$ converges to $\ol{V}[\mu(t)]$
with respect to $\|~\|_{C^1}$. Therefore, we conclude that $\mu$ solves (\ref{eq:MDEd}) with initial value $\mu_0$.

The end of the proof becomes more technical if we do not assume $(M,g)$ to be of bounded geometry.

\subsubsection*{Conclusion 2:}

\subsubsection*{Step 3:}

Next, we use (\ref{eq:limsol}) to extract a subsequence of $\mu^N$ that simply weak-*-converges to a weak-*-continuous
curve of positive Borel measures $\mu(t)$ with $|\mu(t)|\leq 1$. 

Similarly to (\ref{eq:limsol}), we get for all $t,s \in [0,T]$,
\begin{equation}\label{eq:limsolst}
\begin{aligned}
\biggl|\int_M \ph &\der (\mu^N(t) - \mu^N(s)) \\&- \int_s^t \int_M \int_{\Xct(M)} \frac{1}{2} \Lc^2_{\ol{X}[\mu^N(u)]} \ph (x) \der V[\mu^N (u)](X) \der\mu^N (u)(x) \der u \\&-\int_s^t \int_M \Lc_{\ol{V}[\mu^N(u)]}(\ph) \der \mu^N(u) \der u\biggr|\\&\qquad\leq K(\| \nabla \ph \|_\infty + \| \nabla^2 \ph \|_\infty+\|\nabla^3\ph\|_\infty)|t-s|\Om(\tfrac{1}{N}).
\end{aligned}
\end{equation}
Therefore, using the triangle inequality in (\ref{eq:limsolst})
\begin{equation}\label{eq:limsolstl}
\begin{aligned}
\biggl|\int_M \ph &\der (\mu^N(t) - \mu^N(s))\biggr| \\&\leq \biggl|\int_s^t \int_M \int_{\Xct(M)} \frac{1}{2} \Lc^2_{\ol{X}[\mu^N(u)]} \ph (x) \der V[\mu^N (u)](X) \der\mu^N (u)(x) \der u\biggr| \\&\qquad+\biggl|\int_s^t \int_M \Lc_{\ol{V}[\mu^N(u)]}(\ph) \der \mu^N(u) \der u\biggr|\\
 &\qquad+ K(\| \nabla \ph \|_\infty + \| \nabla^2 \ph \|_\infty+\|\nabla^3\ph\|_\infty)|t-s|\Om(\tfrac{1}{N}).
\end{aligned}
\end{equation}
Since
$$
\begin{aligned}
	\biggl|\int_M\int_{\Xct(M)} \frac{1}{2} \Lc^2_{\ol{X}[\mu^N(u)]} \ph (x) \der V[\mu^N (u)](X)&\der\mu^N(u)(x)\biggr|\\&\leq B'(\| \nabla \ph \|_\infty + \| \nabla^2 \ph \|_\infty)
\end{aligned}
$$
and
$$
	\biggl|\int_M \Lc_{\ol{V}[\mu^N(u)]} \ph (x) \der \mu^N (u)(x)\biggr|\leq B''\|\nabla\ph\|_{\infty},
$$
for some fixed constant $B', B''>0$, 
we conclude from (\ref{eq:limsolstl}) that
\begin{equation}\label{eq:limsolstf}
\begin{aligned}
&\biggl|\int_M \ph \der (\mu^N(t) - \mu^N(s))\biggr| \leq C\biggl((\| \nabla \ph \|_\infty + \| \nabla^2 \ph \|_\infty+\|\nabla^3\ph\|_\infty)|t-s|\Om(\tfrac{1}{N}) \\ &\qquad+ (\| \nabla \ph \|_\infty + \| \nabla^2 \ph \|_\infty)|t-s|\biggr),
\end{aligned}
\end{equation}
where $C$ is independent of $N$, $s$, $t$, and $\ph$. 
Hence, the family of functions $t\mapsto\int_M \ph\der \mu^N(t)$ is equicontinuous for all $\ph\in C^\infty_c(M)$. This observation will allow us to extract
by Banach-Alaoglu and a diagonal argument a subsequence of the $\mu^{N_K}$ that converges on all rational numbers $q\in\Qz$ to finite Borel measures $\mu^{N_K}(q)\to\mu(q)$. On the other hand, the Theorem of Arzel\'a--Ascoli yields that for all
$\ph\in C_c^\infty(M)$ the family $f_\ph^N(t):=\int_M\ph\der\mu^N(t)$ converges uniformly for all $t\in [0,T]$ to $f_\ph(t)$. 
Furthermore, by the weak-*-convergence on the rational numbers we get $f_\ph(q)=\int_M\ph\der\mu(q)$ for all $q\in\Qz$ for
all $\ph\in C^\infty_c(M)$.

Since for $q\in\Qz$ we have that $\mu(q)\in\Mc_+(M)$, we can extend $f_\ph(q)=\int_M\ph\der\mu(q)$ to all $\ph\in C_b(M)$. This
function is continuous in the trace topology on $[0,T]\cap\Qz$ for all $\ph\in C_b(M)$. So it can be continunously extended to a 
function $[0,T]\to\Mc_+(M)$. We define $\mu(t)$ as the functional $\ph\mapsto f_\ph(t)$. Now, one readily shows that this
functional is linear and continuous, hence a finite Borel measure by the Riesz representation theorem. By construction,
$t\mapsto\mu(t)$ is weak-*-continuous. Furthermore, again by equicontinuity and uniform convergence of $f^N_\ph$ to $f_\ph$, 
we get that $\mu^N(t)$ weak-*-converges to $\mu(t)$ for all $t\in[0,T]$ like above.

\subsubsection*{Step 4:}

In this step, we show that $\mu(t)$ is a probability measure for all $t \in [0,T]$ and that the $(\mu^N(t))_{N \in \Nz, t\in[0,T]}$ have uniformly bounded $p$--th moments. 

\subsubsection*{Step 4.1:}

First, we prove that the $\mu(t)$ are all probability measures. By (\ref{eq:limsolstl}),  we have for all $\ph \in \Cc^{\infty}_c(M)$,
\begin{equation}\label{eq:muNprobeq1}
	\begin{aligned}
	\biggl|\int_M\ph\der(\mu^N(t)&-\mu^N(s))\biggr| \leq |t-s|\bigl(K_1(\|\nabla^2\ph\|_\infty+\|\nabla\ph\|_\infty)\\
		&+ K_2\|\nabla\ph\|_\infty+K(\|\nabla \ph\|_\infty+\|\nabla^2 \ph\|_\infty+\|\nabla^3 \ph\|_\infty)\om(\de^t_N)\bigr)
	\end{aligned}
\end{equation}
for some $K_1, K_2 > 0$. Take $N\in\Nz$ and $t\in[0,T]$, let $\ep>0$. There exists $R(t)>0$ such that
\begin{equation}
\begin{aligned}
	&\mu^N(0)(M\setminus B(x_0,R(t))) \leq\ep \\
	&\mu^N(t)(M\setminus B(x_0,R(t)))\leq\ep.
\end{aligned}
\end{equation}
We choose $0\leq\ph\in C^\infty_c(M)$ such that
$\ph\equiv 1$ on $B(x_0,R(t))$ and both $\|\ph''\|_\infty\leq\ep$ and $\|\ph'\|_\infty\leq\ep$. Then
(\ref{eq:muNprobeq1}) implies
$$
	\biggl|\int_{M\setminus B(x_0,R(t))}\ph\der(\mu^N(t)-\mu^N(0))\biggr|\leq K_3|T|\ep,
$$
for some $K_3 > 0$ that does not depend on $N$ or $R(t)$.

Now, take $N\in\Nz$ such that
$$
	\biggl\|\int_M\ph\der\mu^N(t)-\int_M\ph\der\mu(t)\biggr\|_\infty\leq\ep.
$$
Then, it follows that
$$
	\begin{aligned}
		\mu(t)(\Rz^n)&\geq\int_M\ph\der\mu(t)\geq\int_M\ph\der\mu(0)-\ep-K_3|T|\ep-\ep\\
			&=\int_M\ph\der\mu^N(0)-\ep-K_3|T|\ep-\ep \geq 1-\ep(2+K_3|T|).
	\end{aligned}
$$
Thus, $\mu(t)$ is a probability measure.

\subsubsection*{Step 4.2:}

Now, we prove by induction that for all $n\in\Nz$ with $n\leq p$, $(\mu^N(t))_{N,t}$ has
uniformly bounded $n$--th moment.

For $n=1$, let $t \in [0,T]$ and choose $R>0$ such that
$$
	\int_{M\setminus B(x_0,R)}d(x,x_0)\der\mu^N(t)\leq\ep\qquad\text{and}\qquad
	\int_{M\setminus B(x_0,R)}d(x,x_0)\der\mu^N(0)\leq\ep.
$$
Let $0\leq\ph\in C^\infty_c(M)$ be such that $\ph(x)=d(x_0,x)$ in $B(x_0,R) \setminus B(x_0,1)$, and
$\|\nabla\ph\|_\infty\leq 1$, $\|\nabla^2\ph\|_\infty\leq 1$, $\|\nabla^3\ph\|_\infty \leq 1$ and $d(x_0,x) \leq \ph(x) \leq 1$ for all $x \in B(x_0,1)$. Such a function
exists for all $R>0$. Then,
\begin{equation}\label{eq:M1est}
	\begin{aligned}
		\Mc_1(\mu^N(i)) &= \int_{M}d(x,x_0)\der\mu^N(i) \geq \int_{M}\ph\der\mu^N(i) - 1\\
			&\geq \int_{B(x_0,R)}d(x_0,x)\der\mu^N(i) - 1 \geq\Mc_1(\mu^N(i))-\eps -1.
	\end{aligned}
\end{equation}
for $i \in \{0,t\}$. But
$$
	\biggl|\int_M\ph\der(\mu^N(t)-\mu^N(0))\biggr|\leq \Kc|t|,
$$
where $\Kc$ is independent of $N$ and $R$. Therefore,
$$
	\Mc_1(\mu^N(t)-\mu^N(0))\leq K|t|+2\ep + 2\leq KT+2\ep +2,
$$
and the $(\Mc_1(\mu^N(t)))_{N,t}$ are uniformly bounded.

Let us now suppose that we have proved the result for $k\in\Nz$ with $1 \leq k\leq p-1$. Now, we get from (\ref{eq:limsolst})
$$
	\begin{aligned}
          \biggl|\int_M\ph\der(\mu^N(t)-\mu^N(s))\biggr| &\leq \int_s^t\int_M K_1(|\nabla\ph(x)|+|\nabla^2\ph(x)|)\der\mu^N(u)(x)\,du \\
			&+ \int_s^t K_2|\nabla\ph(x)|\der\mu^N(u)(x)\,du \\
                        &+ |t-s|K\|\nabla^3\ph\|_\infty\om(\de^t_N).
	\end{aligned}
$$
Again, take $s<t \in [0,T]$, $N \in \Nz$ and $R>0$ such that
$$
\begin{aligned}
	&\int_{M\setminus B(x_0,R)}d(x,x_0)^{(k+1)}\der\mu^N(s)\leq\ep \\
	&\int_{M\setminus B(x_0,R)}d(x,x_0)^{(k+1)}\der\mu^N(t)\leq\ep,
\end{aligned}
$$
using that $k+1 \leq p$ and, hence, the $(\mu^N(t))$ are in $\Wc_k(M)$. 

Let $0\leq\ph\in C^\infty_c(M)$ be such that $\ph(x)=d(x_0,x)^{(k+1)}$ in $B(x_0,R)$, and
$\|\nabla \ph\|_\infty\leq R^k$, $\|\nabla^2 \ph\|_\infty\leq R^k$, and $\|\nabla^3 \ph\|\leq R^k$. Such a function exists for all $R>0$. Note that we do not have to regularize the function around $0$--- unlike in the last step.

Then, for all $u \in [s,t]$,
\begin{equation}
\int_M K_1(|\nabla\ph(x)|+|\nabla^2\ph(x)|)\der\mu^N(u)(x) \leq C \Mc_k(\mu^N(u)), 
\end{equation}
for some constant $C>0$, which does not depend on $R$, $N$ or $u$. Now, the induction hypothesis yields, that these $k$--th moments are uniformly bounded by some $K_k >0$. Thus, we argue as in (\ref{eq:M1est}) to conclude that 
$$
	\Mc_{k+1}(\mu^N(t)-\mu^N(s))\leq K_k|t-s|+2\ep\leq K_kT+2\ep,
$$
and, thus, prove the induction step.

Finally, we show the statement for the $p$--th moments. If $p \in \Nz$, we have shown in the induction proof above that the $p$--th moments are uniformly bounded. 

Otherwise, there exists $k\in \Nz$ with $k > p-1$, for which the statement is already proved. But since the $(\mu^N(t))_{N,t}$ have uniformly bounded $k$--th moments, they have in particular $(p-1)$--st moments bounded by some $K_{p-1} > 0$. Now, since $\mu_0 \in \Wc_p(M)$ and the $(\mu^N(t))_{N,t}$ are all in $\Wc_p(M)$ by Lemma~\ref{le:fflowop}, we proceed completely analogously as in the induction step above to prove that there exist $K_p > 0$ such that $\Mc_p(\mu^N(t)) \leq K_p$ for all $N,t$.

\subsubsection*{Step 5:}

Finally, we conclude that $t \mapsto \mu(t)$ is continuous in $\Wc_p(M)$ and that it satisfies (\ref{eq:MDEd}).

Take $t \in [0,T]$. By Step 4, $(\mu^{N_k}(t))_{k\in\Nz}$ has uniformly bounded $p$--th moments. Furthermore, we have shown that $\mu(t)$ is a probability measure, so $(\mu^{N_k}(t))_{k\in\Nz}$ converges narrowly to $\mu(t)$. Now, for all $k \in \Nz$,
\begin{equation}
\Mc_p(\mu^{N_k}(t)) = \int_0^{\infty} (\mu^{N_k}(t))(\{d(x_0,x)^p \geq \la\}) \der \la,
\end{equation}
and by narrow convergence, for all $\la > 0$, 
\begin{equation}
\lim_{k \to \infty} (\mu^{N_k}(t))(\{d(x_0,x)^p \geq \la\}) = (\mu(t))(\{d(x_0,x)^p \geq \la\}).
\end{equation}

Thus, Fatou's Lemma implies $\Mc_p(\mu(t)) \leq \liminf_{k \to \infty} \Mc_p(\mu^{N_k}(t))$ and, hence, $\mu(t) \in \Wc_p(M)$ and $\Mc_p(\mu(t)) \leq K_p$ for all $t\in[0,T]$.

Take $q < p$ such that $V$ is continuous for the $\Wc_q$--distance. Then, Proposition~\ref{pr:wassersteinidentity} and Lemma~\ref{le:q-tight} yield that $\mu(t)$ is continuous for the $\Wc_q$--topology on $\Wc_p(M)$ and for all $t \in [0,T]$, $\mu^{N_k}(t)$ converges to $\mu(t)$ in the $\Wc_q$--distance.
Hence, $V[\mu^N(t)]$ converges to $V[\mu(t)]$ with respect to $\Wc_p$ and $\ol{V}[\mu^{N_k}(t)]$ converges to $\ol{V}[\mu(t)]$
with respect to $\|~\|_{C^1}$. Therefore, we conclude that $\mu$ solves (\ref{eq:MDEd}) with initial value $\mu_0$. 

Therefore, the continuity of $\mu: [0,T] \to \Wc_p(M)$ is proved in Proposition~\ref{pr:WpHolder}. 
\end{proof}

\begin{re}
Now, we get Theorem~\ref{th:exfirst} as a corollary of Theorem~\ref{th:gencentlim}.
\end{re}

\begin{re}
Note that the Steps 4 and 5 in the proof of Theorem~\ref{th:gencentlim} are trivial if we work with a compact manifold $(M,g)$.
\end{re}

\subsection{Uniqueness results, Grönwall type estimates and the Strong Generalised Central Limit Theorem}

In this last section, we shall see that the solutions of (\ref{eq:MDEd}) are in general not unique. However, we will prove smooth solutions, as introduced in Definition~\ref{de:smoothsol}, are unique under suitable assumptions. 

\subsubsection{The linear equation and LAFAS}
First, we introduce a linear analogon to 
(\ref{eq:MDEd}).

\begin{defi}[LMDE]\label{def:MDEl}
Let $T>0$, $p\geq 2$ and $\mu_0 \in \Wc_p(M)$. Let $V: [0,T] \to \Wc_p(\Xc^{2,\infty}(M))$ be continuous with respect to the $p$--Wasserstein distance. Let $[0,T] \to \Pc(M)$, $t \mapsto \mu(t)$ be a path of measures. We say that $\mu(t)$ satisfies the linear measure differential equation \emph{(\ref{eq:MDEl})} for $V$ with initial value $\mu_0$ if $\mu(0) =\mu_0$, $\mu$ is narrowly continuous and for all $s\in [0,T]$ and all $\ph$ smooth and compactly supported on $M$, we have
\begin{equation}\label{eq:MDEl}
	\tag{LMDE}
\begin{aligned}
\int_M \ph \der \mu(s) - \int_M \ph \der \mu(0) = \int_0^s \int_M \square^{V_t} \ph \der \mu(t) \der t, 
\end{aligned}
\end{equation} 
where for all $t$
\begin{equation}
	\square^{V_t} \ph = \int_{\Xct(M)} \frac{1}{2}\Lc_{X-\ol{V_t}}^2(\ph) + \Lc_X(\ph) \der V_t(X).
\end{equation}
\end{defi}

We introduce an appropriate analogon of the average flow approximation series for (\ref{eq:MDEl}).

\begin{defi}[Linear Average Flow Approximation Series]\label{de:LAFAS}
Let $p \geq 2$, $T>0$,  $\mu_0 \in \Wc_p(M)$, $\nu: [0,T] \to \Wc_p(M)$ be continuous with respect to the $p$--Wasserstein distance and $V$ be a continuous probability vector--field on $\Wc_p(M)$. Let $\de>0$ and $\Pf(\de)$ be a partition of $[0,T]$ with maximal step--length $\pf(\Pf(\de)) \leq \de$. Let $N \coloneqq |\Pf(\de)|$ and index the elements of $\Pf(\de)$ in increasing order. Define
$$
\mu^{\Pf(\de)}(0) = \mu_0,
$$
and for all $l \in \llbracket 0,N-1 \rrbracket$,
$$
\mu^{\Pf(\de)}(x_{l+1}) = (e^{V[\nu_{x_l}]}_{x_{l+1} - x_l}\circ f^{V[\nu_{x_l}]}_{x_{l+1} - x_l})(\mu^{\Pf(\de)}(x_l)).
$$
Furthermore, define $\mu^{\Pf(\de)}(t)$ by 
$$
\mu^{\Pf(\de)}(t) = f^{V[\nu_{x_l}]}_{2(t - x_l)}(\mu^{\Pf(\de)}(x_l)),
$$
for all $t \in \bigl[x_l, \frac{x_{l+1} + x_l}{2} \bigr]$ and 
$$
\mu^{\Pf(\de)}(t) = \biggl(e^{V[\nu_{x_l}]}_{2\bigl(t - \frac{x_{l+1} + x_l}{2}\bigr)}\circ f^{V[\nu_{x_l}]}_{x_{l+1} - x_l}\biggr)\bigl(\mu^{\Pf(\de)}(x_l)\bigr),
$$
for all $t \in \bigl[ \frac{x_{l+1} + x_l}{2}, x_{l+1} \bigr]$.
\end{defi}
Also the definition of smooth solution is adapted immediately.
\begin{defi}[Smooth solution]\label{de:linsmoothsol}\strut\newline
 Let $p\geq 2$, $T>0$ and $V: [0,T] \to \Wc_p(\Xc^{2,\infty}(M))$ be continuous. A narrowly continuous curve $\nu: [0,T] \to \Wc_p(M)$ is called a \emph{smooth solution} of (\ref{eq:MDEl}) if it solves (\ref{eq:MDEl}) for $V$ and there exists a sequence $(\Pf_n)_{n\in\Nz}$ of partitions of $[0,T]$ with maximal step--length $\pf(\Pf_n) \to 0$ such that for all $n\in\Nz$, the LAFAS $\mu^{\Pf_n}$ converges pointwise narrowly to $\nu$.
\end{defi}
Now, we readily obtain an existence result for equation (\ref{eq:MDEl}) which is analogous to the Generalised Central Limit Theorem~\ref{th:gencentlim}.
\begin{theorem}[Generalised Central Limit Theorem---linear version]\label{th:lingencentlim}\strut\newline
Let $p>2$, $T>0$ and $\mu_0 \in \Wc_p(M)$. Let $\bigl( \Pf(\frac{1}{N}) \bigr)_{N \in \Nz}$ be a family of partitions of $[0,T]$. Let $\nu: [0,T] \to \Wc_p(M)$ be continuous with respect to the $p$--Wasserstein distance and $V$ be a continuous probability vector--field on $\Wc_p(M)$. Assume that $V$ has uniformly bounded $p$--th moments, e.g., there exists a constant $B>0$ such that for all $\mu \in \Wc_p(M)$, $\Mc_p(V[\mu]) < B$. 

\begin{enumerate}
	\item Assume that $(M,g)$ is of bounded geometry. Assume that there exists $R > 0$ such that for all $\mu \in \Wc_p(M)$, $V[\mu] \in \Wc_{p+2}(\Xc^{2,\infty}(M))$ and $\Mc_{p+2}(V[\mu]) \leq R^{p+2}$. Then, there exists a subsequence $\biggl(\mu^{\Pf(\frac{1}{N_k})}\biggr)_{k \in \Nz}$ of the Linear Average Flow Approximation Series $\biggl(\mu^{\Pf(\frac{1}{N})}\biggr)_{N \in \Nz}$ with initial value $\mu_0$ (See Definition~\ref{de:LAFAS}) which uniformly converges  with respect to the $p$--Wasserstein distance to a $\frac{1}{p}$--Hölder--continuous solution $\mu: [0,T] \to \Wc_p(M)$ of (\ref{eq:MDEl}) for $t \mapsto V[\nu(t)]$ with initial value $\mu_0$.
	
	\item Now, we only assume that there exists $q < p$ such that $V$ is continuous for the $\Wc_q$--distance.
	Then, there exists a subsequence $\biggl(\mu^{\Pf(\frac{1}{N_k})}\biggr)_{k \in \Nz}$ of the Linear Average Flow Approximation Series with initial value $\mu_0$ which converges to a solution $\mu: [0,T] \to \Wc_p(M)$ of (\ref{eq:MDEl}) for $t \mapsto V[\nu(t)]$ with initial value $\mu_0$. 

	The convergence of $\mu^{\Pf(\frac{1}{N_k})}$ to $\mu$ is pointwise in the $\Wc_{q}$--distance and $\mu(t)$ is continuous in $\Wc_{p}(M)$. 

\end{enumerate}
\end{theorem}
\begin{proof}
The proof is completely analogous to the proof of Theorem~\ref{th:gencentlim}. 
\end{proof}

\subsubsection{Grönwall type estimates and Hausdorff--continuity of the set of solutions}
Now, we prove two Grönwall type estimates for AFAS and LAFAS respectively for different initial values and identical initial values, but different probability vector--fields. 
\begin{theorem}\label{th:gronwallHausdorff}
	Let $p \geq 2$ and $T >0$. Let $\mu_0, \nu_0 \in \Wc_p(M)$. Let $\bigl( \Pf(\frac{1}{N}) \bigr)_{N \in \Nz}$ be a family of partitions of $[0,T]$. Assume:
	\begin{itemize}
	\item $(M,g)$ is of bounded geometry.
	\item $V$ is a probability vector--field which satisfies:
	\begin{itemize}
		\item There exists a constant $R > 0$ such that for all $\mu \in \Wc_p(M)$, $V[\mu] \in \Wc_{p+2}(\Xc^{2,\infty}(M))$ and $\Mc_{p+2}(V[\mu]) \leq R^{p+2}$.
		\item  $V$ is $L$--Lipschitz with respect to the $\Wc_p$--distance for some $L>0$.  
	\end{itemize}
	\end{itemize}
	Let $\ep >0$. Denote $\Gc(p,L,R,M)\coloneqq  C(M)2^{p-2}\bigl(9(p-1)L^2 + pR^2 + 6RL + R + 3L \bigr)$, where $C(M)$ is a constant depending only on $M$ and $C(\Rz^n) = 1$ for all $n \in \Nz$. Then, there exists $M \in \Nz$ such that for all $N \geq M$, $t \in [0,T]$, the AFAS $\mu^{\Pf(\frac{1}{N})}$ and $\nu^{\Pf(\frac{1}{N})}$ of Definition~\ref{de:LAFAS} with initial values $\mu_0$ and $\nu_0$ satisfy
    \begin{equation}
		\begin{aligned}
			\Wc_p(\mu^{\Pf\bigl(\frac{1}{N}\bigr)}(t), \nu^{\Pf\bigl(\frac{1}{N}\bigr)}(t)) \leq \Wc_p(\mu_0, \nu_0) e^{\bigl(\Gc(p,L,R) + \ep\bigr)t} + t\ep.
		\end{aligned}
	\end{equation}
	Let $p>2$. In particular, for every smooth solution $\mu$ of (\ref{eq:MDEd}) with initial value $\mu_0$, there exists a smooth solution $\nu$ of (\ref{eq:MDEd}) with initial value $\nu_0$ such that
	\begin{equation}
		\begin{aligned}
			\Wc_p(\mu(t), \nu(t)) \leq \Wc_p(\mu_0, \nu_0) e^{\Gc(p,L,R)t},
		\end{aligned}
	\end{equation}
	and vice--versa. Let $S_p(V,T,\mu_0)$ be the set of smooth solutions of (\ref{eq:MDEd}) in $\Wc_p(M)$, with initial value $\mu_0 \in \Wc_p(M)$ on $[0,T]$. In other words, the map
	\begin{equation}
		\begin{aligned}
			S_p: \Wc_p(M) &\to \Pz\biggl(\Cc^{0,\frac{1}{p}}\bigl([0,T], \Wc_p(M)\bigr)\biggr)\\
			\mu &\mapsto S_p(V,T,\mu_0),
		\end{aligned}
	\end{equation}
	is Lipschitz with respect to the Hausdorff--distance on $\Pz\biggl(\Cc^{0,\frac{1}{p}}\bigl([0,T], \Wc_p(M)\bigr)\biggr)$.
\end{theorem}
\begin{proof}
	For $(M,g) = (\Rz^n, g^{eucl})$, the result follows by immediate induction using (\ref{eq:estfrightVp}) and (\ref{eq:esterightVp}). For an arbitrary manifold with bounded geometry, we use induction in Lemma~\ref{le:fflowopman}. 
\end{proof}

\begin{lem}\label{le:Hausdorffvectfield}
	Assume $(M,g)$ is of bounded geometry. Let $p\geq2$, $T>0$, $R>0$ and $\mu_0 \in \Wc_p(M)$. Denote by 
	\begin{itemize}
		\item $S_p(V,T,\mu_0)$ the set of smooth solutions of (\ref{eq:MDEl}) in $\Wc_p(M)$ for a curve $V: [0,T] \to \Wc_{p+2}(\Xc^{2,\infty}(M))$, with initial value $\mu_0 \in \Wc_p(M)$ on $[0,T]$,
		\item $\Bc_p \coloneqq 2^{\frac{p+3}{2}} 3 \sqrt{p-1}$.
	\end{itemize}
	 Let $V, W: [0,T] \mapsto \Wc_{p+2}(\Xc^{2,\infty}(M))$ be two curves of vector--field probabilities satisfying $\Mc_{p+2}(V_t), \Mc_{p+2}(W_t) \leq R^{p+2}$ for all $t \in [0,T]$. Let $\bigl( \Pf(\frac{1}{N}) \bigr)_{N \in \Nz}$ be a family of partitions of $[0,T]$. Denote by $\bigl(\mu^{\Pf(\frac{1}{N})}\bigr)_{N \in \Nz}$ and $\bigl(\nu^{\Pf(\frac{1}{N})}\bigr)_{N \in \Nz}$ the respective (LAFAS) with initial value $\mu_0$ associated to $V$ and $W$ respectively. Then, there exists a constant $C(M) >0$ depending only on $M$, with $C(\Rz^n) = 1$ for all $n \in \Nz$, such that the following hold: 
	\begin{itemize}
		\item Suppose $p = 2$. There exists $M \in \Nz$, such that for all $N \geq M$ and for all $t \in \bigl[0, \frac{1}{2(R+1)^2}\bigr]$, 
		\begin{equation}
			\Wc_2\bigl(\mu^{\Pf\bigl(\frac{1}{N}\bigr)}(t), \nu^{\Pf\bigl(\frac{1}{N}\bigr)}(t)\bigr) \leq C(M)6 \Wc_2(V, W)_{\infty} \sqrt{t}.
		\end{equation}
		\item Suppose $p>2$. There exist $\Sf_p >0$, and for all $2 < q \leq 4$, an integer $M_q \in \Nz$, such that for all $N \geq M_q$ and for all $t \in \bigl[0, \frac{\Sf_p}{(R+1)^2}\bigr]$,
		\begin{equation}
			\Wc_p\bigl(\mu^{\Pf\bigl(\frac{1}{N}\bigr)}(t), \nu^{\Pf\bigl(\frac{1}{N}\bigr)}(t)\bigr) \leq  C(M)\Bc_p \Wc_p(V, W)_{\infty} t^{\frac{1}{q}}.
		\end{equation}
		\item In particular, for all $p > 2$, $\mu \in S_p(V,T,\mu_0)$, $\nu \in S_p(W,T,\mu_0)$, and $t \in \bigl[0, \frac{\Sf_p}{(R+1)^2}\bigr]$,
		\begin{equation}
			\Wc_p\bigl(\mu_t, \nu_t\bigr) \leq  C(M)\Bc_p \Wc_p(V, W)_{\infty} \sqrt{t}.
		\end{equation}
	\end{itemize}
	
	Let $p>2$. Furthermore, there exists a constant $\Kc >0$, depending only on $R$, $T$, $p$ and $M$, such that for all $2 < q \leq 4$, there exists $M_q \in \Nz$, such that for all $N \geq M_q$ and all $t \in [0,T]$,
	\begin{equation}
		\Wc_p(\mu^{\Pf\bigl(\frac{1}{N}\bigr)}(t), \nu^{\Pf\bigl(\frac{1}{N}\bigr)}(t)) \leq \Kc \Wc_p(V,W)_{\infty}t^{\frac{1}{q}},
	\end{equation}
	and
	\begin{equation}\label{eq:vectdistest}
		\Wc_p(\mu_t, \nu_t) \leq \Kc \Wc_p(V,W)_{\infty} \sqrt t.
	\end{equation}
	In particular, the map
	\begin{equation}
		\begin{aligned}
			\Sc_p: \Cc^0\bigl([0,T], \Wc_{p+2}(\Xc^{2,\infty}(M))\bigr) &\to \Pz\biggl(\Cc^{0,\frac{1}{p}}\bigl([0,T], \Wc_p(M)\bigr)\biggr)\\
			V &\mapsto S_p(V,T,\mu_0),
		\end{aligned}
	\end{equation}
	is locally Lipschitz--continuous with respect to the Hausdorff--distance on the power set $\Pz\biggl(\Cc^{0,\frac{1}{p}}\bigl([0,T], \Wc_p(M)\bigr)\biggr)$.
\end{lem}
\begin{proof}
	Let us first work with $(M,g) = (\Rz^n, g^{eucl})$ and $p=2$. Let $\ep>0$. Denote for all $N \in \Nz$, $(x_{N,i})_{i \in I_N} \coloneqq \Pf(\frac{1}{N})$ in increasing order. 
	
	\noindent
	\textbf{Proof of the first point:}
	First, we prove by induction that there exists $\ol{M} \in \Nz$, such that for all $t \in \bigl[0, \frac{1}{2(R+1)^2}\bigr]$ and all $N \geq \ol{M}$,
	\begin{equation}
		\Wc_2\bigl(\mu^{\Pf\bigl(\frac{1}{N}\bigr)}(t), \nu^{\Pf\bigl(\frac{1}{N}\bigr)}(t)\bigr)^2 \leq 36 \Wc_2(V, W)_{\infty}^2 t.
	\end{equation}
	
	\noindent
	\textbf{Initialisation:}
	First, (\ref{eq:estfrandV2}) and (\ref{eq:esterandV2}) in Lemma~\ref{le:fflowop2} yield that there exists $M \in \Nz$---depending on $\Wc_2(V,W)_{\infty}$---, such that for all $N \geq M$, 
	\begin{equation}
		\begin{aligned}
			\Wc_2\bigl(\mu^{\Pf\bigl(\frac{1}{N}\bigr)}(x_{N,1}), \nu^{\Pf\bigl(\frac{1}{N}\bigr)}(x_{N,1})\bigr)^2 \leq 18\bigl(1 + \ep\bigr)x_{N,1}\Wc_2(V,W)^2_{\infty}.
		\end{aligned}
	\end{equation}

	\noindent 
	\textbf{Hypothesis:}
	Now, let $i \in I_N$ such that $x_{N,i} \leq \frac{1}{2(R+1)^2}$ and suppose that $$\Wc_2\bigl(\mu^{\Pf\bigl(\frac{1}{N}\bigr)}(x_{N,i}), \nu^{\Pf\bigl(\frac{1}{N}\bigr)}(x_{N,i})\bigr)^2 \leq 36 x_{N,i}\Wc_2(V,W)^2_{\infty}.$$ 
	
	\noindent
	\textbf{Step:}
	Then, by \eqref{eq:estfrandV2}, there exists $\ol{M} \geq M$, such that for all $N \geq \ol{M}$,
	%\begin{equation}
	{\allowdisplaybreaks
		\begin{align}
			&\Wc_2\bigl(f^{V(x_{N,i})}_{x_{N,i+1} -x_{N,i}}(\mu^{\Pf\bigl(\frac{1}{N}\bigr)}(x_{N,i})), f^{W(x_{N,i})}_{x_{N, i+1} -x_{N,i}}(\nu^{\Pf\bigl(\frac{1}{N}\bigr)}(x_{N,i}))\bigr)^2\notag\\
			&\quad \leq \Wc_2\bigl(\mu^{\Pf\bigl(\frac{1}{N}\bigr)}(x_{N,i}), \nu^{\Pf\bigl(\frac{1}{N}\bigr)}(x_{N,i})\bigr)^2 \notag\\
			&\qquad + 2(1 + \ep)(x_{N,i+1} - x_{N,i})\biggl(9\Wc_2(V,W)^2_{\infty}\\
			&\qquad + 3R\Wc_2\bigl(\mu^{\Pf\bigl(\frac{1}{N}\bigr)}(x_{N,i}), \nu^{\Pf\bigl(\frac{1}{N}\bigr)}(x_{N,i})\bigr) \Wc_2(V,W)_{\infty}\biggr)\notag\\\notag\\
			&\quad \leq 36 x_{N,i}\Wc_2(V,W)^2_{\infty} \notag\\
			&\qquad +2(1 + \ep)\Wc_2(V,W)^2_{\infty} \bigl( 9 + 18 (R+1) \sqrt{x_{N,i}}\bigr)(x_{N, i +1} - x_{N,i})\notag\\
			&\quad \leq 36 x_{N,i+1}\Wc_2(V,W)^2_{\infty},\notag
		\end{align}
        }
	%\end{equation}
	since $x_{N,i} \leq \frac{1}{2(R+1)^2}$. Thus, we conclude by induction. 

	\noindent
	\textbf{Proof of the second point:}
	For $p > 2$, the statement follows similarly from equations \eqref{eq:estfrandVp} and \eqref{eq:esterandVp}. Let $2 < q \leq 4$.
	
	\noindent
	\textbf{Initialisation:}
	There exists $M_q \in \Nz$---depending on $\Wc_p(V,W)_{\infty}$---such that there exists $\Kc >0$ depending only on $p$, such that for all $N \geq M$,
	\begin{equation}
		\begin{aligned}
			&\Wc_p\bigl(\mu^{\Pf\bigl(\frac{1}{N}\bigr)}(x_{N,1}), \nu^{\Pf\bigl(\frac{1}{N}\bigr)}(x_{N,1})\bigr)^p \leq \Kc \Wc_p(V,W)^2_{\infty} x_{N,1}^{\frac{p}{2}}\\
			&\qquad \leq  \Bc_p^p \Wc_p(V,W)_{\infty}^p x_{N,1}^{\frac{p}{q}}.
		\end{aligned}
	\end{equation}

	\noindent 
	\textbf{Hypothesis:}
	Now, let $i \in I_N$ and suppose that $$\Wc_p\bigl(\mu^{\Pf\bigl(\frac{1}{N}\bigr)}(x_{N,i}), \nu^{\Pf\bigl(\frac{1}{N}\bigr)}(x_{N,i})\bigr)^p \leq \Bc_p^p x_{N,i}^{\frac{p}{q}}\Wc_p(V,W)^p_{\infty}.$$ 

	\noindent
	\textbf{Step:}
	One readily shows that the induction hypothesis implies that there exits $\ol{M_q} \geq M_q$, such that for all $N \geq \ol{M_q}$,
	\begin{equation}
		\begin{aligned}
			&\Wc_p\bigl(\mu^{\Pf\bigl(\frac{1}{N}\bigr)}(x_{N,i+1}), \nu^{\Pf\bigl(\frac{1}{N}\bigr)}(x_{N,i+1})\bigr)^p \leq \Wc_p(V,W)_{\infty}^p \biggl( \Bc_p^q x_{N,i}\\
			&\qquad + 2^{p} q \biggl(9(p-1)\Bc_p^{q-2}x_{N,i}^{\frac{q-2}{q}} + pR^2\Bc_p^qx_{N,i}\\
			&\qquad \quad + 6R \Bc_p^{q-1}x_{N,i}^{\frac{q-1}{q}} \biggr)(x_{N, i+1} - x_{N,i})\biggr)^{\frac{p}{q}}\\
			&\qquad \leq \Wc_p(V,W)_{\infty}^p \Bc_p^p \biggl(  x_{N,i}\\
			&\qquad + \biggl(\frac{1}{2}x_{N,i}^{\frac{q-2}{q}} + 2^{p+2} \bigl(  pR^2x_{N,i} + 6R x_{N,i}^{\frac{q-1}{q}} \bigr)\biggr)(x_{N, i+1} - x_{N,i})\biggr)^{\frac{p}{q}}.\\
		\end{aligned}
	\end{equation}
	In particular, there exists a constant $\Sf_p > 0$ depending only on $p$, such that if $x_{N,i} \leq \frac{\Sf_p}{(R+1)^2}$,
	\begin{equation}
		\Wc_p\bigl(\mu^{\Pf\bigl(\frac{1}{N}\bigr)}(x_{N,i+1}), \nu^{\Pf\bigl(\frac{1}{N}\bigr)}(x_{N,i+1})\bigr)^p \leq \Bc_p^p\Wc_p(V,W)_{\infty}^p x_{N, i+1}^{\frac{p}{q}}.
	\end{equation}

	\noindent
	\textbf{Proof of the global estimates for (LAFAS):}
	Now, another proof by induction using equations \eqref{eq:estfrandVp} and \eqref{eq:esterandVp} and a subsequent Grönwall--type argument imply the existence of $\Kc >0$, depending only on $R$, $T$, $p$ and $M$, such that for all $2 < q \leq 4$, there exists $M_q \in \Nz$, such that for all $N \geq M_q$ and all $t \in [0,T]$,
	\begin{equation}
		\Wc_p(\mu^{\Pf\bigl(\frac{1}{N}\bigr)}(t), \nu^{\Pf\bigl(\frac{1}{N}\bigr)}(t))_{\infty} \leq \Kc t^{\frac{1}{q}}\Wc_p(V,W)_{\infty}.
	\end{equation}

	\noindent
	\textbf{Proof of the estimates for smooth solutions:}
	Let $\mu \in S_p(V,T,\mu_0)$. By Theorem~\ref{th:lingencentlim}, there exists a sequence of partitions $\bigl(\Pf(\frac{1}{N})\bigr)_{N \in \Nz}$ such that $t \mapsto \mu(t)$ is a uniform limit of $\bigl(\mu^{\Pf(\frac{1}{N})}\bigr)_{N \in \Nz}$. 
	Since we assume $p>2$, Theorem~\ref{th:lingencentlim} yields that there exists an increasing sequence $N_k$ such that $\bigl(\nu^{\Pf(\frac{1}{N_k})}\bigr)_{N \in \Nz}$ converges uniformly to a $\nu \in S_p(W,T,\mu_0)$. 
	The uniform convergence implies that all the above estimates also hold or $\mu$ and $\nu$. In particular, for all $2<q\leq4$ and all $t \in \bigl[0, \frac{\Sf_p}{(R+1)^2}\bigr]$,
	\begin{equation}
		\begin{aligned}
			\Wc_p(\mu_t, \nu_t) \leq \Bc_p t^{\frac{1}{q}}\Wc_p(V,W)_{\infty}.
		\end{aligned}
	\end{equation} 
	Therefore,
	\begin{equation}
		\begin{aligned}
			\Wc_p(\mu_t, \nu_t) \leq \Bc_p t^{\frac{1}{2}}\Wc_p(V,W)_{\infty}.
		\end{aligned}
	\end{equation} 
	In a similar fashion, we show that there exists $\Kc >0$, depending only on $R$, $T$, $p$ and $M$, such that for all $t \in [0,T]$, 
	\begin{equation}
		\begin{aligned}
			\Wc_p(\mu_t, \nu_t) \leq \Kc t^{\frac{1}{2}}\Wc_p(V,W)_{\infty}.
		\end{aligned}
	\end{equation} 
	Thus, for every smooth solution of (\ref{eq:MDEl}) for $V$, there exists a smooth solution of (\ref{eq:MDEl}) for $W$, such that
	\begin{equation}
		\begin{aligned}
			\Wc_p(\mu, \nu)_{\infty} \leq \Kc T^{\frac{1}{2}}\Wc_p(V,W)_{\infty},
		\end{aligned}
	\end{equation}
	and vice--versa. This yields the local Lipschitz regularity of $\Sc_p$.

	The results for an arbitrary manifold with bounded geometry follow analogously by Lemma~\ref{le:fflowopman}.
\end{proof}

\subsubsection{Reduction of the non--linear setting to the linear setting}

Introducing the linear equation (\ref{eq:MDEl}) allows us to derive the uniqueness of solutions of (\ref{eq:MDEd}) which are limits of AFAS from the uniqueness of solutions of (\ref{eq:MDEl}) which are limits of LAFAS using a fixed point argument. This argument is detailed in our second main result, Theorem~\ref{th:AFASuniqueness}.
\begin{theorem}\label{th:AFASuniqueness}
	Let $p > 2$, $T>0$ and $\mu_0 \in \Wc_p(M)$. Assume:
	\begin{itemize}
	\item $(M,g)$ is of bounded geometry.
	\item $V$ is a probability vector--field which satisfies:
	\begin{itemize}
		\item $V$ is Lipschitz with respect to the $p$--Wasserstein distance.
		\item There exists a constant $R > 0$ such that for all $\mu \in \Wc_p(M)$, $V[\mu] \in \Wc_{p+2}(\Xc^{2,\infty}(M))$ and $\Mc_{p+2}(V[\mu]) \leq R^{p+2}$.
	\end{itemize}
	\item For all $\frac{1}{p}$--Hölder--continuous $\mu: [0,T] \to \Wc_p(M)$ with $\mu(0) = \mu_0$, the linear MDE (\ref{eq:MDEl}) has a unique smooth solution (See Definition~\ref{de:linsmoothsol}) for $t \mapsto V[\mu]$ with initial value $\mu_0$ in $\Wc_p(M)$. 
	\end{itemize}

	Then:
	\begin{itemize}
	\item $(\ref{eq:MDEd})$ has a unique smooth solution with initial value $\mu_0$ in $\Wc_p(M)$ on $[0,T]$.
	\item For any sequence of partitions of the interval $[0,T]$, $\biggl(\Pf(\frac{1}{N})\biggr)_{N \in \Nz}$ which satisfies $\lim_{N \to \infty}\pf(\Pf(\frac{1}{N})) = 0$, the sequence of Average Flow Approximation Series $\biggl(\mu^{\Pf(\frac{1}{N})}\biggr)_{N \in \Nz}$ with initial value $\mu_0$ converges uniformly to this unique smooth solution in $\Wc_p(M)$.
	\end{itemize}  
\end{theorem}
\begin{proof}

	Let $T>0$ and $\bigl(\Pf(\frac{1}{N})\bigr)_{N \in \Nz}$ be a sequence of partitions of $[0,T]$. Let $\Sc: \Cc^{\frac{1}{p}}([0,T], \Wc_p(M)) \to \Cc^{\frac{1}{p}}([0,T], \Wc_p(M))$ be the map sending a $\frac{1}{p}$--Hölder--continuous path $\mu: [0,T] \to \Wc_p(M)$ to the unique smooth solution of the corresponding linear MDE (\ref{eq:MDEl}) with initial value $\mu_0$. We will show that $\Sc$ is a contraction with respect to the $\Wc_p$--distance for $T$ small enough.

	Let $c_1,c_2 \in \Cc^{\frac{1}{p}}([0,T], \Wc_p(M))$. Denote by $t \mapsto \mu(t)$ and $t \mapsto \nu(t)$ the smooth solutions of (\ref{eq:MDEl}) for $V \circ c_1$ and $V \circ c_2$ respectively. By assumption, the linear MDEs (\ref{eq:MDEl}) for $V \circ c_1$ and $V \circ c_2$ have unique smooth solutions. Thus, for any sequence of partitions $\bigl(\Pf(\frac{1}{N})\bigr)_{N \in \Nz}$, the Linear Average Flow Approximation Series $\bigl(\mu^{\Pf(\frac{1}{N})}\bigr)_{N \in \Nz}$ and $\bigl(\nu^{\Pf(\frac{1}{N})}\bigr)_{N \in \Nz}$ converge uniformly to those solutions with respect to the $p$--Wasserstein distance by Theorem~\ref{th:lingencentlim}. 
	
	%It follows by induction from (\ref{eq:estimatecompl}) that there exist $T, \Kc>0$, such that for all $t \in [0,T]$, 
	Hence, Lemma~\ref{le:Hausdorffvectfield} yields that there exists a constant $\Kc > 0$, depending only on $p$, $T$, $R$ and $M$, such that for all $t \in [0,T]$,
	\begin{equation}
		\begin{aligned}
			\Wc_p(\mu(t), \nu(t)) \leq \Kc t^{\frac{1}{2}}\Wc_p(V \circ c_1,V \circ c_2)_{\infty}.
		\end{aligned}
	\end{equation} 
	In particular, 
	\begin{equation}
		\begin{aligned}
			\Wc_p(\mu, \nu)_{\infty} \leq L\Kc T^{\frac{1}{2}} \Wc_p(c_1,c_2)_{\infty}.
		\end{aligned}
	\end{equation} 
	Therefore, choosing $T$ small enough, ensures that $\Sc$ is a contraction. Thus, Banach's fixed point theorem yields that $\Sc$ has a unique fixed point $s: [0,T] \to \Wc_p(M)$. Hence, (\ref{eq:MDEd}) has a unique solution $s: [0,T] \to \Wc_p(M)$ with initial value $\mu_0$ on $[0,T]$, which is a limit 
	of Linear Average Flow Approximation Series associated to $s$. 

	Let us conclude by showing that $s$ is also the unique solution of \eqref{eq:MDEd} in $\Wc_p(M)$ with initial value $\mu_0$, which is a limit of Average Flow Approximation Series. Suppose there exist two such solutions $s_1$ and $s_2$. Hence, there exist $\bigl(\Pf_1(\frac{1}{N})\bigr)_{N \in \Nz}$ and $\bigl(\Pf_2(\frac{1}{N})\bigr)_{N \in \Nz}$, two sequences of partitions of $[0,T]$ and Average Flow Approximation Series $s^{\Pf_1(\frac{1}{N})}_1$ and $s^{\Pf_2(\frac{1}{N})}_2$ converging uniformly in $\Wc_p(M)$ respectively to $s_1$ and $s_2$ with respective 
	steps given by $\Pf_1(\frac{1}{N})$ and $\Pf_2(\frac{1}{N})$. Let $\ep >0$. Since the convergence is uniform and $V$ is $L$--Lipschitz with respect to the $p$--Wasserstein distance, there exists $N_0 \in \Nz$, such that for all $n \geq N_0$,
	\begin{equation}\label{eq:uniformconvest}
		\begin{aligned}
			&\Wc_p\biggl(\bigl(t \mapsto V[s_1(t)]\bigr), \bigl(t \mapsto V[s_1^{\Pf\bigl(\frac{1}{N}\bigr)}(t)]\bigr)\biggr)_{\infty} \leq \ep, \\
			&\Wc_p\biggl(\bigl(t \mapsto V[s_2(t)]\bigr), \bigl(t \mapsto V[s_2^{\Pf\bigl(\frac{1}{N}\bigr)}(t)]\bigr)\biggr)_{\infty} \leq \ep. 
		\end{aligned}
	\end{equation}
	Now, it suffices to show that this implies that the LAFAS $\sf^{\Pf_1(\frac{1}{N})}_1$ and $\sf^{\Pf_2(\frac{1}{N})}_2$ associated to the linear MDEs (\ref{eq:MDEl}) for $s_1$ and $s_2$ also converge to $s_1$ and $s_2$ respectively.
    This follows again by induction using (\ref{eq:uniformconvest}) and (\ref{eq:estfrandVp}) in Lemma~\ref{le:fflowop2}. Thus, $s_1$ and $s_2$ are both fixed points of $\Sc$ and, hence, $s_1 = s_2$. 

	Finally, we obtain a similar estimate for an arbitrary manifold with bounded geometry by a straightforward application of Nash's embedding theorem. The result follows for arbitrary $T>0$ by a standard local--to--global argument. 
\end{proof}

\begin{co}\label{co:densesubset}
	Let $p>2$ and $T>0$. Let $\Dc \in \Wc_p(M)$ be a dense subset. With the assumptions of Theorem~\ref{th:gronwallHausdorff}, suppose that for all $\mu_0 \in \Dc$ and all $\frac{1}{p}$--Hölder--continuous $\mu: [0,T] \to \Wc_p(M)$ with $\mu(0) = \mu_0$, the linear MDE (\ref{eq:MDEl}) has a unique smooth solution (see Definition~\ref{de:smoothsol}) in $\Wc_p(M)$. 
	Then, for all $\nu_0 \in \Wc_p(M)$, there exists a unique solution $\nu: [0,T] \to \Wc_p(M)$ of (\ref{eq:MDEd}) with initial value $\nu_0$ which is a uniform limit of Average Flow Approximation Series (see Definition~\ref{de:AFAS}) in $\Wc_p(M)$ on $[0,T]$.
\end{co}
\begin{proof}
	By Theorem~\ref{th:AFASuniqueness}, for all $\mu_0 \in \Dc$, $S(V,T,\mu_0)$ is a singleton. Hence, by Theorem~\ref{th:gronwallHausdorff}, $S(V,T,\nu_0)$ is a singleton for all $\nu_0 \in \Wc_p(M)$. 
\end{proof}
\begin{re}
	Corollary~\ref{co:densesubset} allows us to reduce our study to the study of a linear PDE with initial value in a functional space which is dense in $\Wc_p(M)$.
\end{re}

\subsubsection{Strong Generalised Central Limit Theorem and uniqueness of smooth solutions}
Now, we are ready to prove our third and last main result, the uniqueness of smooth solutions of (\ref{eq:MDEd}). First, we prove that the smooth solutions of (\ref{eq:MDEl}) are unique for sufficiently regular initial values. 
Then, we combine Theorem~\ref{th:AFASuniqueness} and some of our Grönwall type estimates to conclude.
\begin{pr}\label{pr:uniquelinsmoothinit}
	Let $R\geq 1$, $p > 2$, $\mu_0 \in \Cc_c^{\infty}(M)$, and $V: [0,T] \mapsto \Wc_p(\Xc^{2,\infty}(M))$ be a $\frac{1}{p}$--Hölder--continuous curve such that for all $t \in [0,T]$, $V_t \in \Wc_{p+2}(\Xc^{2,\infty}(M))$ and $\Mc_{p+2}(V_t) \leq R^{p+2}$. Then there exists a unique smooth solution of equation (\ref{eq:MDEl}) for $V$ with initial value $\mu_0$ in $\Wc_p(M)$. 
\end{pr}
\begin{proof}
	Let $\ep>0$. First, Theorem~\ref{th:ellipticapprox} yields that there exists $W: [0,T] \mapsto \Wc_p(\Xc^{2,\infty}(M))$, which is $\frac{1}{p}$--Hölder--continuous with uniformly bounded $(p+2)$--th moments and $\la >0$, such that for all $t \in [0,T]$, $\square^{W(t)}$ is $\la$--elliptic and $\Wc_p^{\infty}(V,W) \leq \ep$. 
	Since the $W(t)$ have uniformly bounded $(p+2)$--th moments,
	one readily shows that for all $i,j \in \llbracket 0, n-1 \rrbracket$,
	\begin{equation}
		\begin{aligned}
			A_{i,j}^W: [0,T] \times M &\to \Rz \\
			(t,x) &\mapsto a_{i,j}^{W(t)}(x),
		\end{aligned}
	\end{equation}
	and,
	\begin{equation}
		\begin{aligned}
			B_{j}^W: [0,T] \times M &\to \Rz \\
			(t,x) &\mapsto b_{j}^{W(t)}(x)
		\end{aligned}
	\end{equation}
	are uniformly in $W^{2,\infty}(M)$ for fixed $t$ and uniformly in $\Cc^{0,\frac{1}{p}}([0,T], \Rz)$ for fixed $x$. Thus,
	(\ref{eq:MDEl}) has a unique distributional solution $s$ for $W$ with initial value $\mu_0$.
	%%%%%%%%%%%%%%%%%%%%%%%%%%%%%%%%%%%%%%%%%%%
	Now $L:=\tfrac{\partial}{\partial t}-\square^{W(t)}$ is a second-order parabolic operator on $M\times(0,T]$, 
	which is uniformly parabolic and the coefficients are uniformly $W^{2,\infty}$ in space and uniformly
	$\tfrac 1p$--H\"older in time, so they are in $C^{2\al,\al}_{loc}(M\times(0,T])$ for
	$\al=\tfrac 1p$. If $u\in\dot{\Bc}'(M\times(0,T])$ (\cite{nigsch2018space} where the decay is only spatially but uniformly in
	$[t,T]$ for all $t>0$) satisfies $Lu=0$ in the sense of distributions and $u(0)=\mu_0$, by the parabolic Schauder theory, one has
	$$
		u\in C^{2+2\alpha,\,1+\alpha}_{loc}(M\times(0,T]),
	$$
	hence $u\in\Wc_p(M)$ and is a classical solution \cite{lieberman1996second,krylov1996lectures}.
	Moreover, $u$ is unique among all distributional solutions in $\dot{\Bc}'(M\times(0,T])$, by 
	the parabolic maximum/comparison principle on unbounded domains 
	\cite[Ch.~Parabolic Equations]{protter2012maximum}, 
	\cite{eidelman2012parabolic,aronson1967bounds,friedman2008partial}.
	%%%%%%%%%%%%%%%%%%%%%%%%%%%%%%%%%%%%%%%%%%%
	Now, take two smooth solutions $s_1$ and $s_2$ of $(\ref{eq:MDEl})$ for $V$ with initial value $\mu_0$. Then, Lemma~\ref{le:Hausdorffvectfield} yields that 
	there exists $C>0$ independent of $\ep$, such that $\Wc_p^{\infty}(s,s_1) \leq C T^{\frac{1}{2}}\ep$ and $\Wc_p^{\infty}(s,s_2) \leq C T^{\frac{1}{2}}\ep$. As $\ep$ was chosen arbitrarily, we conclude $s_1 = s_2$. 
\end{proof}
\begin{theorem}[Generalised Central Limit Theorem---Strong version]\label{th:gencentlimstrong}
	Let $p > 2$ and $T >0$. Let $\mu_0 \in \Wc_p(M)$. Let $\bigl( \Pf(\frac{1}{N}) \bigr)_{N \in \Nz}$ be a family of partitions of $[0,T]$. Assume:
	\begin{itemize}
	\item $(M,g)$ is of bounded geometry.
	\item $V$ is a probability vector--field which satisfies:
	\begin{itemize}
		\item There exists a constant $R > 0$ such that for all $\mu \in \Wc_p(M)$, $V[\mu] \in \Wc_{p+2}(\Xc^{2,\infty}(M))$ and $\Mc_{p+2}(V[\mu]) \leq R^{p+2}$.
		\item  $V$ is $L$--Lipschitz with respect to the $\Wc_p$--distance for some $L>0$.  
	\end{itemize}
	\end{itemize}
	Then, (\ref{eq:MDEd}) has a unique smooth solution $\mu: [0,T] \to \Wc_p(M)$ (compare Definiton~\ref{de:smoothsol}) with initial value $\mu_0$ and the Average Flow Approximation Series $\bigl( \mu^{\Pf(\frac{1}{N})} \bigr)_{N \in \Nz}$ (see Definition~\ref{de:AFAS}) converges uniformly to $t \mapsto \mu(t)$ in $\Wc_p(M)$.
	Furthermore, $t \mapsto \mu(t)$ is $\frac{1}{p}$--Hölder--continuous. 
\end{theorem}
\begin{proof}
	Let $\mu_0 \in \Cc^{\infty}_c(M)$. Proposition~\ref{pr:uniquelinsmoothinit} yields that for all $\frac{1}{p}$--Hölder--con\-ti\-nu\-ous $\mu: [0,T] \to \Wc_p(M)$ with $\mu(0) = \mu_0$, the linear MDE (\ref{eq:MDEl}) has a unique smooth solution for $t \mapsto V[\mu(t)]$ with initial value $\mu_0$ in $\Wc_p(M)$.
	Hence, Theorem~\ref{th:AFASuniqueness} yields that (\ref{eq:MDEd}) has a unique smooth solution with initial value $\mu_0$ in $\Wc_p(M)$. Finally, we extend the uniqueness on arbitrary $\mu_0 \in \Wc_p(M)$ with the last point
	of Theorem~\ref{th:gronwallHausdorff}. Now, the rest of the statement is a consequence of Theorem~\ref{th:gencentlim}.
\end{proof}
\begin{re}
	In particular, any AFAS uniformly converges to the unique smooth solution with Theorem~\ref{th:gencentlim}.
\end{re}
\begin{re}\label{re:irravgvec}
	Theorem~\ref{th:gencentlimstrong} still holds if we slightly relax the spatial regularity of our vector--field probabilities. 
	For all the above proofs, it is sufficient to require our tangent vectors $V$ to be supported in $\Xc^{1,\infty}(M)$, such that
	for all $X \in \supp(V)$, $X - \ol{V} \in \Xc^{2,\infty}(M)$. In other words, the centered vector--field probability 
	$\tau^{-\ol{V}}_* V \in \Wc_p(\Xc^{2,\infty}(M))$, while we only assume $W^{1,\infty}$--regularity for the average vector--field $\ol{V}$. In fact, while some of the above results 
	fail in all generality if we weaken these assumptions even further, Theorem~\ref{th:gencentlimstrong} still holds if we only assume 
	$\tau^{-\ol{V}}_* V \in \Wc_p(\Xc^{1 + \de,\infty}(M))$ for some $\de >0$.
\end{re}
\begin{re}
	Smooth solutions of (\ref{eq:MDEd}) are particularly relevant for applications, since they are predictable. Knowing $\mu_0$ and $V[\mu_0]$ is sufficient to compute an approximation of $\mu(t)$, for small $t>0$, by the means of AFAS. To approximate any non--smooth solution of (\ref{eq:MDEd}), even 
	for $t > 0$ arbitrarily small, one needs to know $V$ on a whole neighborhood of $\mu_0$ in $\Wc_p(M)$.
\end{re}
\subsubsection{A counter--example to general uniqueness of solutions}
General solutions of \eqref{eq:MDEd} are however not unique in general, even if $V$ is Lipschitz and the $\square^V_{\mu}$ are uniformly elliptic for all $\mu \in \Wc_p(M)$. Let us give an example of two distinct solutions to (\ref{eq:MDEd}) with same initial value. 
\begin{theorem}
	There exists a probability vector--field $V$ which is Lipschitz with respect to the $p$--Wasserstein distance and such that $\square^V_{\mu}$ is strongly elliptic for all $\mu \in \Wc_p(M)$, such that solutions to \eqref{eq:MDEd} are  
	not unique.
\end{theorem}
\begin{proof}
	Let $(M,g) \coloneqq (\Rz^2, g^{\text{eucl}})$ and $M > 1$, $0<m<1$. Let $\rho$ and $r$ be the maps sending a symmetric matrix $A \in \Mc_2(\Rz)$ onto its largest and smallest eigenvalues respectively and 
	$v_1,v_2: \Sc_2(\Rz)\setminus \Rz \id$ be the maps sending a symmetric matrix with $1$--dimensional eigenspaces onto the unitary generator of the eigenspace $E_{\rho(A)}$
	in $\{(x,y) \in \Rz^2 \mid x > 0\} \cup \{(0,y) \in \Rz^2 \mid y > 0\}$ and the orthogonal unitary vector in the same set respectively. For all $\rho > 0$, we also define the auxiliary functions
	\begin{equation}
		\begin{aligned}
			\De_{\rho}: \Rz_+ &\to \Rz_+\\
			\si_2 &\mapsto \frac{\rho}{\rho - \log(\rho) - 1}(\si_2 - \log(\rho) - 1) + \frac{1}{\rho - \log(\rho) - 1}(\rho - \si_2).
		\end{aligned}
	\end{equation}
	and $c: x \mapsto \min \bigl(M, \max\bigl(m, x\bigr)\bigr)$. Furthermore, define
	\begin{equation}
		\begin{aligned}
			P_1: \Sc_2(\Rz)\setminus \Rz \id &\to \Rz^2 \\
			A &\mapsto \sqrt{\frac{2}{\pi}{c(\rho(A))}}v_1(A)
		\end{aligned}
	\end{equation}
	and
	\begin{equation}
		\begin{aligned}
			P_2: \Sc_2(\Rz)\setminus \Rz \id &\to \Rz^2 \\
			A &\mapsto \sqrt{\frac{2}{\pi}c\bigl(\De_{\rho(A)}(r)\bigr)}v_2(A).
		\end{aligned}
	\end{equation}
	For any $r \in \Rz$ and $v_1,v_2 \in \Rz^2$ with $\bigl<v_1,v_2\bigr> = 0$, denote by $\Cc_r$ the circle with radius $r$ and by $\Ec_{v_1,v_2}$ the ellipse with semimajor axis $v_1$ and semiminor axis $v_2$. 
	Finally, define the probability vector--field 
	\begin{equation*}
		\begin{aligned}
			V: \Wc_p(\Rz^2) &\mapsto \Wc_p(\Xc^{2,\infty}(\Rz^2)) \\
%			\mu &\mapsto \begin{cases}
%				1_{E_{P_1(\cov(\mu)), P_2(\cov(\mu))}} &\text{ if $\cov(\mu) \notin \Rz \id$}, \\
%				1_{\Cc_r} &\text{ if $\cov(\mu) = r \id$},
%			\end{cases} 
		\end{aligned}
	\end{equation*}
	\begin{equation}
		\begin{aligned}
%			V: \Wc_p(\Rz^2) &\mapsto \Wc_p(\Xc^{2,\infty}(\Rz^2)) \\
			\mu &\mapsto \begin{cases}
				1_{E_{P_1(\cov(\mu)), P_2(\cov(\mu))}} &\text{ if $\cov(\mu) \notin \Rz \id$}, \\
				1_{\Cc_r} &\text{ if $\cov(\mu) = r \id$},
			\end{cases} 
		\end{aligned}
	\end{equation}
	where, by abuse of notation, for all $X \in \Rz^2$, we identify the constant vector--field $v_X: y \mapsto X$ with $X$ here.
	Now, we show that $V$ is locally Lipschitz with respect to the $p$--Wasserstein distance around $\mu_0$, where $\mu_0$ is the standard normal distribution. Using Proposition~\ref{pr:VarLip}, there exist $R>0$, $B>0$ and $b >0$, such that for all $\mu \in B(\mu_0,R)$, $r(\cov(\mu))\geq b$ and $\rho(\cov(\mu))\leq B$. Thus, there exist $C>0$, such that for all $\rho,r_1,r_2 \in [b,B]$, $\bigl|\De_{\rho}(r_1) - \De_{\rho}(r_2)\bigr| \leq C|r_1 - r_2|$. Now, Proposition~\ref{pr:VarLip} yields that $V$ is locally Lipschitz around $\mu_0$. Furthermore, $V$ is uniformly bounded by $M>0$. 

	Now, we show that there exist two distinct solutions to (\ref{eq:MDEd}) with initial value $\mu_0$. Let $T>0$. Define
	\begin{equation}
		\begin{aligned}
			&\mu_1: [0,T] \mapsto \biggl((x,y) \mapsto \frac{1}{2\pi\sqrt{(2t+1)e^{2t}}}e^{-\frac{1}{2}\bigl(\frac{x^2}{2t + 1} + \frac{y^2}{e^{2t}}\bigr)} \biggr) \\
			&\mu_2: [0,T] \mapsto \biggl((x,y) \mapsto \frac{1}{2\pi\sqrt{(2t+1)e^{2t}}}e^{-\frac{1}{2}\bigl(\frac{y^2}{2t + 1} + \frac{x^2}{e^{2t}}\bigr)} \biggr). \\
		\end{aligned}
	\end{equation}
	Then, one readily shows that for all $t \in [0,T]$,
	\begin{equation}
		\begin{aligned}
			&\frac{\der}{\der t} \mu_1(t) = \frac{\der^2}{\der x^2} \mu_1(t) + \cov(\mu_1(t))_{2,2}\frac{\der^2}{\der y^2} \mu_1(t), \\
			&\frac{\der}{\der t} \mu_2(t) = \frac{\der^2}{\der y^2} \mu_2(t) + \cov(\mu_2(t))_{1,1}\frac{\der^2}{\der x^2} \mu_2(t).
		\end{aligned}
	\end{equation}
	Now, for all $t \in [0,\sqrt{\log(M)}]$, it holds that $\rho(\cov(\mu_2(t))) = e^{2t} \leq M$, and for all $\ph \in \Cc_c^{\infty}(\Rz^2)$,
	\begin{equation}
		\begin{aligned}
			&\int_{\Xc^{2,\infty}(\Rz^2)} \Lc_{X}^2(\ph) \der V[\mu_2(t)](X) \\
			&\qquad = \int_{\Xc^{2,\infty}(\Rz^2)} X_1^2 \der V[\mu_2(t)](X) \frac{\der^2}{\der x^2}\ph + \int_{\Xc^{2,\infty}(\Rz^2)} X_2^2 \der V[\mu_2(t)](X) \frac{\der^2}{\der y^2}\ph \\
			&\qquad + 2\int_{\Xc^{2,\infty}(\Rz^2)} X_1 X_2 \der V[\mu_2(t)](X) \frac{\der^2}{\der x \der y}\ph \\
			&\qquad = \cov(\mu_2(t))_{1,1}\frac{\der^2}{\der x^2} \ph + \De_{\rho(\cov(\mu_2(t)))}(\cov(\mu_2(t))_{2,2})\frac{\der^2}{\der y^2} \ph.
		\end{aligned}
	\end{equation}
	Finally, $\De_{\rho(\cov(\mu_2(t)))}(\cov(\mu_2(t))_{2,2}) = \De_{e^{2t}}(2t+1) = 1$. Therefore,
	$\mu_2$ does indeed satisfy (\ref{eq:MDEd}) on $[0,\sqrt{\log(M)}]$. Analogously, we show that 
	$\mu_1$ does satisfy (\ref{eq:MDEd}) on $[0,\sqrt{\log(M)}]$. Therefore, we have shown that (\ref{eq:MDEd})
	has two distinct solutions.
\end{proof}

\subsection{Applications}

We finish this work with some interesting applications of our Generalised Central Limit Theorems.

\subsubsection{Classical Central Limit Theorem}

Let us quickly justify why we call Theorem~\ref{th:gencentlim} and Theorem~\ref{th:gencentlimstrong} generalisations of the Central Limit Theorem. 
\begin{theorem}[Central Limit Theorem]
Let $p > 2$ and $(X_i)_{i\in\Nz}$ be a sequence of i.i.d random variables with image in $\Rz^n$ and finite $p$--th moment. Denote their covariance--matrix by $\Si^2$. Then, the sequence $\bigl(\frac{1}{\sqrt{n}}\sum_{i = 0}^{n-1} (X_i - \Ez(X_i))\bigr)_{n \in \Nz}$ converges to $\Nc(0, \Si^2)$ with respect to the $p'$--Wasserstein distance, for all $p' < p$.
\end{theorem}
\begin{proof}
Let $\mu \in \Wc_p(\Rz^n)$ be the distribution of the $X_i$. For all $v \in \Rz^n$ we define $X_v: \Rz^n \to \Rz^n$ to be the  vector--field that is constantly equal to $v$. The map 
\begin{equation}
\begin{aligned}
i: &\Rz^n \to \Xc^{2,\infty}(M)\\
 &v \mapsto X_v - \Ez(X_1)
\end{aligned}
\end{equation}
is an isometry with respect to the $W^{2,\infty}$--norm. Hence it is in particular measurable and $i_* \mu \in \Wc_p(\Xc^{2,\infty}(M))$. Now, define the constant probability vector--field  $V: \nu \mapsto i_* \mu$. Note that $\ol{V}$ is constant equal to $0$. Furthermore, for all $x\in \Rz^n$ and all $t\in [0,T]$, 
\begin{equation}
\rh^{X_v}_t(x) = x + tv.
\end{equation}
Hence, for all $x \in \Rz^n$, all $t$ and all $i\in\Nz$, $x + \sqrt{t}X_i = \rh^{X_{X_i}}_{\sqrt{t}}(x)$. Thus, the random variable $x + \sqrt{t}(X_i - \Ez(X_i))$ has distribution $f^V_{t} (\de_x)$. 

Since the $X_i$ are independent by hypothesis, we readily show that for all $n\in \Nz$, the distribution of 
$\frac{1}{\sqrt{n}}\sum_{i = 0}^{n-1} (X_i - \Ez(X_i))$ is 
just $\bigl(f^V_{\frac{1}{n}}\bigr)^{\circ n} (\de_0) = \mu^n(1)$, where the $\mu^n$ are the Average Flow Approximation Series of Definition~\ref{de:AFAS}. Now, Theorem~\ref{th:gencentlim} gives that any subsequence of $\mu^N$ has a subsequence that converges to a solution of (\ref{eq:MDEd}) with initial condition $\de_0$ with respect to the $p'$--Wasserstein distance. 

Let $(e_i)_{i \in \llbracket 1, n \rrbracket}$ be the canonical basis of $\Rz^n$ and take $x,v \in \Rz^n$. The $X_v$ are all constant, so (\ref{eq:LDclassic}) yields 
\begin{equation}
\Lc^2_{X_v - \Ez(X)}\ph(x) = (X_v - \Ez(X_1))_i(x) (X_v - \Ez(X_1))_j(x)\frac{\d^2 }{\d x_i\d x_j}\ph(x).
\end{equation}
Thus,
\begin{equation} 
\begin{aligned}
&\int_{\Xct(M)} \Lc^2_{X_v - \Ez(X_1)}\ph(x) \der i_* \mu(X_v)\\
 &\qquad= \biggl( \int_{\Xct(M)}\!\! (X_v - \Ez(X_1))_i(x) (X_v - \Ez(X_1))_j(x) \der i_* \mu(X_v) \biggr) \frac{\d^2 }{\d x_i\d x_j}\ph(x)\\ &\qquad= \biggl( \int_{\Xct(M)} (v_i- \Ez(X_1)) (v_j- \Ez(X_1)) \der \mu(v) \biggr) \frac{\d^2 }{\d x_i\d x_j}\ph(x)\\ &\qquad= (\Si^2)_{i,j}\frac{\d^2 }{\d x_i\d x_j}\ph(x).
\end{aligned}
\end{equation}
Hence $\mu(t)$ solves the equation 
\begin{equation}\label{eq:diffeq}
\begin{aligned}
\int_M \ph \der \mu(s) - \int_M \ph \der \mu(0) &= 
\frac{1}{2}\int_0^s \int_M \sum_{i,j} (\Si^2)_{i,j}\frac{\d^2 }{\d x_i\d x_j}\ph(x) \der \mu(t) \der t,
\end{aligned}
\end{equation} 
for all $s \in [0,T]$ and all $\ph \in \Cc^{\infty}_c(M)$.
We know that the Gaussian diffusion $\mu(t) = \Nc(0, t\Si^2)$ is a solution of (\ref{eq:diffeq}). Furthermore, solutions of the semi-parabolic equation (\ref{eq:diffeq}) with initial condition $\mu_0 = \de_0$ are unique in $\Wc_p(\Rz^n)$ for $p > 3$.

Therefore, any subsequence of $\mu^N$ has a subsequence that converges to $\mu(t) = \Nc(0, t\Si^2)$ with respect to the $\Wc_{p'}$--distance. Hence, $\mu^N$ converges to $\mu(t) = \Nc(0, t\Si^2)$ with respect to the $\Wc_{p'}$--distance. In particular, $$\lim_{n \to \infty} \mu^n(1) = \Nc(0, \Si^2)$$ with respect to the $\Wc_{p'}(\Rz^n)$--distance.  
\end{proof}

\subsubsection{Lipschitz vector--fields}

Let us prove that our Generalised Central Limit Theorems directly yield similar results on $\Rz^n$ for Lipschitz vector--fields with bounded derivatives up to order two.
Define for all $n \in \Nz$,
\begin{equation}
	\begin{aligned}
		\Xc^{2,\Lip}(\Rz^n) &\coloneqq \{ X \in \Xct(\Rz^n) \mid \exists L > 0,\\ &\qquad\qquad\forall x \in \Rz^n: \|\nabla X(x)\| \leq L, \|\nabla^2 X(x)\| \leq L \}
	\end{aligned}
\end{equation}
and for all $X \in \Xc^{2,\Lip}(\Rz^n)$,
\begin{equation}
	\begin{aligned}
		\|X\|_{W^{2,\Lip}} &\coloneqq  \|X(0)\| + \sup_{x \in \Rz^n} \|\nabla X(x)\| + \sup_{x \in \Rz^n} \|\nabla^2 X(x)\|.\\
	\end{aligned}
\end{equation}
Then, $(\Xc^{2,\Lip}(\Rz^n), \|\cdot\|_{W^{2,\Lip}})$ is a Banach space. 
Furthermore, endow $\Rz^n$ with the Riemannian metric
\begin{equation}
	\begin{aligned}
		l_x(v,w) \coloneqq \frac{1}{1 + \|x\|^2}\bigl<v,w\bigr>.
	\end{aligned}
\end{equation} 
\begin{theorem} \textbf{\normalshape(Generalised Central Limit Theorem---Lipschitz vector--fields)}\label{th:gencentlimLip}
	Let $p > 2$ and $T >0$. Let $\mu_0 \in \Wc_p(\Rz^n,l)$. Let $\bigl( \Pf(\frac{1}{N}) \bigr)_{N \in \Nz}$ be a family of partitions of $[0,T]$. Let $V: \Wc_p(\Rz^n,l) \mapsto \Wc_p(\Xc^{2,\Lip}(\Rz^n))$ such that:
	\begin{itemize}
		\item There exists a constant $R > 0$ such that for all $\mu \in \Wc_p(\Rz^n,l)$, $V[\mu] \in \Wc_{p+2}(\Xc^{2,\Lip}(\Rz^n))$ and $\Mc_{p+2}(V[\mu]) \leq R^{p+2}$.
		\item  $V$ is Lipschitz with respect to the $\Wc_p$--distance.  
	\end{itemize}
	Then, (\ref{eq:MDEd}) has a unique smooth solution $\mu: [0,T] \to \Wc_p(\Rz^n,l)$ (see Definiton~\ref{de:smoothsol}) with initial value $\mu_0$ and the Average Flow Approximation Series $\bigl( \mu^{\Pf(\frac{1}{N})} \bigr)_{N \in \Nz}$ (see Definition~\ref{de:AFAS}) converges uniformly to $t \mapsto \mu(t)$ in $\Wc_p(\Rz^n,l)$.
	Furthermore, $t \mapsto \mu(t)$ is $\frac{1}{p}$--Hölder--continuous. 
\end{theorem}
\begin{proof}
	The Riemannian manifold $(\Rz^n,l)$ is of bounded geometry and it is easy to see that there exists a Lipschitz embedding $$i: \Xc^{2,\Lip}(\Rz^n) \to \Xc^{2,\infty}((\Rz^n,l)).$$ Hence,
	there exists a Lipschitz embedding $$I: \Wc_p\bigl(\Xc^{2,\Lip}(\Rz^n)\bigr) \to \Wc_p\bigl(\Xc^{2,\infty}((\Rz^n,l))\bigr).$$
	Now, the result follows directly from Theorem~\ref{th:gencentlimstrong} for the probability vector--field $I \circ V$.
\end{proof}

\begin{re}
	Note that the solution $\mu: [0,T] \to \Wc_p(\Rz^n,l)$ of Theorem~\ref{th:gencentlimLip} is in general not a solution in $\Wc_p(\Rz^n)$, even if $\mu_0 \in \Wc_p(\Rz^n)$.
	We need slightly stronger assumptions on the probability vector--field $V$ to ensure that $\mu$ is indeed a solution in $\Wc_p(\Rz^n)$. 
\end{re}

\begin{pr}
	Let $\mu_0 \in \Wc_p(\Rz^n)$. With the assumptions of Theorem~\ref{th:gencentlimLip}, suppose furthermore that there exists $C > 0$ such that for all $\mu \in \Wc_p(\Rz^n,l)$, 
	$V[\mu]$ is supported in the centered ball with radius $C$ in $\Xc^{2,Lip}(\Rz^n)$. Then, the smooth solution $\mu: [0,T] \to \Wc_p(\Rz^n,l)$ of Theorem~\ref{th:gencentlimLip} is $\frac{1}{p}$--Hölder--continuous in $\Wc_p(\Rz^n)$ and 
	the convergence of Average Flow Approximation Series is uniform in $\Wc_p(\Rz^n)$.
\end{pr}
\begin{proof}
	We proceed analogously as in the proof of Lemma~\ref{le:fflowop2}, using the Grönwall--estimate 
	\begin{equation}
		\|\rho_t^X(x) - x\| \leq t\bigl(\|X(0)\| + \sup_{x \in \Rz^n}\|\nabla X(x)\|\|x\|\bigr)e^{ \sup_{x \in \Rz^n}\|\nabla X(x)\| t},
	\end{equation}
	for all $t > 0$, $x \in \Rz^n$ and $X \in \Xc^{2,Lip}(\Rz^n)$.
\end{proof}

\subsubsection{Time--dependent probability vector--fields}

Finally, we show that the Generalised Central Limit Theorems can be extended to time--dependent 
probability vector--fields without too much effort. 

\begin{theorem}[Time--dependent Generalised Central Limit Theorem]\label{th:gentcenlimtime}\strut\newline
	Let $p>2$, $T>0$ and let $(M,g)$ be a Riemannian manifold. Let $\mu_0 \in \Wc_p(M)$ and $t \mapsto V_t$ be a curve of probability vector--fields 
	satisfying:
	\begin{itemize}
		\item There exists $R > 0$ such that for all $t \in [0,T]$ and all $\mu \in \Wc_p(M)$, $V_t[\mu] \in \Wc_{p+2}(\Xc^{2,\infty}(M))$ and $\Mc_{p+2}(V_t[\mu]) \leq R^{p+2}$.
		\item There exists $L>0$ such that for all $t,s \in [0,T]$ and all $\mu,\nu \in \Wc_p(M)$, $\Wc_p\bigl(V_t[\mu], V_s[\nu]\bigr) \leq L\bigl(|t-s| + \Wc_p(\mu,\nu)\bigr)$.  
	\end{itemize}
	We can define a flow--equation \emph{(MDET)}, approximation series and a notion of smooth solutions for this time--dependent curve of 
	probability vector--field analogously to \eqref{eq:MDEd}, and Definitions~\ref{de:AFAS} and \ref{de:smoothsol}. Then, \emph{(MDET)} has a unique 
	smooth solution $\mu: [0,T] \to \Wc_p(M)$ for $t \mapsto V_t$ with initial value $\mu_0$.
\end{theorem}

\begin{proof}
First, note that for any Riemmannian manifold $(M,g)$, $\Rz \times M$ is naturally endowed with a Riemannian structure by simply taking the product of $g$ with the Euclidean metric.
Let $e_T$ be the constant vector--field $(1, 0)$ on $\Rz \times M$, denote by $\pi_T$ the projection on the time coordinate and define
\begin{equation}
	\begin{aligned}
		\Ec: \Xc^{2,\infty}(\Rz \times M) &\to \Xc^{2,\infty}(\Rz \times M)\\
		X &\mapsto \bigl((x,t) \mapsto X(x)\bigr),\\
		\tau_T: \Xc^{2,\infty}(\Rz \times M) &\to \Xc^{2,\infty}(\Rz \times M)\\
		X &\mapsto X + e_T.
	\end{aligned}
\end{equation}
Furthermore, define
\begin{equation}
	\begin{aligned}
		&\Wc_{p,d}(\Rz \times M) \coloneqq \biggl\{\mu \in \Wc_p(\Rz \times M) \,\biggl| \\
		&\qquad \exists N \in \Nz, \exists (x_i)_{i \in \llbracket 0, n-1\rrbracket } \in \Rz^N, \supp\biggl(\bigl(\pi_T\bigr)_*\mu\biggr) \subset \bigcup_{i=0}^{n-1}\{x_i\} \biggr\}.
	\end{aligned}
\end{equation}
For all $\mu \in \Wc_{p,d}(\Rz \times M)$, there exists a unique family of measures $(\mu_t)_{t \in \Rz}$ in $\Wc_p(M)$,
such that for all $\ph \in \Cc^{\infty}_c(\Rz \times M)$,
\begin{equation}
	\int_{\Rz \times M} \ph \der \mu = \int_{\Rz}\int_M \ph(t,x) \der \mu_t(x) \der \bigl(\pi_T\bigr)_*\mu_t. 
\end{equation}
Furthermore, for any $\mu,\nu \in \Wc_{p,d}(\Rz \times M)$, any optimal transport plan $\ga$
between $\mu$ and $\nu$, there exists a family $(\ga_{t,s})_{s,t\in\Rz}$ such that for all 
$\ph \in \Cc^{\infty}_c(\Rz \times M \times \Rz \times M)$,
\begin{equation}
	\int_{\Rz \times M \times \Rz \times M} \ph \der \mu = \int_{\Rz \times \Rz} \int_M \ph(t,x,s,y) \der \ga_{t,s}(x) \der \bigl(\pi_T, \pi_T\bigr)_*\ga(t,s),
\end{equation}
and all $t,s \in \Rz$, $\ga_{t,s}$ is an optimal transport plan between $\mu_t$ and $\nu_s$.
Now, let $V: [0,T] \times \Wc_p(M) \to \Wc_p(\Xc^{2,\infty}(M))$ be a curve of probability vector--fields. Then, $V$ yields a  
map on $\Wc_{p,d}(\Rz \times M)$, defined by
\begin{equation}
	\begin{aligned}
		V_T: \Wc_{p,d}(\Rz \times M) &\to \Wc_p\bigl(\Xc^{2,\infty}(\Rz \times M)\bigr) \\
		\mu &\mapsto \int_{\Rz} \bigl((\tau_T)_*\Ec_*V_t[\mu_t]\bigr) \der \bigl((\pi_T)_*\mu\bigr)(t).
	\end{aligned}
\end{equation}
Let $\ep >0$. Take $\mu,\nu \in \Wc_{p,d}(M)$ and let $\ga$ be an optimal transport plan between $\mu$ and $\nu$. Then, $(\pi_t,\pi_t)_* \ga$ is a transport 
plan between $(\pi_t)_* \mu$ and $(\pi_t)_* \nu$. It readily follows that
\begin{equation}
	\Wc_p\bigl((\pi_t)_* \mu, (\pi_t)_* \nu\bigr) \leq \Wc_p(\mu,\nu).
\end{equation}
Furthermore, for all $s,t \in \Rz$, let $\Ga_{s,t}$ be a transport plan between
$V_t[\mu_t]$ and $V_s[\nu_s]$ such that 
\begin{equation}
	\int_{\Xc^{2,\infty}(M) \times \Xc^{2,\infty}(M)} \|X - Y\|_{W^{2,\infty}}^p \der \Ga_{t,s}(X,Y) \leq \Wc_p(V_t[\mu_t],V_s[\nu_s])^p + \ep.
\end{equation}
Then $\bigl(\tau_t \circ \Ec, \tau_t \circ \Ec\bigr)_* \Ga_{s,t}$ is a transport plan between the vector--field probabilities $\bigl((\tau_T)_*\Ec_*V_t[\mu_t]\bigr)$
and $\bigl((\tau_T)_*\Ec_*V_s[\nu_s]\bigr)$. It readily follows that
\begin{equation}
	\begin{aligned}
		&\int_{\Xc^{2,\infty}(M) \times \Xc^{2,\infty}(M)} \|X - Y\|_{W^{2,\infty}}^p \der \Ga_{t,s}(X,Y)\\
		&\qquad  = \int_{\Xc^{2,\infty}(\Rz \times M) \times \Xc^{2,\infty}(\Rz \times M)} \|X - Y\|_{W^{2,\infty}}^p \bigl(\tau_t \circ \Ec, \tau_t \circ \Ec\bigr)_* \Ga_{s,t}(X,Y),\\
	\end{aligned}
\end{equation}
and
\begin{equation}
	\Wc_p\bigl(\bigl((\tau_T)_*\Ec_*V_t[\mu_t]\bigr), \bigl((\tau_T)_*\Ec_*V_s[\nu_s]\bigr)\bigr) = \Wc_p(V_t[\mu_t],V_s[\nu_s]). 
\end{equation}
Thus,
\begin{equation}
	V_T\Ga \coloneqq \int_{\Rz \times \Rz} \bigl(\tau_t \circ \Ec, \tau_t \circ \Ec\bigr)_* \Ga_{s,t} \der (\pi_t,\pi_t)_* \ga(s,t)
\end{equation}
is a transport plan between $V_T(\mu)$ and $V_T(\nu)$. It follows that
\begin{equation}
	\begin{aligned}
		&\Wc_p\bigl(V_T(\mu), V_T(\nu)\bigr)^p\\
		& \leq \int_{\Xc^{2,\infty}(\Rz \times M) \times \Xc^{2,\infty}(\Rz \times M)} \|X - Y\|_{W^{2,\infty}(M)}^p \der V_T \Ga (X,Y)\\
		&= \int_{\Rz \times \Rz} \int_{\Xc^{2,\infty}(M) \times \Xc^{2,\infty}(M)} \|X - Y\|_{W^{2,\infty}(M)}^p \der \Ga_{s,t}(X,Y)\der (\pi_t,\pi_t)_* \ga(t,s)\\
		&\leq \int_{\Rz \times \Rz} \Wc_p(V_t[\mu_t], V_s[\nu_s])^p\der (\pi_t,\pi_t)_* \ga(t,s) + \ep\\
		&\leq \int_{\Rz \times \Rz} 2^p L^p\bigl(|t-s|^p + \Wc_p(\mu_t,\nu_s)^p\bigr) \der (\pi_t,\pi_t)_* \ga(t,s) + \ep\\
		&= \int_{\Rz \times \Rz} 2^p L^p\biggl(|t-s|^p + \int_{M \times M}d(x,y)^p \der \ga_{t,s}(x,y)\biggr) \der (\pi_t,\pi_t)_* \ga(t,s) + \ep\\
		&\leq 2^pL^p \int_{\Rz \times \Rz} \int_{M \times M} |t-s|^p + d(x,y)^p \der \ga_{t,s}(x,y)\der (\pi_t,\pi_t)_* \ga(t,s) + \ep\\
		&\leq 2^{\frac{3p}{2}}L^p \Wc_p(\mu,\nu)^p + \ep.
	\end{aligned}
\end{equation}
Since $\ep$ was arbitrary, it follows that
\begin{equation}
	\Wc_p\bigl(V_T(\mu), V_T(\nu)\bigr) \leq 2^{\frac{3}{2}}L \Wc_p(\mu,\nu).
\end{equation}
Hence, $V_T$ is Lipschitz on $\Wc_{p,d}(\Rz \times M)$. Since $\Wc_{p,d}(\Rz \times M)$ is dense in $\Wc_p(\Rz \times M)$ and $\Wc_p(\Xc^{2,\infty}(\Rz \times M))$ is complete by Lemma~\ref{le:Wassersteincomplete}, we can extend it to a Lipschitz probability vector--field $V_t: \Wc_p(\Rz \times M)\to \Wc_p(\Xc^{2,\infty}(\Rz \times M))$.
Now, by Theorem~\ref{th:gencentlimstrong}, there exists a unique smooth solution $t \mapsto M_t$ of \eqref{eq:MDEd} in $\Wc_p(\Rz \times M)$ for $V_T$ with initial value $(i_0)_* \mu_0$,
where for all $t>0$,
\begin{equation}
	\begin{aligned}
		i_t: M &\to \Rz \times M\\
		x &\mapsto (t,x).\\
	\end{aligned}
\end{equation}
Finally, for all $t \in [0,T]$, define $\mu(t) \coloneqq (\pi_M)_* M_t$, where $\pi_M: \Rz \times M \to M$ is the projection on $M$. Then, 
one readily shows that $t \mapsto \mu(t)$ is a smooth solution of (MDET) with initial value $\mu_0$. 

Suppose there exists another smooth solution $t \mapsto \nu(t)$ of (MDET) with initial value $\mu_0$. Then, the curve $N: t \mapsto (i_t)_* \nu(t)$ is a smooth solution of \eqref{eq:MDEd} in $\Wc_p(\Rz \times M)$ for $V_T$ with initial value $(i_0)_* \mu_0$. By uniqueness of such solutions, it holds for all $t \in [0,T]$, $N_t = M_t$, and therefore $\nu(t) = \mu(t)$. Hence, the smooth solution of (MDET) with initial value $\mu_0$ is unique.
\end{proof}

\section{Conclusion}
We have shown that \eqref{eq:MDEd} has a solution for any initial value $\mu_0 \in \Wc_p(M)$ with only very weak assumptions on the vector--field probability $V$. Furthermore, we have introduced the notion of smooth solutions of \eqref{eq:MDEd} and have proved that such solutions exist under these very assumptions and can be constructed as uniform limits of Average Flow Approximation Series. 
Thus, we have also proved a strong generalised version of the Central Limit Theorem.
Smooth solutions are the only predictable solutions of \eqref{eq:MDEd} and are, thus, the most meaningful solutions for possible applications of the theory. They allow us to define new differential structures on the Wasserstein spaces $\Wc_p(M)$ for $p > 2$ and a general Riemannian manifold $(M,g)$, precisely by identifiying smooth solutions with the smooth curves on $\Wc_p(M)$. 
Although \eqref{eq:MDEd} is a possibly degenerate second order partial differential equation with measure dependent coefficients, we have also shown that only slightly stronger assumptions on $(M,g)$ and $V$, notably that $V$ is Lipschitz, ensure that those smooth solutions are unique. However, general solutions of \eqref{eq:MDEd} are not unique. Still, the new differential
structure defined over smooth solutions of \eqref{eq:MDEd} satisfies the most crucial properties we would expect of a differential structure on a geometric space, namely existence and uniqueness of flows. Furthermore, the structure is finer than the classical one defined by the continuity equation, such that some important standard diffusion processes become smooth, e.g, the standard Wiener process. 

%\backmatter

\bibliographystyle{plainnatnew}
\bibliography{MDEarxiv}

\end{document}